\documentclass[a4paper, 12pt, twoside]{article}

\usepackage{enumerate}
\usepackage{enumitem}

\usepackage[section]{placeins}	
\usepackage{epigraph}
\usepackage[font={small}]{caption}
\usepackage{appendix}

\usepackage{amsmath,amssymb,mathrsfs}
\usepackage{amsfonts}			
\usepackage{nicefrac}			

\usepackage{bbm} 

\usepackage{stmaryrd}
\usepackage{calc,ifthen,xspace}

\usepackage{graphicx}			

\usepackage{tikz}
\usepackage{tikz-cd}
\usetikzlibrary{knots}

\usepackage{xspace}
\usepackage{textcomp}
\usepackage{array}
\usepackage{hyphenat}

\usepackage{subfigure}
\usepackage[subfigure]{tocloft}
\usepackage{etoolbox} 
\patchcmd{\thebibliography}
  {\settowidth}
  {\setlength{\itemsep}{-2pt plus 0.1pt}\settowidth}
  {}{}
\apptocmd{\thebibliography}
  {\small}
  {}{}

\usepackage[pdftex,pdfborder={0 0 0},
			colorlinks=true,
			linkcolor=blue,
			citecolor=red,
			pagebackref=true,
			]{hyperref}	

\usepackage[top=2.5cm, bottom=2cm, left=3cm, right=2.5cm, headheight=15pt]{geometry}
\usepackage{cleveref}
\crefformat{section}{\S#2#1#3} 
\crefformat{subsection}{\S#2#1#3}
\crefformat{subsubsection}{\S#2#1#3}

\AtBeginDocument{}

\renewcommand{\leq}{\ensuremath{\leqslant}} 
\renewcommand{\geq}{\ensuremath{\geqslant}}

\newcommand{\Z}{\mathbb Z}

\DeclareMathOperator{\Der}{Der}

\DeclareMathOperator{\tr}{Tr}

\DeclareMathOperator{\ima}{Im}

\DeclareMathOperator{\Hom}{Hom}
\DeclareMathOperator{\Aut}{Aut}

\DeclareMathOperator{\Inn}{Inn}
\DeclareMathOperator{\inn}{inn}

\DeclareMathOperator{\gr}{gr}

\DeclareMathOperator{\qd}{\triangleright}
\DeclareMathOperator{\rqd}{\triangleleft}
\DeclareMathOperator{\Stab}{Stab}
\DeclareMathOperator{\Tr}{Tr}

\newcommand*{\qdn}[1]{\overset{\scalebox{0.4}{$#1$}}{\triangleright}}

\usepackage{amsthm}

\theoremstyle{plain}
\newtheorem{theo}{Theorem}
\newtheorem{fait}[theo]{Fact} 
\newtheorem{lem}[theo]{Lemma}
\newtheorem{prop}[theo]{Proposition}

\newtheorem{cor}[theo]{Corollary}
\newtheorem{fact}[theo]{Fact}

\numberwithin{equation}{theo}

\theoremstyle{definition}
\newtheorem{defi}[theo]{Definition}
\newtheorem{nota}[theo]{Notation}
\newtheorem{hypo}[theo]{Assumption}
\newtheorem{conv}[theo]{Convention}
\newtheorem{ex}[theo]{Example}
\newtheorem{rmq}[theo]{Remark}

\let\oldpagenumbering\pagenumbering
\renewcommand{\pagenumbering}[1]{%
	\cleardoublepage
	\oldpagenumbering{#1}
}

\author{Jacques \scshape{Darn\'e}}
\title{Nilpotent quandles}
\date{October 3, 2026}

\hypersetup{
    pdfauthor={Jacques Darn\'e},
    pdfsubject={Preprint},
    pdftitle={Nilpotent quandles},
    pdfkeywords={quandles}
}

\begin{document}

\maketitle

\begin{abstract}
A \emph{nilpotent} quandle is a quandle whose inner automorphism group is nilpotent. Such quandles have been called \emph{reductive} in the literature, but it turns out that their behaviour is in fact very close to nilpotency for groups. In particular, we show that it is easy to characterise generating sets of such quandles, and that they have the Hopf property. We also show how to construct free nilpotent quandles from free nilpotent groups. We then use the properties of nilpotent quandles to describe a simple presentation of their associated group, and we use this to recover the classification of abelian quandles by Lebed and Mortier \cite{Lebed-Mortier}. We also study \emph{reduced} quandles, and we show that the reduced fundamental quandle is equivalent, as an invariant of links, to the reduced peripheral system, sharpening a previous result of Hughes \cite{Hughes}. Finally, we give a characterisation of nilpotency in terms of the associated invariants of braids.
\end{abstract}

\numberwithin{theo}{section}

\section*{Introduction}

Quandles are algebraic structures whose relevance to knot theory has been discovered independently in 1982 by Joyce \cite{Joyce} and Matveev \cite{Matveev}. They introduced the fundamental quandle of a knot, and they both showed that it is a complete invariant of oriented knots, up to reversed mirror image. The fundamental quandle is in fact a complete invariant of links up to taking reversed mirror images of non-split components \cite[Cor.~5.3]{FR_racks_and_links}.

Quandles (and \emph{racks}, which are close relatives) have been the subject of a rich literature, studying them from many different points of view, with many different motivations, ranging from the knot-theoretic context \cite{Eisermann'}, or closely related contexts such as knotted surfaces \cite{CJKLS}
or set-theoretic solutions of the Yang-Baxter equation \cite{Lebed-Vendramin, Jedlicka}
to singularity theory ~\cite{Brieskorn} or classification of Hopf algebras~\cite{Andruskiewitsch-Grana}. Their cohomology has also been widely studied, algebraically \cite{Szymik} or via the construction of classifying spaces \cite{FR_rack_space}.
The algebraic study of quandles has also flourished \cite{Joyce2, Eisermann, EGM, Bardakov1, Bonatto}. Needless to say, this brief mention of the literature on racks and quandles cannot pretend to be exhaustive, but only to give some idea of the vastness of the topic. The interested reader can find an introduction to (part of) this vast literature in the survey \cite{Carter}.

\paragraph*{Nilpotent quandles.} The present paper is an attempt at building a theory of nilpotency for quandles which would closely resemble the theory of nilpotent groups. As it turns out, the right candidates for such a theory were already present in the literature: the quandles that we call \emph{$n$-nilpotent} are the ones called \emph{$n$-reductive} elsewhere. The latter terminology comes from universal algebra \cite{Pilitowska}: these quandles are characterised by some equations, and algebras satisfying them were called \emph{$n$-reductive} in this context. We feel, however, that these quandles deserve to be called nilpotent. Here are some reasons why:
\begin{itemize}[itemsep = -1pt, topsep = 4pt]
    \item Eisermann's quandle coverings \cite{Eisermann} are close analogues to central extensions of groups. In the same way as nilpotent groups are obtained as iterated central extensions of the trivial group, nilpotent quandles are obtained as iterated coverings of the trivial quandle (Th.~\ref{nilpotency_via_coverings}).
    \item For nilpotent groups, surjectivity of morphisms and generation by subsets can be tested at the level of abelianisations. A similar statement holds for nilpotent quandles (Prop.~\ref{surj_between_nilp_qdles} and Cor.~\ref{generators_of_nilp_qdles}).
    \item Finitely generated nilpotent groups are Hopfian; so are finitely generated nilpotent quandles (Th.~\ref{Hopf_prop}).
    \item Nilpotent quandles are exactly the ones whose enveloping group is nilpotent (Cor.~\ref{Q_nilp_ssi_G(Q)_nilp}).
\end{itemize}

Using the latter characterisation, we get some answers to \cite[Pb.~3.6]{Bardakov1}, where it was asked if one can classify quandles having a nilpotent enveloping group. In particular, we recover the classification of abelian quandles from~\cite{Lebed-Mortier}.

Our notion of nilpotency for quandles is different from the one obtained by Bonatto and Stanovsk\'y by specialization of a general notion of nilpotency in categories of universal algebras in \cite{Bonatto}. Our requirement is stronger than theirs, because of \cite[Lem.~6.2]{Bonatto}. We show below that it is strictly stronger: no non-trivial connected quandle is nilpotent in our sense (Lemma \ref{nilp_connected_are_*}).

\paragraph*{Main results.} Since we are exploring a notion that has not been studied for itself before (although appearing in many guises in different contexts), the paper contains many new results. Among these, let us first quote the ones that give different characterisations of $c$-nilpotent quandles, which we define as the ones whose inner automorphism group is $(c-1)$-nilpotent. Namely, a quandle $Q$ is $c$-nilpotent if and only if:
\begin{itemize}[itemsep = -1pt, topsep = 4pt]
    \item It is $c$-reductive (Cor.~\ref{nilp_is_reductive}).
    \item The map $Q \rightarrow *$ to the one-element quandle is a composition of $c$ covering maps.
    \item Colourings of a welded link diagram by elements of $Q$ are invariant under inserting a welded braid in $\Gamma_c(wP_n)$, for any number $n$ of strands (Th.~\ref{nilp_via_action_of_wP}).
\end{itemize}
As far as the author knows, only the first of the above equivalences was already known \cite[Th.~4.1]{Jedlicka}.

Let us now quote what we consider to be the four main results of our paper:
\begin{itemize}[itemsep = -1pt, topsep = 4pt]
    \item The free $c$-nilpotent quandle on $x_1, ..., x_n$ is the union of conjugacy classes of the $x_i$ in the free $c$-nilpotent group on $x_1, ..., x_n$ (Prop.~\ref{free_nilp_qdl}). A similar statement holds for \emph{free reduced quandles} (Prop.~\ref{RQ_injects_into_RF}).
    \item Finitely generated nilpotent quandles satisfy the Hopf property (Th.~\ref{Hopf_prop}).
    \item Enveloping groups of nilpotent quandles admit a nice presentation (Th.~\ref{presentation_of_G(Q)}).
    \item The reduced fundamental quandle is equivalent to a weak version of the reduced peripheral system of links (Cor.~\ref{RQ_as_inv}).
\end{itemize}
The first one gives a good understanding of free objects in the category of $c$-nilpotent quandles. The second one adapts a very important property of nilpotent groups to nilpotent quandles. The third one offers a good calculation tool for computing enveloping groups of our quandles. Finally, the last one, more topological in nature, is a sharper version of the main result of \cite{Hughes}.

\paragraph*{Outline of the paper.} After a quick reminder of the group-theoretic notions we use (\cref{reminders}), we recall the basics of quandle theory in~\cref{general_quandles}. Next, in~\cref{sec_nilp_qdls}, we define nilpotent quandles and we begin their study, giving the first two characterisations quoted above, and studying universal nilpotent quotients and free nilpotent quandles. The next section (\cref{sec_nilp_group_and_qdls}) is devoted to adapting some group-theoretic results to the world of nilpotent quandles. In particular, we show that nilpotent quandles are generated by representatives of their orbits (Cor.~\ref{generators_of_nilp_qdles}), before turning to the proof of the Hopf property for finitely generated nilpotent quandles. In~\cref{sec_construction}, we revisit Joyce's construction of quandles from group-theoretic data, a key tool used many times in the sequel. In particular, we re-interpret it as a categorical equivalence between suitable categories. We then use this construction to find our presentation of enveloping groups of nilpotent quandles, and to give a useful injectivity criterion. We then turn our attention to the particular case of $2$-nilpotent quandles (called \emph{abelian} in~\cite{Lebed-Mortier}), whose study is the subject of~\cref{section_2-nilp}. Finally, the last two sections are more closely related to the theory of knots and links. Namely, \cref{sec_reduced} is devoted to the study of \emph{reduced} quandles, which are the ones used to get link-homotopy invariants of links. The main result of this section is a precise comparison between the universal reduced quotient of the fundamental quandles and the classical \emph{reduced peripheral system} of links. The last section (\cref{sec_braids}) gives our characterisation of nilpotent quandles in terms of colourings. Let us finish this outline by mentioning the two appendices: the first one deals with what our notion of nilpotency becomes for racks, while in the second one, a trace argument is used to show the existence of non-tame automorphisms of free nilpotent quandles.

\paragraph{Acknowledgements} 
The author wishes to thank the anonymous referee for their careful reading of the paper, which led to substantial improvements (in particular in \cref{par_crossed_sets} and \cref{sec_reduced}). He also wishes to thank Lorenzo Traldi for the interest he showed in the first version of the paper, and for numerous helpful comments. No AI has been used in this research, except for some proofreading.

\setcounter{tocdepth}{2}
\tableofcontents

\section{Notations and reminders: group theory}\label{reminders}

\begin{nota}
Let $G$ be a group. If $x, y \in G$, we denote by $[x,y]$ their commutator $xyx^{-1}y^{-1}$, and we use the usual exponential notations $x^y = y^{-1}xy$ and ${}^y\! x = yxy^{-1}$ for conjugation in $G$. If $A, B \subseteq G$ are subsets of $G$, we denote by $[A,B]$
 the subgroup generated by commutators $[a,b]$ with $a \in A$ and $b \in B$.
 \end{nota}
 
 The \emph{lower central series} $\Gamma_*(G)$ of a group $G$ is defined by $\Gamma_1(G) := G$ and $\Gamma_{i+1}(G) := [G, \Gamma_i(G)] \subseteq \Gamma_i(G)$. The group $G$ is called \emph{nilpotent} (resp.~\emph{residually nilpotent}) if we have $\Gamma_{c+1}(G) = \{1\}$ for some integer $c \geq 1$ (resp.~$\bigcap \Gamma_i(G) =  \{1\}$). The minimal such $c$ is its \emph{nilpotency class}.
  
\smallskip 
Let us mention that the associated graded abelian group \[\gr(\Gamma_*(G)) := \bigoplus \Gamma_i(G)/\Gamma_{i+1}(G)\] bears the structure of a \emph{graded Lie ring}, whose bracket is induced by commutators in $G$. This Lie ring is easily seen to be \emph{generated by its degree $1$}, its degree-one part being $\Gamma_1(G)/\Gamma_2(G) = G^{ab}$ \cite[Prop.~1.9]{Darne1}.

\smallskip 
Let $F_n$ be the free group on $n$ generators $x_1, ..., x_n$. By a theorem of Magnus, $\gr(F_n) \cong \mathbb L_n$ is the free Lie ring on $n$ generators $X_i = \overline x_i$, seen as a graded Lie ring via the usual degree of Lie monomials. Let $F_{n,c} := F_n/\Gamma_{c+1}$ denote the free $c$-nilpotent group. Magnus' theorem implies that $\gr(F_{n,c}) \cong \mathbb L_{n,c}$ is the free $c$-nilpotent Lie ring, which is the quotient of $\mathbb L_n$ by its elements of degree at least $c+1$.

\section{General theory of quandles}\label{general_quandles}

This section is a brief introduction to quandles and to the general facts we need about them. This is mostly classical material, that can be found for instance in \cite{Joyce, Andruskiewitsch-Grana, Eisermann}. In order to help the reader, we have tried to make this introduction as self-contained as possible.

\begin{defi}\label{def_qdl}
A \emph{quandle} is a set $Q$ endowed with a binary operation $\qd$ satisfying the following axioms:
\begin{enumerate}[itemsep=-0.3em, topsep = 4pt]
\item $\forall x, y, z \in Q,\ x \qd (y \qd z) = (x \qd y) \qd (x \qd z),$
\item $\forall x \in Q,\ x \qd (-) : Q \rightarrow Q$ is a bijection,
\item $\forall x \in Q,\ x \qd x = x$.
\end{enumerate}
Quandles and their morphisms form a category, denoted by $\mathcal Qdl$.
\end{defi}

The first axiom means that the maps $x \qd(-)$ are endomorphisms of $(Q, \qd)$. The second axiom says that they are in fact automorphisms.

\begin{rmq}
We prefer this convention to the more usual one, where the law, usually denoted by $\triangleleft$, is distributive on the left side. In that, like for instance the authors of \cite{Andruskiewitsch-Grana} or \cite{Szymik}, we follow our habit of having groups acting on the left. The difference between the two conventions is only a matter of easy translation.
\end{rmq}

\begin{ex}[Trivial quandles]
Any set $X$ has a canonical quandle law, defined by $x \qd y = y$ for all $x,y \in X$. Such quandles are called \emph{trivial}. A morphism between two trivial quandles is just a map between the underlying sets.
\end{ex}

\begin{nota}
If $x$ is an element of the quandle $Q$, we denote by $x \qdn{-1} (-)$ the inverse of $x \qd (-)$. More generally, for any $n \in \Z$, we denote by $x \qdn{n} (-)$ the $n$-th iterate of $x \qd (-)$, for any $n \in \Z$.
\end{nota}

\begin{defi}\label{def_behaviour}
Let $x$ be an element of a quandle $Q$. The automorphism $x \qd(-)$ of $Q$ is called the \emph{behaviour} of $x$. Two elements having the same behaviour are said to \emph{behave in the same way}.
\end{defi}

\begin{defi}
Let $G$ be a group. The law $\qd$ defined by $g \qd h = {}^g\!h = ghg^{-1}$ turns $G$ into a quandle, called its \emph{conjugation quandle} $c(G)$.
\end{defi}

\begin{conv}
Conjugation is the default quandle structure on a group $G$: it is the one implied when no other structure is specified. For instance, we may speak of morphisms from a quandle $Q$ to a group $G$, meaning quandle morphisms $Q \rightarrow c(G)$.
\end{conv}

\begin{defi}
Let $Q$ be a quandle. Its \emph{enveloping group} $G(Q)$ is the group defined by the   following presentation:
\[G(Q) = \langle Q \mid \forall x, y \in Q,\ x \qd y = xyx^{-1} \rangle.\]
The assignment $q \mapsto 1$ (for $q \in Q$) induces a morphism $\epsilon_Q : G(Q) \rightarrow \Z$.
\end{defi}

\begin{ex}\label{G(triv)}
If $Q$ is a trivial quandle, then the relations defining $G(Q)$ are saying exactly that the generators commute with each other. Thus, in this case, $G(Q) = \Z[Q]$ is the free abelian group with basis $Q$. For instance, $G(*) = \Z$, which allows us to interpret the morphism $\epsilon_Q$ as $G(Q \rightarrow *)$.
\end{ex}

The canonical map from $Q$ to $G(Q)$ is in fact (by definition of $G(Q)$) a quandle morphism $\eta_Q: Q \rightarrow c(G(Q))$, which turns out to be the unit of an adjunction. Indeed, the following proposition is easy to check: 
\begin{prop}\label{adj_c_G}
The constructions $c(-)$ and $G(-)$ define functors between the category of quandles and the category $\mathcal Grp$ of groups, which are  adjoint to each other:
\[ G : \begin{tikzcd}\mathcal Qdl \ar[r, shift left] & \mathcal Grp \ar[l, shift left] \end{tikzcd} : c.\]
\end{prop}

Notice that $\eta_Q$ does not need to be injective in general. In fact, it is often not injective (see for instance Example~\ref{Q12} below). It thus makes sense to give a name to quandles injecting into their enveloping group:
\begin{defi}
A quandle $Q$ is called \emph{injective} when $\eta_Q: Q \rightarrow G(Q)$ is injective.
\end{defi}

Remark that a quandle is injective if and only if it injects into $c(G)$ for \emph{some} group~$G$.

\begin{rmq}\label{faithful_vs_inj}
Injectivity is weaker than \emph{faithfulness}, which asks for the action of $Q$ on itself to be faithful, that is, for the canonical map $Q \rightarrow \Inn(Q)$ to be injective (see Definitions~\ref{def_action} and~\ref{def_inner_aut} and Example~\ref{adjoint_action}). In our context, faithfulness will be irrelevant, since the only nilpotent quandles acting faithfully on themselves are $\varnothing$ and~$*$ (see Remark~\ref{nilp_faithful} below). On the contrary, injectivity plays an important role in our theory of nilpotent quandles.  
\end{rmq}


\begin{defi}\label{def_action}
An \emph{action} of a quandle $Q$ on another quandle $L$ is a morphism from $Q$ to $c(\Aut(L))$. Using the above adjunction, it can also be seen as an action of the group $G(Q)$ on $L$ by automorphisms (that is, a morphism from $G(Q)$ to $\Aut(L)$).
\end{defi}

\begin{ex}[Adjoint action]\label{adjoint_action}
Any quandle $Q$ acts on itself by $x \mapsto x \qd (-)$. Indeed, this defines a morphism from $Q$ to the conjugation quandle of $\Aut(Q)$, since:
\[(x \qd y) \qd z = (x \qd y) \qd (x \qd( x \overset{-1}{\qd}  z)) = x \qd (y \qd ( x \overset{-1}{\qd}  z)).\]
This action (or the corresponding action of $G(Q)$ on $Q$) is called the \emph{adjoint action}. In the sequel, the adjoint action of $g \in G(Q)$ on $Q$ is simply denoted by $q \mapsto g \cdot q$.
\end{ex}

\begin{lem}\label{equivariance_of_eta}
Let $Q$ be a quandle. For every $x,y \in Q$ and every $g \in G(Q)$, we have:
\[\eta(x) \cdot y = x \qd y \hspace{2em} \text{ and }\hspace{2em}  \eta(g \cdot x) = g\eta(x)g^{-1}.\]
\end{lem}

\begin{proof}
The first formula is just a reformulation of the definition of the adjoint action: the element $\eta(y)$ acts via $y \qd(-)$. The second one says that $\eta$ is $G(Q)$-equivariant, with respect to the adjoint action on $Q$ and the conjugation action of $G(Q)$ on itself. It is thus enough to check equivariance with respect to the generators $g = \eta(y)$ ($y \in Q$) of $G(Q)$. By construction of the enveloping group, we have $\eta(\eta(y) \cdot x) = \eta(y \qd x) = \eta(y)\eta(x)\eta(y)^{-1}$, and the conclusion follows.
\end{proof}

\begin{lem}\label{equivariance}
Let $f: Q \rightarrow R$ be a quandle morphism. Then $f$ is equivariant with respect to adjoint actions. Precisely, if $f_*: G(Q) \rightarrow G(R)$ is the morphism induced by~$f$, then:
\[\forall g \in G(Q),\ \forall x \in Q,\  f(g \cdot x) = f_*(g) \cdot f(x). \]
\end{lem}

\begin{proof}
Recall that we denote by $\eta : Q \rightarrow G(Q)$ the canonical map. It is enough to check equivariance with respect to the generators $g = \eta(y)$ ($y \in Q$) of $G(Q)$. By definition of the adjoint action,
$f(\eta(y) \cdot x) = f(y \qd x) = f(y) \qd f(x) = f_*(\eta(y)) \cdot f(x),$ whence our claim.
\end{proof}

\begin{rmq}
Lemma~\ref{equivariance_of_eta} can also be seen as a particular case of Lemma~\ref{equivariance}, applied to $\eta : Q \twoheadrightarrow \eta(Q)$, by noticing that $G(Q) \cong G(\eta(Q))$ canonically, and making the proper identifications.
\end{rmq}

\begin{defi}\label{def_inner_aut}
The \emph{group of inner automorphisms} of a quandle $Q$, denoted by $\Inn(Q)$, is the subgroup of $\Aut(Q)$ generated by the automorphisms $x \qd(-)$ (for $x \in Q$). Equivalently, it is the image of $G(Q)$ in $\Aut(Q)$ by the adjoint action.
\end{defi}

The adjoint action (or, equivalently, the action by inner automorphisms) induces a decomposition of a quandle $Q$ into \emph{orbits}. The corresponding equivalence relation is easily seen to be compatible with the quandle law, allowing us to define the \emph{quandle of orbits} $\pi_0(Q) = Q/G(Q)$. It is a trivial quandle, and any morphism from $Q$ to a trivial quandle factors uniquely through the projection $Q \twoheadrightarrow \pi_0(Q)$. For this reason, we call it the \emph{trivialisation} of $Q$.

\begin{rmq}
The orbits of $Q$ are often called \emph{connected components}, and a quandle with one connected component is then called \emph{connected}. This terminology is motivated by the study of the fundamental quandle of a link: orbits of the fundamental quandle do correspond to the components of the link. However, since connected components of a quandle need not themselves be connected, we feel that this terminology is somewhat misleading, and must be used with some care. We will mostly avoid it, but we use the notation $\pi_0$. In fact, the reader should keep in mind that taking the trivialization of a quandle is very similar to taking the abelianisation of a group, and bears little resemblance to taking the group of connected components of a topological group. In this regard, see Cor.~\ref{G(Q)_ab}.
\end{rmq}

The decomposition into orbits is a particular case of the following construction:
\begin{lem}\label{quotient_by_group}
Let $Q$ be a quandle and $G$ be a group acting on $Q$ by automorphisms such that the image of the associated morphism $G \rightarrow \Aut(Q)$ is normalized by $\Inn(Q)$. Then the decomposition of $Q$ into $G$-orbits is compatible with the quandle law. The quandle of $G$-orbits will be called the \emph{quotient of $Q$ by $G$}, denoted by $Q/G$.
\end{lem}

\begin{proof}
Up to replacing $G$ by its image in $\Aut(Q)$, we can assume that $G$ is a subgroup of $\Aut(Q)$ normalized by $\Inn(Q)$. Let $x, y \in Q$, and $g,h \in G$. We need to prove that $(gx) \qd (hy)$ belongs to the $G$-orbit of $x \qd y$. Since $(gx) \qd (hy) = g (x \qd (g^{-1}hy))$, it is enough to prove that $x \qd Gy \subseteq G (x \qd y)$. By taking $z = x \qd y$, this is equivalent to showing that $x \qd G (x \qdn{-1} z) \subseteq G z$. But $x \qd(-)$ is an element of $\Inn(Q)$, and $G$ is normalized by $\Inn(Q)$, whence the result, noticing that the same is true if we replace $\qd$ by $\qdn{-1}$.
\end{proof}

\begin{rmq}\label{congruence_verif}
Compatibility of an equivalence relation with the quandle law is understood as compatibility with $\qd$ and with $\qdn{-1}$ (which are both needed for defining $\mathcal Qdl$ as a variety, in the sense of universal algebra). Notice that an equivalence relation on a quandle may well be compatible with $\qd$ and not with $\qdn{-1}$, the simplest example being the relation defined by $n \sim m$ if $m,n \geq 0$ on the quandle $Q_{0,1} = \Z \sqcup \{t\}$, where elements of $\Z$ act trivially, and $t$ acts by $t \qd n = n+1$ on $\Z$. But then there is no compatible quandle law on the quotient, which is why we stress that both compatibilities need to be checked in order to get well-behaved quotients.
\end{rmq}

\begin{nota}\label{nota_congruence}
Under the hypotheses of Lemma~\ref{quotient_by_group}, we will often write $x \equiv y\ [G]$ to say that two elements $x$ and $y$ of $Q$ are in the same $G$-orbit.
\end{nota}

Of particular importance to us are the quotients of $Q$ by normal subgroups of $G(Q)$ (acting on $Q$ via the adjoint action), or, equivalently, by normal subgroups of $\Inn(Q)$. Notice that not every quotient of a quandle $Q$ is of this form (nor, in fact, a quotient by a group of automorphisms). However, given a quotient of $Q$, there is a best approximation by such a quotient.

\begin{prop}\label{quotients_by_sbgp_of_G(Q)}
Let  $p: Q \twoheadrightarrow R$ be a surjective quandle morphism, and let $H$ be a subgroup of $G(Q)$. Then $p$ factors through $Q \twoheadrightarrow Q/H$ if and only if the subgroup $p_*(H)$ of $G(R)$ acts trivially on $R$ via the adjoint action. 
\end{prop}

\begin{proof}
Let $H$ be a subgroup of $G(Q)$. Then $p$ factors through $Q \twoheadrightarrow Q/H$ (as a set map, which must be a quandle morphism) if and only if for every $h \in H$ and every $x \in Q$, we have $p(h \cdot x) = p(x)$. But Lemma~\ref{equivariance} says that $p(h \cdot x) = p_*(h) \cdot p(x)$. Since $p$ is surjective, we see that $p$ factors through $Q \twoheadrightarrow Q/H$ if and only if $p_*(H)$ acts trivially on $R$.
\end{proof}

The hypotheses of Proposition~\ref{quotients_by_sbgp_of_G(Q)} are satisfied for $H = \ker(p_*)$. The corresponding intermediate quotient will play an important role in \cref{par_R(Q)_in_detail}, so we state:
\begin{cor}\label{quotients_by_ker_G(p*)}
Let $K$ be the kernel of the composition of $p_* : G(Q) \rightarrow G(R)$ with the canonical morphism $G(R) \twoheadrightarrow \Inn(R)$. Then $p$ factors through $Q \twoheadrightarrow Q/K$, and if $H$ is a subgroup of $G(Q)$, then $p$ factors through $Q \twoheadrightarrow Q/H$ if and only if $H \subseteq K$.
\end{cor}

\begin{proof}
Notice that $p_*(H)$ acts trivially on $R$ if and only if $H \subseteq K$.
\end{proof}

The next two lemmas allow us to compare $G(Q)$ with $\Inn(Q)$.

\begin{lem}\textup{\cite[Prop.~2.37]{Eisermann}}\label{central_ext}
For any quandle $Q$, the projection $G(Q) \twoheadrightarrow \Inn(Q)$ is a central extension.
\end{lem}

\begin{proof}
The morphism $G(Q) \twoheadrightarrow \Inn(Q)$ can be composed with $\varphi \mapsto \varphi_*$ from $\Aut(Q)$ to $\Aut(G(Q))$, to get a morphism $c: G(Q) \rightarrow \Aut(G(Q))$. It is easy to see (by checking this on the generators of $G(Q)$ corresponding to elements of $Q$) that this morphism $c$ is the usual morphism sending any element $g \in G(Q)$ to the associated inner automorphism. Since $\ker(c) = \mathcal Z(G(Q))$, the conclusion follows.
\end{proof}

The kernel of $G(Q) \twoheadrightarrow \Inn(Q)$ (which consists of elements $g \in G(Q)$ acting trivially on $Q$) is not equal to the whole center $\mathcal Z(G(Q))$ in general, but there is one important case when it is:

\begin{lem}\label{Inn(Q)_and_G(Q)/Z}
If a quandle $Q$ is injective, then $\Inn(Q) \cong G(Q)/\mathcal Z(G(Q))$.
\end{lem}

\begin{proof}
If $Q$ injects into $G(Q)$, then the action of $G(Q)$ on $Q$ identifies with the action of $G(Q)$ by conjugation on its subset $Q$. Since $Q$ generates $G(Q)$, an element $g \in G(Q)$ acts trivially on $Q$ if and only if it acts trivially on $G(Q)$, and the latter means that $g \in \mathcal Z(G(Q))$.
\end{proof}

\begin{ex}\label{Q12}
There are numerous examples of quandles with an abelian enveloping group which are not trivial. If $Q$ is such a quandle, then $\mathcal Z(G(Q)) = G(Q)$, but $\Inn(Q) \ncong G(Q)/\mathcal Z(G(Q)) = 1$ (else $Q$ would be trivial). Such examples will be characterised in \cref{Applications_classification}. The simplest one is $Q_{1,2}$, which consists of $3$ elements $a,b,c$, with $a \qd(-)$ exchanging $b$ and $c$, and $b \qd(-) = c \qd(-) = id$ (this is the one from~\cite[\S 6]{Joyce}). Then $G(Q_{1,2})$ is easily seen to be free abelian on $\eta(a)$ and $\eta(b) = \eta(c)$. 
\end{ex}

\subsection*{Inner automorphisms and surjections}

The construction $\Inn(-)$ is not functorial, at least not in any way that would turn the canonical maps $c_Q: Q \rightarrow \Inn(Q)$ into a natural transformation. Indeed, if $q$ is any element of any quandle $Q$, naturality with respect to the inclusion $\{q\} \hookrightarrow Q$ would force $c_Q(q)$ to be trivial, which is impossible. However, this construction is functorial when restricted to the subcategory of surjective morphisms:
\begin{prop}\textup{\cite[Prop.~2.26]{Eisermann}}\label{Inn(surj)}
Let $p: Q \rightarrow R$ be a surjective morphism. Then $p$ induces a unique morphism $p_*: \Inn(Q) \rightarrow \Inn(R)$ such that the following diagram commutes:
\[\begin{tikzcd}
Q \ar[r, "c_Q"] \ar[d, two heads, "p"] &\Inn(Q) \ar[d, dashed, two heads, "p_*"] \\
R \ar[r, "c_Q"] &\Inn(R).
\end{tikzcd}\]
This construction defines a functor from the category of surjections between quandles  to the category of surjections between groups.
\end{prop}

\begin{rmq}
We use the notation $p_*$ for both the induced morphisms $G(Q) \rightarrow G(R)$ and $\Inn(Q) \rightarrow \Inn(R)$. Since the latter is induced by the former (see the proof of the Proposition), this should not be a source of confusion. In particular, Lemma~\ref{equivariance} works with both of them as $p_*$. Whenever we  really need to distinguish between the two, we will call the first one $G(p)$, as for instance in the following proof.
\end{rmq}

\begin{proof}[Proof of Proposition~\ref{Inn(surj)}]
Uniqueness is obvious, since $c_Q(Q)$ generates $\Inn(Q)$. To see that $p_*$ exists, we use that $\Inn(Q)$ is a quotient of $G(Q)$, which is a functorial construction, so that the morphism we seek must fit into the following commutative diagram:
\[\begin{tikzcd}
Q \ar[r] \ar[d, two heads, "p"] 
&G(Q)\ar[r, two heads, "\pi"] \ar[d, two heads,  "G(p)"]
&\Inn(Q) \ar[d, dashed, two heads, "p_*"] \\
R \ar[r] 
&G(R)\ar[r, two heads, "\pi"]
&\Inn(R).
\end{tikzcd}\]
Thus we only need to show that $\pi \circ G(p)$ vanishes on $\ker(\pi)$. Let $\gamma \in \ker(\pi)$, that is, let $\gamma \in G(Q)$ such that $\pi(\gamma) = id_{Q}$. 
The latter means that $\gamma$ acts trivially on $Q$. Let $z \in R$. By surjectivity of $p$, there exists $t \in Q$ such that $z = p(t)$. 
Then (using Lemma~\ref{equivariance}) $G(p)(\gamma) \cdot z = G(p)(\gamma) \cdot p(t) = p(\gamma \cdot t) = p(t) = z$,
which means that the action of $G(p)(\gamma)$ on $R$ is trivial: $\pi \circ G(p)(\gamma) = id_{R}$, as announced.
\end{proof}

\subsection*{Quandle coverings}

The theory of quandle coverings was developed by Eisermann in \cite{Eisermann}, where a Galois correspondence was introduced for classifying them. This Galois correspondence, inspired by the Galois theory of covering spaces, and similar in some ways to Kervaire's theory of coverings of perfect groups, is in fact a particular case of Janelidze's categorical Galois theory \cite{Even}. Using terminology from Definition~\ref{def_behaviour}, coverings are defined as follows:

\begin{defi}\label{def_covering}
A \emph{quandle covering} is a surjective quandle morphism $p : Q \twoheadrightarrow R$ such that all the elements of each fiber $p^{-1}(x)$ behave in the same way.
\end{defi}

In other words, a surjective morphism $p : Q \twoheadrightarrow R$ is a covering if and only if the adjoint action $Q \rightarrow \Inn(Q)$ factors through $p$. Using the universal property of the enveloping group, one sees that it is equivalent to the existence of a group morphism making the following diagram commute:
\begin{equation}\label{covering_diag}
\begin{tikzcd}
Q \ar[r] \ar[d, two heads, "p"] 
&G(Q)\ar[r, two heads, "\pi"] \ar[d, two heads,  "p_*"]
&\Inn(Q) \ar[d, two heads,  "p_*"]\\
R \ar[r] 
&G(R)\ar[ru, dashed, "\exists q"] \ar[r, two heads, "\pi"] 
&\Inn(R).
\end{tikzcd}
\end{equation}

The following result will be used later for interpreting nilpotency in terms of coverings:
\begin{prop}\textup{\cite[Prop.~2.49]{Eisermann}}\label{central_ext_cov}
Let $p : Q \twoheadrightarrow R$ be a quandle covering. Then both $p_* : G(Q) \twoheadrightarrow G(R)$ and $p_* : \Inn(Q) \twoheadrightarrow \Inn(R)$ are central extensions.
\end{prop}

\begin{proof}
In the diagram \eqref{covering_diag}, the $\pi$ are central extensions (Lem.~\ref{central_ext}). Firstly, $\pi = q p_*$ implies that $\ker(p_*) \subseteq \ker(\pi)$ is central, whence our first statement. Secondly, for all $\varphi \in \Inn(Q)$, there is some $g \in G(R)$ such that $q(g) = \varphi$. If $p_*(\varphi) = id$, then $\pi(g) = id$, hence $g$ is central in $G(R)$, and $\varphi$ must be central in $\Inn(Q)$, since $q$ is surjective. This proves our second claim. 
\end{proof}

\section{Nilpotent quandles}\label{sec_nilp_qdls}

We now introduce our notion of nilpotent quandles, and we begin their study. In particular, we recover the precise link between our definition and the classical notion of \emph{reductive} quandles (Cor.~\ref{nilp_is_reductive}), we characterise nilpotent quandles in terms of coverings (Theorem~\ref{nilpotency_via_coverings}), and we study universal nilpotent quandles and free nilpotent ones, which we describe in group-theoretic terms (see Prop.~\ref{reflection_functors} and~\ref{free_nilp_qdl}).

\subsection{Definition}

\begin{defi}\label{def_nilp_qdl}
A quandle $Q$ is called \emph{nilpotent of class at most $c$}, or \emph{$c$-nilpotent}, if its group of inner automorphisms is $(c-1)$-nilpotent.
\end{defi}

\begin{ex}\label{1-nilp=triv}
A quandle $Q$ is $1$-nilpotent when $\Inn(Q)$ is trivial, that is, when $q \qd (-)$ is the identity for all $q \in Q$. Thus a quandle is $1$-nilpotent if and only if it is trivial.
\end{ex}

\begin{ex}\label{2-nilp=ab}
A quandle $Q$ is $2$-nilpotent when $\Inn(Q)$ is abelian, that is, when $x \qd (y \qd (-)) = y \qd (x \qd (-)) $ for all $x,y \in Q$. We call such quandles \emph{abelian}, following \cite{Lebed-Mortier}, whose classification we recover in \cref{Applications_classification}.
\end{ex}

Since $\Inn(Q)$ is not functorial in $Q$, it helps to have a characterization of nilpotency in terms of $G(Q)$:
\begin{lem}\label{c-nilp_in_terms_of_G}
A quandle $Q$ is $c$-nilpotent if and only if $\Gamma_c(G(Q))$ acts trivially on $Q$ via the adjoint action.
\end{lem}

\begin{proof}
The morphism associated to the adjoint action is a surjection $\pi$ from $G(Q)$ onto $\Inn(Q)$. As a consequence, $\pi\left(\Gamma_c(G(Q))\right) = \Gamma_c(\Inn(Q))$, whence the result.
\end{proof}

\begin{cor}\label{Q_nilp_ssi_G(Q)_nilp}
A quandle $Q$ is nilpotent if and only if $G(Q)$ is. If $Q$ is nilpotent of class $c$, then $G(Q)$ is nilpotent of class $c$ or $c-1$. Moreover, if $Q$ is injective, then $Q$ and $G(Q)$ have the same nilpotency class (whenever they are nilpotent).
\end{cor}

\begin{proof}
Combine Lemmas \ref{c-nilp_in_terms_of_G} and \ref{central_ext} to get the first two statements. Lemma \ref{Inn(Q)_and_G(Q)/Z} implies the last one.
\end{proof}

The link between the nilpotency class of $Q$ and that of $G(Q)$, though strong, is not an obvious one. For instance, there are non-trivial quandles with abelian enveloping groups (see Example~\ref{Q12}, and \cref{par_2-orbits_2-nilp}). At this point, the reader may wonder why our notion of nilpotency class should be better than the one using the nilpotency class of $G(Q)$ (that is, calling a quandle $Q$ $c$-nilpotent if $G(Q)$ is $c$-nilpotent). The reason is in fact very simple: nilpotency in a group is described by certain identities, and identities in $\Inn(Q)$ translate (by evaluating them at every element of $Q$) into equivalent identities \emph{in the quandle $Q$}. Whereas identities in $G(Q)$ cannot be translated into identities in $Q$, since $Q$ does not in general embed into $G(Q)$. This explains why our notion of nilpotency defines Birkhoff subcategories of the category of quandles, whereas the other notion (using $G(Q)$) does not. In particular, we have the following:

\begin{lem}\label{quotients_and_subquandles_of_nilp}
If $Q$ is a $c$-nilpotent quandle, then any quotient and any subquandle of $Q$ is $c$-nilpotent.
\end{lem}

This Lemma follows directly from the above remarks, but we give a direct proof in terms of our definition.

\begin{proof}[Proof of Lemma~\ref{quotients_and_subquandles_of_nilp}]
If $Q$ surjects onto a quandle $R$, then $\Inn(Q)$ surjects onto $\Inn(R)$ (Prop.~\ref{Inn(surj)}), hence $R$ is $c$-nilpotent when $Q$ is. For subquandles, we can use Lemma~\ref{c-nilp_in_terms_of_G}. Namely, if $R$ is a subquandle of $Q$, then the action of $G(R)$ on $R$ is obtained by restricting the action of $G(Q)$ on $Q$ through the morphism $G(R) \rightarrow G(Q)$ induced by the inclusion. This morphism sends $\Gamma_c(G(R))$ into $\Gamma_c(G(Q))$, which acts trivially on $Q$, hence also on $R$.
\end{proof}

\begin{ex}\label{ex_G(subqdl)_not_ab}
On the contrary, if $Q$ is a quandle having an abelian enveloping group, there can be subquandles of $Q$ whose enveloping group is not abelian. For instance, let us consider the quandle $Q_{n,n}$ from \cref{par_2-orbits_2-nilp}  (with $n \neq 1$, possibly $n = 0$). It is the disjoint union $(\Z/n) \sqcup (\Z/n)$ of two orbits, where any element acts trivially on its orbit and acts via the permutation $k \mapsto k+1$ on the other. Let us embed $Q_{n,n}$ into the quandle $Q$, defined as the disjoint union $Q_{n,n} \cup \{z\}$, where the elements of $Q_{n,n}$ act trivially on $z$, and $z$ acts via $k \mapsto k+1$ on both orbits of $Q_{n,n}$. One can easily check that $Q$ is a quandle with three orbits, generated by representatives $x$, $y$ and $z$ of its orbits (this can also be deduced from our classification results -- see Example~\ref{ex_G(subqdl)_not_ab_bis}). For instance, one can take $x$ (resp.~$y$) to be $0$ in the first (resp.~the second) copy of $\Z/n$. Its enveloping group is generated by $\eta(x)$, $\eta(y)$ and $\eta(z)$. Since $x \qd z = y \qd z = z$, both $\eta(x)$ and $\eta(y)$ commute with $\eta(z)$. Moreover, $x$ and $z$ both act via the same permutation on the orbit of $y$, so $x \qd y = z \qd y$, which implies that $\eta(x) \eta(y) \eta(x)^{-1} = \eta(z) \eta(y) \eta(z)^{-1} = \eta(y)$. This means that $\eta(x)$ and $\eta(y)$ also commute with each other, hence $G(Q)$ is abelian (necessarily isomorphic to $\Z^3$). However, $G(Q_{n,n})$ is not (\cref{par_2-orbits_2-nilp}).
\end{ex}

\subsubsection*{Characterisation by identities}

The following result displays two explicit sets of generating identities for $c$-nilpotency of a quandle. The first one will be used in the proof of Theorem~\ref{nilp_via_action_of_wP} below. The second one can also be found in~\cite[Th.~4.1]{Jedlicka}.
\begin{prop}\label{generating_identities}
For any quandle $Q$, the following families of elements, both indexed by $(q_1, ..., q_c) \in Q^c$, are two systems of generators of the group $\Gamma_c(G(Q))$:
\begin{itemize}[itemsep=-0.3em, topsep = 4pt]
\item The $[\eta(q_1), [\eta(q_2), ..., [\eta(q_{c-1}), \eta(q_c)] \cdots ]]$,
\item The $\eta \left( \left( \cdots ((q_1 \qd q_2) \qd q_3) \qd \cdots \right) \qd q_c \right) \cdot \eta \left( \left( \cdots (q_2 \qd q_3) \qd \cdots \right) \qd q_c \right)^{-1}$,
\end{itemize}
where, by convention, both the above expressions denote only $\eta(q_1)$ for $c = 1$.
\end{prop}

\begin{proof}
Recall that if $g_1, ..., g_n$ are elements of a group $G$, we denote by $[g_1, ..., g_n]$ their \emph{iterated commutator} $[g_1, [g_2, ..., [g_{n-1}, g_n] \cdots ]]$ (which is only $g_1$ when $n=1$, by convention). We also introduce the notation $w_n(q_1, ..., q_n) = \left( \cdots ((q_1 \qd q_2) \qd q_3) \qd \cdots \right) \qd q_n$, for elements $q_1, ..., q_n$ of a quandle (which is also only $q_1$ when $n=1$, by convention). 

\smallskip

We show both statements by induction on $c \geq 1$. They are clearly true for $c = 1$. Let $c \geq 2$, and let us suppose that they hold for $c-1$. Let us first recall that $g \eta(q) g^{-1} = \eta(g \cdot q)$ (for all $g \in G(Q)$ and all $q \in Q$ -- see Lemma~\ref{equivariance_of_eta}) and that $G(Q)$ acts on $Q$ by automorphisms. From this, we deduce that both our candidates for generating sets of $\Gamma_c(G(Q))$ are stable under conjugation by elements of $G(Q)$. As a consequence, it is enough to show that $\Gamma_c(G(Q))$ is normally generated by each family. Equivalently, we need to show that the above elements belong to $\Gamma_c$, and that $\Gamma_c = 1$ in the quotient of $G(Q)$ by the relations $[\eta(q_1), ..., \eta(q_c)] = 1$ (resp. $\eta(w_c(q_1, ..., q_c)) = \eta(w_{c-1}(q_2, ..., q_c))$), for $(q_i)_i \in Q^c$. 

\smallskip

 Elements of the first family clearly belong to $\Gamma_c(G(Q))$. Moreover, the relations $[\eta(q_1), ..., \eta(q_c)] = 1$ say that generators of the group (the $\eta(q_1)$) must commute with the elements $[\eta(q_2), ..., \eta(q_c)]$ which are, by the induction hypothesis, generators of $\Gamma_{c-1}$. This means that modulo these relations, $\Gamma_{c-1}$ is central, hence $\Gamma_c = 1$, proving our first claim.

\smallskip

Now, notice that since $\eta$ is a quandle morphism, for every $q_1, ..., q_n \in Q$, we have $\eta(w_n(q_1, ..., q_n)) = w_n(\eta(q_1), ..., \eta(q_n))$ (where the quandle law on the right is conjugation in $G(Q)$). Thus, elements of the second family can be written:
\begin{equation}\label{eq_relation_reductive}
\left(\eta(w_{c-1}(q_1, ..., q_{c-1})) \qd \eta(q_c) \right) \left(\eta(w_{c-2}(q_2, ..., q_{c-1})) \qd \eta(q_c) \right)^{-1}.
\end{equation}
By induction, $\eta(w_{c-1}(q_1, ..., q_{c-1})) \cdot \eta(w_{c-2}(q_2, ..., q_{c-1}))^{-1}$ is an element of $\Gamma_{c-1}$. As a consequence, using the formula $(u \qd x) (v \qd x)^{-1} = v[v^{-1}u, x] v^{-1}$, we deduce that the expressions \eqref{eq_relation_reductive} define elements of $\Gamma_c$, as expected. Moreover, killing these elements is equivalent to imposing the relations:
\[\eta(w_{c-1}(q_1, ..., q_{c-1})) \qd \eta(q_c) = \eta(w_{c-2}(q_2, ..., q_{c-1})) \qd \eta(q_c),\]
which mean that generators of the group (the $\eta(q_c)$) must commute with the elements $\eta(w_{c-1}(q_1, ..., q_{c-1})) \cdot \eta(w_{c-2}(q_2, ..., q_{c-1}))^{-1}$. The latter are, by the induction hypothesis, generators of $\Gamma_{c-1}$. This means that modulo these relations, $\Gamma_{c-1}$ is central, hence $\Gamma_c = 1$, proving our second claim.
\end{proof}

\begin{rmq}
The fact that $\eta(Q)$ is stable by conjugation is crucial in the proof that the $[s_1, ..., s_c]$ (for $s_1, ..., s_c \in \eta(Q)$) generate $\Gamma_c(G(Q))$. Indeed, if $S$ is a set of generators of a group which is not stable under conjugation, it is not even clear that the commutators of the form $[s_1, ..., s_c]$ (for $s_1, ..., s_c \in S$) should normally generate $\Gamma_c(G)$ in $G$, let alone generate it. 
\end{rmq}

In particular, we recover the following classical consequence, which is the case of quandles in~\cite[Th.~4.1]{Jedlicka}:
\begin{cor}\label{nilp_is_reductive}
A quandle $Q$ is $c$-nilpotent if and only if it is $c$-reductive, that is, if and only if it satisfies the following identities:
\[\forall q_1, ..., q_{c+1} \in Q,\ \left( \cdots ((q_1 \qd q_2) \qd q_3) \qd \cdots \right) \qd q_{c+1} = \left( \cdots (q_2 \qd q_3) \qd \cdots \right) \qd q_{c+1}.\]
\end{cor}

\begin{proof}
From the formula:
\[\eta \left( \left( \cdots ((q_1 \qd q_2) \qd q_3) \qd \cdots \right) \qd q_c \right) \cdot q_{c+1} = \left( \cdots ((q_1 \qd q_2) \qd q_3) \qd \cdots \right) \qd q_{c+1},\]
we deduce that the second family of Proposition~\ref{generating_identities} acts trivially on $Q$ if and only if $Q$ is $c$-reductive. Then the conclusion follows from Proposition~\ref{generating_identities} and Lemma~\ref{c-nilp_in_terms_of_G}.
\end{proof}

\subsubsection*{Characterisation in terms of coverings}

The next Lemma will enable us to give a characterisation of nilpotency in terms of Eisermann's quandle coverings (whose definition was recalled as Definition~\ref{def_covering}).
\begin{lem}\label{Orbits_under_Z(I)}
In any quandle $Q$, two elements in the same orbit under $\mathcal Z (\Inn(Q))$ behave in the same way (in the sense of Definition~\ref{def_behaviour}).
\end{lem}
\begin{proof}
Let $\gamma \in \mathcal Z (\Inn(Q))$, and $x \in Q$. Then, since $\gamma$ is an automorphism of $Q$ which commutes with $x \qd (-)$:
\[ \forall y \in Q,\ \ \gamma(x) \qd y = \gamma\left(x \qd \gamma^{-1}(y)\right) = x \qd y.\]
This means exactly that $x$ and $\gamma(x)$ have the same behaviour. 
\end{proof}

\begin{rmq}\label{nilp_faithful}
This implies that a nilpotent quandle that is faithful has at most one element (as announced in Remark~\ref{faithful_vs_inj}). Indeed, on the one hand, the conclusion is clear if $\Inn(Q)$ is trivial. On the other hand, if $\Inn(Q)$ is not trivial and nilpotent, then its center is not trivial, so it contains a non-trivial central element $\gamma$. Since $\gamma \neq id$, there exists some $x \in Q$ such that $\gamma(x) \neq x$. But Lemma~\ref{Orbits_under_Z(I)} says that $\gamma(x) \qd (-) = x \qd(-)$. In other words, $\gamma(x)$ and $x$ are two distinct elements of $Q$ having the same image in $\Inn(Q)$, so $Q$ is not faithful.
\end{rmq}

\begin{theo}\label{nilpotency_via_coverings}
A quandle $Q$ is nilpotent of class less than $c$ if and only if there exist a chain of quandle coverings:
\[\begin{tikzcd}
Q \ar[r, two heads, "p_c"] &\bullet \ar[r, two heads, "p_{c-1}"] &\bullet \ar[r, two heads] &\cdots \ar[r, two heads, "p_1"] &*.
\end{tikzcd}\]
\end{theo}

\begin{proof}
As a consequence of Lemma \ref{Orbits_under_Z(I)}, we see that for any quandle $Q$ and for any subgroup $G$ of $\mathcal Z (\Inn(Q))$, the quotient map $Q \twoheadrightarrow Q/G$ (defined using Lemma \ref{quotient_by_group}) is a covering. In particular, when $Q$ is $(c+1)$-nilpotent, the projection $Q \twoheadrightarrow Q/\Gamma_c$ is a covering onto a $c$-nilpotent quandle (see the first paragraph of the proof of Proposition~\ref{reflection_functors} below), which proves one implication. Conversely, if $Q \twoheadrightarrow R$ is a covering onto a $c$-nilpotent quandle, then it induces a projection $\Inn(Q) \twoheadrightarrow \Inn(R)$, which is a central extension by Proposition~\ref{central_ext_cov}, thus $Q$ is $(c+1)$-nilpotent.
\end{proof}

\begin{ex}
A quandle $Q$ is $2$-nilpotent if and only if it is a covering of a trivial quandle. In fact, it is $2$-nilpotent if and only if the canonical projection $Q \twoheadrightarrow \pi_0(Q)$ is a covering (see \cref{section_2-nilp}).
\end{ex}

\subsection{Universal nilpotent quotients}\label{par_univ_nilp_quotients}

Let $Q$ be a quandle. We can turn it into a $c$-nilpotent quandle by taking its quotient by $\Gamma_c(\Inn Q)$ (see Lemma \ref{quotient_by_group}), which we denote, for short, by $Q/\Gamma_c$. We now show that this is in fact the universal $c$-nilpotent quotient of $Q$, in the same way as $G/\Gamma_{c+1}G$ is the universal $c$-nilpotent quotient of a group $G$.
\begin{prop}\label{reflection_functors}
The full subcategory $\mathcal Qdl_c \subset \mathcal Qdl$ of $c$-nilpotent quandles is a reflexive subcategory, and the left adjoint to its inclusion into $\mathcal Qdl$ is the functor $(-)/\Gamma_c$ defined above.
\end{prop}

\begin{proof}
Let $Q$ be any quandle. The quotient map $p: Q \twoheadrightarrow Q/\Gamma_c$ induces a surjection from $G(Q)$ onto $G(Q/\Gamma_c)$, and thus from $\Gamma_c\left(G(Q)\right)$ onto $\Gamma_c\left(G(Q/\Gamma_c)\right)$. Equivariance of $p$ with respect to the adjoint actions (Lemma \ref{equivariance}), applied to $g \in \Gamma_c\left(G(Q)\right)$ and $x\in Q$, gives:
\[p_*(g)\cdot p(x) = p(g \cdot x) = p(x),\]
by definition of $p$. As $x$ and $g$ vary, $p(x)$ and $p_*(g)$ describe $Q/\Gamma_c$ and $\Gamma_c\left(G(Q/\Gamma_c)\right)$ respectively, so we deduce that the action of the latter on the former is trivial, whence the $c$-nilpotency of $Q/\Gamma_c$  by Lemma \ref{c-nilp_in_terms_of_G}.

Now, let  $f: Q \rightarrow R$ be a quandle morphism whose target is $c$-nilpotent. The subgroup $f_*(\Gamma_c(G(Q))) \subseteq \Gamma_c(G(R))$ acts trivially on $R$. As a consequence, by Proposition~\ref{quotients_by_sbgp_of_G(Q)}, $f$ factorizes uniquely through $Q/\Gamma_c$.
\end{proof}

We remarked above that the enveloping group functor $G(-)$ sends $c$-nilpotent quandles to $c$-nilpotent groups, because of Lemma \ref{central_ext}. The converse is true for the functor $c(-)$:

\begin{lem}\label{c(G)_nilp}
A group $G$ is $c$-nilpotent if and only if its conjugation quandle $c(G)$ is.
\end{lem}

\begin{proof}
The group $\Inn(c(G))$ identifies with $\Inn(G)$, which is isomorphic to $G/\mathcal Z(G)$. Thus, it is $(c-1)$-nilpotent if and only if $G$ is $c$-nilpotent.
\end{proof}

As a consequence, the adjunction between $c(-)$ and $G(-)$ (Prop.~\ref{adj_c_G}) restricts to an adjunction between the full subcategory $\mathcal Qdl_c \subset \mathcal Qdl$ and the full subcategory $\mathcal Grp_c \subset \mathcal Grp$ of $c$-nilpotent groups. This allows us to write the following square of adjunctions:
\[\begin{tikzcd}[column sep=40, row sep=40]
\mathcal Qdl   \ar[r, shift left, "G"] \ar[d, shift right, swap, "(-)/\Gamma_c"]
& \mathcal Grp \ar[l, shift left, "c"] \ar[d, shift right, swap, "(-)/\Gamma_{c+1}"] \\
\mathcal Qdl_c   \ar[r, shift left, "G"] \ar[u, shift right, hook]
& \mathcal Grp_c. \ar[l, shift left, "c"] \ar[u, shift right, hook]
\end{tikzcd}\]
The square of right adjoints obviously commutes, so the square of left adjoints does too. Hence we have proved:
\begin{prop}\label{G(Q/Gamma)}
Let $Q$ be any quandle. Then $G(Q/\Gamma_c) = G(Q)/\Gamma_{c+1}$.
\end{prop}

Recall that $1$-nilpotent quandles are trivial ones (Ex.~\ref{1-nilp=triv}), hence $Q/\Gamma_1$ coincides with the trivialization functor $\pi_0$. Thus, using the calculation of the enveloping group of a trivial quandle (Ex.~\ref{G(triv)}), we see that for $c = 1$, Proposition \ref{G(Q/Gamma)} reduces to the following classical result, which is in fact easy to derive directly from the presentation of $G(Q)$:
\begin{cor}\label{G(Q)_ab}
Let $Q$ be any quandle. Then $G(Q)^{ab} = \Z[\pi_0 Q]$.
\end{cor}
In particular, the abelianisation of $G(Q)$ is always a free abelian group.

\subsection{Free nilpotent quandles}\label{par_free_nilp}

Let $X$ be a set, and $F[X]$ be the free group on $X$. Recall that the free nilpotent quandle $Q[X]$ on $X$ is the subquandle of the conjugation quandle of $F[X]$ given by the (disjoint) union of the conjugacy classes of the elements of the basis $X$. It is actually not difficult to check that this quandle is indeed free on the set $X$: if $f: X \rightarrow Q$ is a map from $X$ to some quandle $Q$, then $\eta_Q \circ f : X \rightarrow G(Q)$ extends uniquely to a group morphism $f_\# : F[X] \rightarrow G(Q)$. Then one checks easily that the unique extension $\tilde f$ of $f$ to a quandle morphism from $Q[X]$ to $Q$ is defined, in terms of the adjoint action of $G(Q)$ on $Q$ (see Example~\ref{adjoint_action}), by:
\[\forall x \in X,\ \forall w \in F[X],\ \tilde f ({}^w\!x) = f_\#(w)\cdot f(x).\]

We now show that the analogous construction works for nilpotent quandles:

\begin{prop}\label{free_nilp_qdl}
The free $c$-nilpotent quandle on a set $X$ is the subquandle of the conjugation quandle of the free $c$-nilpotent group $F[X]/\Gamma_{c+1}$ given by the (disjoint) union of the conjugacy classes of the elements of the basis $X$.
\end{prop}
\begin{proof}
By composing adjunctions, we know that the free $c$-nilpotent quandle on a set $X$ is given by $Q[X]/\Gamma_c$ (see Prop.\ \ref{reflection_functors}). Because of Proposition \ref{G(Q/Gamma)}, we also know that its enveloping group is the free $c$-nilpotent quotient of $G(Q[X])$. But $G(Q[X]) \cong F[X]$ (by composing adjunctions), so that $G\left(Q[X]/\Gamma_c\right)$ is the free nilpotent group $F[X]/\Gamma_{c+1}$. We then deduce from the description of $Q[X]$ recalled above that the image of $\eta : Q[X]/\Gamma_c \rightarrow F[X]/\Gamma_{c+1}$ is exactly the quandle described in the proposition. In order to prove our claim, we need to prove that this map $\eta$ is injective.

Consider the commutative diagram of canonical maps:
\[\begin{tikzcd}
Q[X] \ar[d, two heads, "p"] \ar[r, hook] & F[X] \ar[d, two heads,"q"]\\
Q[X]/\Gamma_c \ar[r, "\eta"] & F[X]/\Gamma_{c+1}.
\end{tikzcd}\]
Consider two elements of $Q[X]/\Gamma_c$. We can write them as $p({}^v\!x)$ and $p({}^w\!y)$, where $v,w \in F[X]$ and $x,y \in X$, so that ${}^v\!x$ and ${}^w\!y$ are two elements of $Q[X]$. Suppose that their images by $\eta$ are the same, that is, that $q({}^v\!x) = q({}^w\!y)$. First, we see that $x = y$, by projecting $q({}^v\!x)$ and $q({}^w\!y)$ into $(F[X]/\Gamma_{c+1})^{ab} = \Z[X]$. Then we can write:
\begin{align*}
q({}^v\!x) = q({}^w\!x)\ 
&\Leftrightarrow\ {}^v\!x \equiv {}^w\!x \pmod{\Gamma_{c+1}F[X]} \\
&\Leftrightarrow\ {}^{w^{-1}v}\!x \equiv x \pmod{\Gamma_{c+1}F[X]} \\
&\Leftrightarrow\ [w^{-1}v,x] \in \Gamma_{c+1}F[X].
\end{align*} 
By applying \cite[Lem.~6.3]{Darne1}, we deduce that there exists an integer $m$ such that $w^{-1}v x^m \in \Gamma_c F[X]$. Since we can replace $v$ by $v x^m$ without changing the element ${}^v\!x$, we can assume that $w^{-1}v$ is in $\Gamma_c F[X]$. Then $\gamma :=  v w^{-1} = w(w^{-1}v)w^{-1} \in \Gamma_c F[X]$ satisfies ${}^\gamma({}^w\!x) = {}^v\!x$. Hence ${}^v\!x$ and ${}^w\!x$ are in the same orbit modulo $\Gamma_c F[X]$, which means exactly that $p({}^v\!x) = p({}^w\!x)$. Thus we have proved that $\eta$ is injective, whence the result.
\end{proof}

\subsection{Residually nilpotent quandles}

Recall that a group $G$ is called \emph{residually nilpotent} if any non-trivial element $x \in G$ has a non-trivial class in some nilpotent quotient of $G$. In particular, nilpotent groups are residually nilpotent. By analogy, we make the following definition:
\begin{defi}
A quandle $Q$ is called \emph{residually nilpotent} when, for all $x,y \in Q$ satisfying $x \neq y$, there exists a morphism $f: Q \twoheadrightarrow R$ to a nilpotent quandle $R$ such that $f(x) \neq f(y)$.
\end{defi}

Remark that it is equivalent to asking that the classes of $x$ and $y$ are different in $Q/\Gamma_c$ for some $c \geq 1$ (the morphism $f$ of the definition factors through $Q/\Gamma_c$ if $c$ is the nilpotency class of $R$).

\begin{prop}\label{Inn_res_nilp}
If $Q$ is a residually nilpotent quandle, then $\Inn(Q)$ is a residually nilpotent group.
\end{prop}

\begin{proof}
Suppose that $Q$ is residually nilpotent, and let $\varphi \in \Inn(Q) - \{id\}$. There is some $x \in Q$ such that $\varphi(x) \neq x$. Since $Q$ is residually nilpotent, there is some $c \geq 1$ such that $\varphi(x) \not\equiv x\ [\Gamma_c(\Inn(Q))]$. This implies that $\varphi \notin \Gamma_c(\Inn(Q))$, whence our claim.
\end{proof}

There is a universal residually nilpotent quotient of a quandle $Q$, which is the quotient $Q/{\sim}$ by the equivalence relation:
\[x \sim y \ \ \Leftrightarrow\ \ \forall c \geq1,\ x \equiv y\ [\Gamma_c].\]
This relation is coarser than congruence modulo $\Gamma_\infty(\Inn(Q))$, possibly strictly so. Otherwise said, it is not clear whether the converse of Proposition~\ref{Inn_res_nilp} should hold in general. 

\begin{rmq}
The relation to the residual nilpotency of the enveloping group is also not clear, despite Lemma~\ref{central_ext}.
\end{rmq}

For injective quandles, the situation is clearer, because of:
\begin{lem}
Let $G$ be a residually nilpotent group. Then $c(G)$ is a residually nilpotent quandle.
\end{lem}

\begin{proof}
Let $x,y \in G$ such that $x \neq y$. Since $G$ is residually nilpotent, there is a group morphism $f: G \rightarrow N$ such that $f(x) \neq f(y)$ and $N$ is nilpotent. By Lemma~\ref{c(G)_nilp}, $c(f) : c(G) \rightarrow c(N)$ is a morphism to a nilpotent quandle, whence the conclusion.
\end{proof}

In particular, since a subquandle of a residually nilpotent quandle is clearly residually nilpotent, we get:
\begin{prop}\label{res_nilp_for_inj}
If $Q$ is injective and $G(Q)$ is residually nilpotent, then $Q$ is too.
\end{prop}

\begin{ex}
Free quandles are residually nilpotent, since free groups are, and they are injective.
\end{ex}

\begin{ex}
Reduced free quandles (introduced in~\cref{par_reduced_free} below) are residually nilpotent; see Remark~\ref{RQ[X]_res_nilp}.
\end{ex}

\section{Nilpotency for groups and quandles}\label{sec_nilp_group_and_qdls}

Nilpotent groups enjoy very nice properties, and we now investigate analogues of these properties for nilpotent quandles.

\subsection{Basic properties of nilpotent quandles}

\subsubsection*{Some group theory}

Let us recall a few classical elementary results from the theory of nilpotent groups:
\begin{prop}\label{surj_between_nilp_groups}
Let $G$ be a group and let $N$ be a nilpotent group. A morphism $f : G \rightarrow N$ is surjective if and only if the induced morphism $\bar f : G^{ab}\rightarrow N^{ab}$ is. 
\end{prop}

\begin{proof}
Suppose that the induced morphism $\bar f : G^{ab}\rightarrow N^{ab}$ is surjective. This morphism is the degree-one part of the morphism $\gr(f) : \gr(\Gamma_*G) \rightarrow \gr(\Gamma_*N)$ between the graded Lie rings associated to the lower central series. These Lie rings are generated in degree one \cite[Prop.~1.9]{Darne1} (see also~\cref{reminders}), thus $\gr(f)$ is surjective. Then, by applying the Five Lemma to the diagram:
\[\begin{tikzcd}
\gr_i(\Gamma_*G) \ar[r, hook] \ar[d, two heads, "\gr_i(f)"]
& G/\Gamma_{i+1} \ar[r, two heads] \ar[d, "\bar f"]
& G/\Gamma_i \ar[d, "\bar f"] \\
\gr_i(\Gamma_*N) \ar[r, hook] 
& N/\Gamma_{i+1} \ar[r, two heads]
& N/\Gamma_i,
\end{tikzcd}\]
we deduce, by induction, that the induced morphism $\bar f : G/\Gamma_i\rightarrow N/\Gamma_i$ is surjective. Since there exists some $i$ satisfying $N = N/\Gamma_i$, we get that $f : G \twoheadrightarrow G/\Gamma_i \cong N/\Gamma_i = N$ is surjective.
\end{proof}

\begin{cor}\label{generators_of_nilp_groups}
Let $S \subseteq G$ be a subset of a nilpotent group $G$. Then $S$ generates $G$ if and only if its image in $G^{ab}$ generates $G^{ab}$.
\end{cor}

\begin{proof}
Apply Prop.~\ref{surj_between_nilp_groups} to the morphism $F[S] \rightarrow G$ induced by the inclusion $S \subseteq G$ (where $F[S]$ is the free group on $S$).
\end{proof}

\begin{cor}\label{cyclic_abelianisation}
Let $G$ be a nilpotent group. If $G^{ab}$ is cyclic, then $G$ is cyclic.
\end{cor}

\begin{proof}
Consider $t \in G$ whose class $\bar t$ generates $G^{ab}$, and apply Cor.~\ref{generators_of_nilp_groups} to $S = \{t\}$. 
\end{proof}

\subsubsection*{Some quandle theory}

We now show that the analogues of the above group-theoretic results hold for quandles. In fact, the following result will allow us to deduce quandle-theoretic properties from group-theoretic ones:

\begin{prop}\label{surj_on_G}
A quandle morphism $f: Q \rightarrow R$ is surjective if and only if the induced group morphism $f_*: G(Q) \rightarrow G(R)$ is.
\end{prop}

\begin{proof}
The ``only if" part follows readily from the fact that the image of $R$ generates $G(R)$. Let us show the converse. The identification $G(Q)^{ab} \cong \Z[\pi_0 Q]$ from Corollary~\ref{G(Q)_ab} is natural in $Q$. In particular, the map induced by $f_*$ on the abelianisations identifies with $\Z[\pi_0(f)]$. Thus, if $f_*$ is surjective, $\pi_0(f)$ must be too. This means that for each $y \in R$, there exists an element $x'$ of $Q$ such that $f(x')$ is in the orbit of $y$. Let us write $y = h \cdot f(x')$ for some $h \in G(R)$. Since $f_*$ is surjective, $h$ must be equal to $f_*(g)$, for some $g \in G(Q)$. Then, using Lemma~\ref{equivariance}, $y = f_*(g) \cdot f(x') = f(g \cdot y) \in f(Q)$, and our claim is proved.
\end{proof}

\begin{prop}\label{surj_between_nilp_qdles}
Let $Q$ be a quandle and let $R$ be a nilpotent quandle. A morphism $f : Q \rightarrow R$ is surjective if and only $\pi_0(f)$ is.
\end{prop}

\begin{proof}
Suppose that the induced morphism $\pi_0(f) : \pi_0(Q)\rightarrow \pi_0(R)$ is surjective. We have $G(-)^{ab} \cong \Z[\pi_0 (-)]$ (Cor.~\ref{G(Q)_ab}), hence the induced morphism from $G(Q)^{ab}$ to $G(R)^{ab}$ is surjective. Then, using the fact that $G(R)$ is nilpotent (Lem.~\ref{central_ext}), we can apply Proposition \ref{surj_between_nilp_groups} to deduce that $f_* : G(Q) \rightarrow G(R)$ is surjective. Thus Proposition~\ref{surj_on_G} gives the desired conclusion.
\end{proof}

\begin{cor}\label{generators_of_nilp_qdles}
Let $S \subseteq Q$ be a subset of a nilpotent quandle $Q$. Then $S$ generates $Q$ if and only if there is one element of $S$ in each orbit of $Q$.
\end{cor}

\begin{proof}
Apply Prop.~\ref{surj_between_nilp_qdles} to the morphism $Q[S] \rightarrow Q$ induced by the inclusion $S \subseteq Q$ (where $Q[S]$ is the free quandle on $Q$).
\end{proof}

\begin{cor}\label{nilp_connected_are_*}
Let $Q$ be connected (that is, $|\pi_0(Q)| = 1$) and nilpotent. Then $Q \cong *$.
\end{cor}

\begin{proof}
Apply Cor.~\ref{generators_of_nilp_qdles} to $S = \{q\}$ for any $q \in Q$. 
\end{proof}

\begin{rmq}
Let $Q$ be a quandle whose enveloping group is cyclic. Then (using Lemma \ref{Q_nilp_ssi_G(Q)_nilp}), we can apply Cor.~\ref{nilp_connected_are_*} to show that $Q \cong *$. The reader can convince themself that this is not a completely obvious result. In fact, we will see later (\cref{par_2-orbits_2-nilp}) that quandles with \emph{abelian} enveloping groups are not necessarily trivial.
\end{rmq}

\begin{cor}\label{Hom(connected, nilp)}
Any morphism from a connected quandle to a nilpotent one is constant.
\end{cor}

\begin{proof}
If $f : Q \rightarrow R$ is such a morphism, then $f(Q)$ is connected (a quotient of a connected quandle is connected by Lemma~\ref{equivariance}) and nilpotent by Lemma~\ref{quotients_and_subquandles_of_nilp}, so it is reduced to a point.
\end{proof}

Proposition~\ref{surj_on_G} also allows us to study epimorphisms in the category of quandles. Precisely, we deduce from the case of groups:
\begin{prop}
Surjections coincide with epimorphisms in the category of quandles.
\end{prop}

\begin{proof}
Recall that surjections are always epimorphisms. Precisely, in a category endowed with a faithful forgetful functor $U$ to $Set$, if we call surjective a morphism $p$ such that $U(p)$ is a surjection, then we easily see that surjections must be epimorphisms. Recall also that epimorphisms are surjections in the category of groups. Indeed, a non-surjective morphism $\varphi : G \rightarrow H$ with image $K \subsetneq H$ will satisfy $i_1 \circ \varphi = i_2 \circ \varphi$ if $i_1$ and $i_2$ are the canonical maps from $H$ to $H *_K H$ ($\varphi$ coequalizes its cokernel pair), but the usual theory of amalgamated products of groups says that $i_1 \neq i_2$, so $\varphi$ is not an epimorphism. Finally, recall that left adjoints preserve epimorphisms. Now, if $p$ is an epimorphism of quandles, then $G(p)$ is an epimorphism of groups (since $G$ is a left adjoint). Hence $G(p)$ is surjective, and Proposition~\ref{surj_on_G} allows us to conclude that $p$~itself is surjective.
\end{proof}

\begin{rmq}
This answers a question asked by the referee. It will not be used in the paper, but it should reassure the categorically-minded reader who would wonder whether surjections are the right epimorphisms to consider in this context. Notice that this says that the class of surjection also coincides with the class of strong epimorphisms, or with the class of regular epimorphisms (that is, coequalizers) in the category of quandles. In particular, it says that our version of the Hopf property for quandle (see \cref{par_Hopf} below) is the only reasonable one. 
\end{rmq}

\subsection{Decomposition as an extension}

One very important tool in the study of nilpotent groups, which allows us to reason by induction on the nilpotency class, is the decomposition of a $c$-nilpotent group $G$ into a central extension:
\[\begin{tikzcd}
\Gamma_c(G) \ar[r, hook] &G \ar[r, two heads] &G/\Gamma_c(G).
\end{tikzcd}\]

There is a similar decomposition of a nilpotent quandle, that will play the same role. It is a decomposition into an extension of a quandle, in the following sense (which is slightly different from Eisermann's \cite[Def.~4.14]{Eisermann} -- see Remark ~\ref{rk_extensions}):
\begin{defi}
Let $R$ be a quandle and $\Lambda$ be an abelian group. A \emph{$\Lambda$-extension} (or a \emph{$\Lambda$-covering}) of $R$ consists in a covering $p: Q \twoheadrightarrow R$, together with a group action $\Lambda \circlearrowright Q$ by quandle automorphisms, such that for each $r \in R$, the action of $\Lambda$ preserves the fibre $p^{-1}(r)$ and acts transitively on it. We call $\Lambda$ a \emph{structure group} of the extension. If moreover $\Lambda$ acts freely on the fibers, then the extension is called \emph{principal} (or a \emph{principal} $\Lambda$-covering). Since $R$ can be recovered as the quotient $Q/\Lambda$, we will sometimes simply say that $Q$ (endowed with some action of $\Lambda$) is a $\Lambda$-extension.
\end{defi}

The fact that $\Lambda$ preserves the fibers of the covering $p$ implies that for all $\lambda \in \Lambda$ and $x \in Q$, the elements $\lambda \cdot x$ and $x$ (which have the same image by $p$) behave in the same way. As a consequence, the condition saying that the action is by quandle morphisms reads:
\begin{equation}\label{eq_commuting_actions}
\forall \lambda \in \Lambda,\ \forall x,y \in Q,\ \lambda \cdot (x \qd y) =  (\lambda \cdot x) \qd (\lambda \cdot y) = x \qd (\lambda \cdot y).
\end{equation}

\begin{rmq}\label{rk_extensions}
Since $\Lambda$ is assumed to act transitively on the fibers, we can, instead of asking $p$ to be a covering, ask all the elements in a given $\Lambda$-orbit to behave in the same way. This, together with the above, allows us to reformulate our definition in terms closer to Eisermann's \cite[Def.~4.14]{Eisermann}: a $\Lambda$-extension of $R$ (in our sense) is a surjective morphism $p: Q \twoheadrightarrow R$, together with an action of the group $\Lambda$ on the set $Q$ such that:
\begin{itemize}[itemsep=-0.3em, topsep = 4pt]
    \item For all $x, y \in Q$ and for all $\lambda \in \Lambda$, $(\lambda \cdot x) \qd y = x \qd y$ \text{ and } $x \qd (\lambda \cdot y) = \lambda \cdot (x \qd y)$;
    \item The action preserves the fibers of $p$ and is transitive on each one.
\end{itemize}
Thus, our definition is weaker than his, which corresponds exactly to our \emph{principal extensions}. However, as will become apparent after the discussion below, if $R$ is connected, every $\Lambda$-extension (in our sense) which is itself connected is a principal extension of $R$ by some quotient of $\Lambda$. 
\end{rmq}

The condition \eqref{eq_commuting_actions} says that the action of $\Lambda$ commutes with the $x \qd(-)$, that is, with the action of $\Inn(Q)$. In other words, the action of $\Inn(Q)$ on $Q$ is $\Lambda$-equivariant: for all $g \in \Inn(Q)$, for all $\lambda \in \Lambda$, and all $q \in Q$, we have $g \cdot (\lambda \cdot q) = \lambda \cdot (g \cdot q)$. As a consequence:
\begin{fact}\label{comparison_of_Lambda-orbits}
Let $Q$ be a $\Lambda$-extension. If $q$ and $q'$ are in the same orbit of $Q$, then $\Lambda \cdot q \cong \Lambda \cdot q'$ as $\Lambda$-sets, an isomorphism being induced by any $g \in \Inn(Q)$ sending $q$ to~$q'$.
\end{fact}

Recall that for an abelian group $\Lambda$, isomorphism classes of transitive $\Lambda$-sets correspond bijectively to subgroups of $\Lambda$. Precisely, on the one hand, we associate to each transitive $\Lambda$-set $X$ the stabilizer $H$ of any point $x \in X$. The fact that $H$ does not depend on the choice of $x$ comes from the commutativity of $\Lambda$, since stabilizers of two points of $X$ are conjugate. On the other hand, we associate to each subgroup $H$ of $\Lambda$ the $\Lambda$-set $\Lambda/H$. These constructions are inverse to each other, an isomorphism $X \cong \Lambda/\Stab(x)$ being given by $\bar g \mapsto g \cdot x$. Thus, the above fact implies that if $p$ is a $\Lambda$-extension, the induced map from the orbit of some $x$ to the orbit of $p(x)$ is a principal $(\Lambda/\Stab(x))$-extension.

\medskip

Let us get back to nilpotent quandles. We already have a projection of a $(c+1)$-nilpotent quandle $Q$ onto its universal $c$-nilpotent quotient $Q/\Gamma_c$
\begin{prop}\label{dec_as_ext}
Let $Q$ be a $(c+1)$-nilpotent quandle. Then $Q \twoheadrightarrow Q/\Gamma_c$ is a $\Gamma_c(\Inn(Q))$-extension.
\end{prop}

\begin{proof}
The projection $Q \twoheadrightarrow Q/\Gamma_c$ is the quotient of $Q$ by the action of $\Lambda := \Gamma_c(\Inn(Q))$ (as in Lem.~\ref{quotient_by_group}). This action is by quandle automorphisms and, by definition of the quotient, the fibres of this projection are transitive $\Lambda$-sets. Moreover, since $Q$ is $(c+1)$-nilpotent, we have seen (see the proof of Th.~\ref{nilpotency_via_coverings}) that the projection onto $Q/\Gamma_c$ is a covering (it is a quotient by a central subgroup of $\Inn(Q)$).
\end{proof}

\subsection{The Hopf property}\label{par_Hopf}

The \emph{Hopf property} is one very important feature of finitely generated (residually) nilpotent groups: any surjective endomorphism of such a group is bijective. 
Groups satisfying this Hopf property are also called \emph{hopfian}. This can be seen as a kind of finiteness property (any finite group is obviously hopfian).

We now show that (residually) nilpotent quandles do satisfy the same property:
\begin{theo}\label{Hopf_prop}
Let $Q$ be a finitely generated nilpotent quandle. Then $Q$ is \emph{hopfian}, that is, any surjective endomorphism of $Q$ is bijective.
\end{theo}

\begin{cor}
A finitely generated residually nilpotent quandle is hopfian.
\end{cor}

\begin{proof}[Proof of the Corollary]
Let $Q$ be a residually nilpotent quandle and $f: Q \twoheadrightarrow Q$ be a surjective morphism. Let $x, y \in Q$ be distinct elements. Since $Q$ is residually nilpotent, their classes $\bar x$ and $\bar y$ are distinct in $Q/\Gamma_c$, for some integer $c \geq 1$. The endomorphism $f$ induces a surjective endomorphism $\bar f$ of the finitely generated nilpotent quandle $Q/\Gamma_c$, which must be injective by Theorem~\ref{Hopf_prop}. Thus $\overline{f(x)} = \bar f(\bar x) \neq \bar f(\bar y) = \overline{f(y)} $, whence $f(x) \neq f(y)$.
\end{proof}

\begin{proof}[Proof of Theorem \ref{Hopf_prop}]
The proof is by induction on the nilpotency class $c$ of $Q$, relying on the decomposition of Prop.~\ref{dec_as_ext} for the induction step. The case $c = 1$ is obvious. Indeed, the only generating set of a trivial quandle is the whole quandle, so a finitely generated trivial quandle must be finite, and a surjection from a finite set onto itself is a bijection. Now, let us suppose that the result holds for finitely generated $c$-nilpotent quandles, and let us show that any finitely generated $(c+1)$-nilpotent quandle is hopfian.

Let $f$ be a surjective endomorphism of a finitely generated $(c+1)$-nilpotent quandle~$Q$, and let $x, y \in Q$ such that $f(x) = f(y)$. We want to show that $x = y$. Let us first consider the endomorphism $\bar f$ of $Q/\Gamma_c$ induced by $f$. It is a surjective endomorphism of a finitely generated $c$-nilpotent quandle, to which we can apply our induction hypothesis: $\bar f$ is an automorphism. If $\bar x$ and $\bar y$ denote the respective classes of $x$ and $y$ in $Q/\Gamma_c$, we have $\bar f(\bar x) = \overline{f(x)} = \overline{f(y)} =  \bar f(\bar y)$, hence $\bar x = \bar y$. Using the description of the universal $c$-nilpotent quotient as the quotient by $\Lambda := \Gamma_c(\Inn(Q))$, we see that it means that $x$ and $y$ are in the same orbit under the action of $\Lambda$. Thus, in order to conclude, we need to show that the restriction of $f$ to each $\Lambda$-orbit is injective.

Our quandle $Q$ is finitely generated, so $\Inn(Q)$ is too. Since  finitely generated nilpotent groups are noetherian (see the remarks at the end of \cref{par_Hopf}), the subgroup $\Lambda$ is also finitely generated. The surjective endomorphism $f$ induces a surjective endomorphism $f_*$ of $\Inn(Q)$ (see Proposition~\ref{Inn(surj)}). By definition of the lower central series, the surjectivity of $f$ implies $f_*(\Lambda) = \Lambda$ (in fact, because of the Hopf property for finitely generated nilpotent groups, $f_*$ must be an automorphism, inducing an automorphism of the characteristic subgroup $\Lambda$ of $\Inn(Q)$). Recall that $f$ is $f_*$-equivariant (Lemma~\ref{equivariance}): for all $q \in Q$ and all $g \in \Inn(Q)$, $f(g \cdot q) = f_*(g) \cdot f(q)$.

Using the fact that $f_*(\Lambda) = \Lambda$ and the equivariance, we see that for every $x \in Q$, the endomorphism $f$ restricts to a surjective map $\Lambda \cdot x \twoheadrightarrow \Lambda \cdot f(x)$. Since $x$ and $f(x)$ may not be in the same orbit of $Q$, we cannot make a direct use of the comparison between $\Lambda$-orbits from Fact \ref{comparison_of_Lambda-orbits}. However, as in the beginning of the proof, $\pi_0(Q)$ is a finitely generated trivial quandle, so it is finite, and $f$ induces a permutation $\pi_0(f)$ of this set, which must be of finite order $n$. And if we show that the surjective endomorphism $f^n$ is injective, then $f$ will be. As a consequence, by replacing $f$ with $f^n$, we can assume that $\pi_0(f)$ is the identity. Then $x$ and $f(x)$ are in the same orbit of $Q$, so Fact \ref{comparison_of_Lambda-orbits} says that $\Lambda \cdot x$ and $\Lambda \cdot f(x)$ are isomorphic as (transitive) $\Lambda$-sets. Thus, after identifying them, we can apply Lemma~\ref{alpha-Endo_Lambda-sets} to the $f_*$-equivariant map $\Lambda \cdot x \twoheadrightarrow \Lambda \cdot f(x)$ to get that the restriction of $f$ to each $\Lambda$-orbit is injective, whence our claim. 
\end{proof}

The above proof relies on Lemma~\ref{alpha-Endo_Lambda-sets} below, whose proof uses the next two Lemmas.

\begin{lem}\label{Endo_Lambda-sets}
Let $\Lambda$ be an abelian group. Let $X$ be a transitive $\Lambda$-set. Then any equivariant map $f: X \rightarrow X$ is a bijection.
\end{lem}

\begin{proof}
The commutativity of $\Lambda$ implies that for $\lambda \in \Lambda$, $\lambda \cdot(-) : X \rightarrow X$ is $\Lambda$-equivariant Let us choose $x \in X$ and (by transitivity) $\lambda_0 \in \Lambda$ such that $f(x) = \lambda_0 \cdot x$.  For any $\lambda \in \Lambda$, we have $f(\lambda \cdot x) = \lambda \cdot f(x) = \lambda \lambda_0 \cdot x = \lambda_0 \cdot (\lambda \cdot x)$. Thus $f = \lambda_0 \cdot (-)$ is invertible (with inverse $\lambda_0^{-1} \cdot (-)$).
\end{proof}

\begin{lem}\label{full_stabilization}
Let $G$ be a group, and $H$ be a normal subgroup such that $G/H$ is Hopfian. If an automorphism $\alpha$ of $G$ satisfies $\alpha(H) \subseteq H$, then $\alpha^{-1}(H) \subseteq H$, i.e.\ $\alpha(H) = H$.
\end{lem}

\begin{proof}
If $\alpha(H) \subseteq H$, then $\alpha$ induces an endomorphism $\bar \alpha$ of $G/H$, which is clearly surjective. Since $G/H$ is hopfian, $\bar \alpha$ is bijective. But $\ker(\bar \alpha)$ is the projection of $\alpha^{-1}(H)$ in $G/H$, which is trivial if and only if $\alpha^{-1}(H) \subseteq H$.
\end{proof}

\begin{lem}\label{alpha-Endo_Lambda-sets}
Let $\Lambda$ be a finitely generated abelian group, and $\alpha \in \Aut(\Lambda)$. Let $X$ be a transitive $\Lambda$-set. Then any $\alpha$-equivariant map $f: X \rightarrow X$ is a bijection.
\end{lem}

\begin{proof}
For all $\lambda \in \Lambda$,  for all $x \in X$, we have $f(\lambda \cdot x) = \alpha(\lambda) \cdot f(x)$. Thus, if $\lambda \in \Stab(x)$ (for some $x \in X$), then  $\alpha(\lambda) \in \Stab(f(x))$. Since $X$ is a transitive $\Lambda$-set and $\Lambda$ is abelian, the subgroup $H := \Stab(x)$ of $\Lambda$ does not depend on $x \in X$, so the previous sentence translates as: $\alpha(H) \subseteq H$. Since $\Lambda/H$ is abelian and finitely generated, it is hopfian, so we can apply Lemma~\ref{full_stabilization} to get $\alpha(H) = H$. Now, $f$ can be seen as a $\Lambda$-equivariant morphism from $X$ to $\alpha^*X$, where $\alpha^*X$ is $X$ as a set, on which $\lambda \in \Lambda$ acts via the action of $\alpha(\lambda)$ on $X$. Moreover, in the same way as $X \cong \Lambda/H$, we have $\alpha^*X \cong \Lambda/\alpha^{-1}(H)$, since $\alpha^{-1}(H)$ is the common stabilizer of the elements of $\alpha^*X$. Here, $\alpha^{-1}(H) = H$, so $X \cong \alpha^*X$ as $\Lambda$-sets. By identifying these, we can see $f$ as an equivariant map from $X$ to itself, which has to be an isomorphism by Lemma~\ref{Endo_Lambda-sets}.
\end{proof}

\subsubsection*{Noetherianity}

Another important finiteness property of nilpotent groups is \emph{noetherianity}. Recall that a group is called \emph{noetherian} if every ascending chain of subgroups stops, which is equivalent to every subgroup being finitely generated. For instance, cyclic groups are noetherian. It is easy to check that an extension of two noetherian groups is noetherian. Since finitely generated nilpotent groups are polycyclic, they are noetherian. 

By replacing the words ``group" and ``subgroup", respectively, by ``quandle" and ``subquandle" in the definition above (in fact in both of them), we get the notion of \emph{noetherian quandles}. One could then ask if finitely generated nilpotent quandles are noetherian. It turns out that they are not, in general:

\begin{ex}
In $2$-nilpotent quandles, orbits are trivial subquandles. A trivial subquandle is finitely generated if and only if it is finite. And there are (many) finitely generated $2$-nilpotent quandles with infinite orbits, which are thus not noetherian. Such quandles will appear in \cref{Applications_classification}; the simplest one is $Q_{1,0} = \{a\} \sqcup \Z$, whose law is given by $a \qd n = n+1$ and $n \qd (-) = id$ for $n \in \Z$, so that $\{a\}$ and $\Z$ are its orbits.
\end{ex}

\section{Construction of quandles}\label{sec_construction}

We now turn to the construction of quandles from group-theoretic data. In \cref{qdl_from_grp}, we recall, with a slightly new point of view, a classical construction of any quandle from such data. Then, in~\cref{par_crossed_sets}, we reinterpret this construction as a functorial correspondence between quandles and so-called \emph{peripheral systems}.

\subsection{Quandles from groups}\label{qdl_from_grp}

We first recall a construction due to Joyce \cite[\S 7]{Joyce}, allowing one to construct any quandle from a group $G$ together with some group-theoretic data. Our setting is slightly different from Joyce's, in order to allow group actions that are not necessarily faithful. We begin with the following easy result:

\begin{prop}\label{prop_qd_from_gp}
Let $Q$ be a quandle, and let a group $G$ act on $Q$ by quandle automorphisms, such that the image of the corresponding morphism $G \rightarrow \Aut(Q)$ contains $\Inn(Q)$. Let us choose one element $q_i$ in each $G$-orbit. Let $z_i \in G$ act via $q_i \qd(-)$, and let $H_i$ denote $\Stab(q_i)$. Then the $G$-equivariant map 
\[f:\left\{\renewcommand{\arraystretch}{1.5}
\begin{array}{clc}
\bigsqcup\limits_{i \in I} G/H_i &\longrightarrow &Q \\
gH_i &\longmapsto & g \cdot q_i
\end{array}
\right.\]
is a bijection. The corresponding quandle law on $\bigsqcup G/H_i$ (with respect to which $f$ becomes a quandle isomorphism) reads:
\[\forall i,j \in I,\ \forall x,y \in G,\ \ xH_i\ \qd\  yH_j := x z_i x^{-1} yH_j.\]
\end{prop}

\begin{proof}
The kernel of the action of $G$ on $Q$ is contained in all the $H_i$, so we can replace $G$ by its image in $\Aut(Q)$ without altering the construction: we can assume that $G$ is a subgroup of $\Aut(Q)$ containing $\Inn(Q)$. The fact that $f$ is well-defined, $G$-equivariant and bijective comes directly from the decomposition of the $G$-set $Q$ into $G$-orbits: $Q =  \bigsqcup G \cdot q_i$. We only need to check that the quandle law described in the statement is the one obtained by transport of structure along $f$:
\begin{align*}
\forall i,j \in I,\ \forall x,y \in G,\ \ f(xH_i) \qd f(yH_j)
&= (x \cdot q_i) \qd (y \cdot q_j) = x \cdot (q_i \qd( x^{-1} y \cdot q_j)) \\
&= (x z_i x^{-1} y) \cdot q_j = f(x z_i x^{-1} y H_j),
\end{align*}
where we have used that $x \in G$ acts via a quandle automorphism and that $z_i$ acts via $q_i \qd(-)$.
\end{proof}

\begin{ex}\label{Q_from_G(Q)}
By applying Proposition~\ref{prop_qd_from_gp} to $G(Q)$ acting on $Q$, we see that although one cannot recover $Q$ from $G(Q)$, one can recover $Q$ from $G(Q)$, together with elements $z_i$ ($= \eta(q_i)$) and subgroups $H_i$ ($=\Stab(q_i)$), indexed by orbits of $Q$ (that is $I = \pi_0(Q)$). Using the notation introduced in Lemma~\ref{lem_qd_from_gp} below, we have:
\[Q \cong \mathcal Q \left(G(Q), (H_i,z_i)_{i \in I} \right).\]
This isomorphism will play a fundamental role in the sequel, so we introduce the notation:
\[\mathcal P\left(Q, (q_i)_{i \in I}\right) := \left(G(Q), (\Stab(q_i),\eta(q_i))_{i \in I} \right).\]
We will study the above isomorphism $Q \cong \mathcal Q (\mathcal P(Q))$ with a categorical perspective in \cref{par_crossed_sets}, where it will be seen as a natural isomorphism $\mathcal Q \circ \mathcal P \cong id$ between functors (Proposition~\ref{QP=1}). Note that this recipe is the one used to recover the fundamental quandle from the peripheral system, and we will also use it to recover the reduced fundamental quandle from the reduced peripheral system (see~\cref{par_reduced_periph} for both statements). 
\end{ex}

\begin{ex}\label{Q_from_Aut(Q)}
We can also apply Proposition~\ref{prop_qd_from_gp} to $G = \Aut(Q)$ acting on $Q$. If we choose representatives $q_i$ of $G$-orbits, we can take $z_i = q_i \qd (-)$ and $H_i=\Stab(q_i)$, we get (using again the notation introduced in Lemma~\ref{lem_qd_from_gp} below):
\[Q \cong \mathcal Q \left(\Aut(Q), (\Stab(q_i), q_i \qd (-))_{i \in Q/\!\Aut(Q)} \right).\]
However, this construction is usually less relevant than the previous one, because of the lack of functoriality of $\Aut(Q)$ with respect to $Q$.
\end{ex}

\begin{rmq}
In the sequel, we will often use this result with a free group as our acting $G$. Then, the distinction between $G$ and its image in $\Aut(Q)$ will come in handy: the latter can have a much more complicated structure than the former.
\end{rmq}

\begin{rmq}\label{hyp_G_contains_Inn(Q)}
In the statement of Proposition~\ref{prop_qd_from_gp}, we ask that for any inner automorphism $\varphi$ of $Q$, there should be an element of $G$ which acts via $\varphi$. This may seem too strong an hypothesis, since we have used it only for the $q_i \qd(-)$. However, if $z_i$ acts via $q_i \qd(-)$, then $g z_i g^{-1}$ acts via $(g \cdot q_i) \qd(-)$. Indeed, since $g$ acts via a quandle automorphism, we have:
\[\forall q \in Q,\ \ g z_i g^{-1}\cdot q = g \cdot (q_i \qd (g^{-1}\cdot q)) = (g \cdot q_i) \qd q.\]
But by definition of the $q_i$, the $g \cdot q_i$ describe all the elements of $Q$. As a consequence, if there are $z_i$ in $G$ acting via the $q_i \qd(-)$, then any inner automorphism of $Q$ must be realised as the action of some element of $G$.
\end{rmq}

Proposition~\ref{prop_qd_from_gp} tells us how to recover a given quandle from some group-theoretic data. It is then natural to ask when we can construct a quandle from such data.

\begin{lem}\label{lem_qd_from_gp}
Let $G$ be a group, $(H_i)_{i\in I}$ be a family of subgroups of $G$, and $(z_i)_{i\in I}$ be a family of elements of $G$ indexed by the same set $I$. Then the formulas:
\[\forall i,j \in I,\ \forall x,y \in G,\ \ xH_i\ \qd\  yH_j := x z_i x^{-1} yH_j\]
do define a quandle structure on $\bigsqcup G/H_i$ if and only if the following conditions hold, for all $i \in I$:
\begin{itemize}[itemsep = -3pt, topsep= 3pt]
\item The action of $H_i$ on $\bigsqcup G/H_j$ commutes with the action of $z_i$.
\item $z_i \in H_i$.
\end{itemize}
When these conditions hold, the resulting quandle will be denoted by $\mathcal Q \left(G, (H_i,z_i)_{i \in I} \right)$.
\end{lem}

\begin{proof}
The operation is well-defined if and only if, for all $i \in I$, for all $x,y \in G$, for all $h \in H_i$ and $k \in H_j$, we have $(xh) z_i (xh)^{-1} (yk)H_j = x z_i x^{-1} yH_j$. This is equivalent to $h z_i h^{-1} x^{-1} y H_j = z_i x^{-1} yH_j$, which in turn (by defining $t :=  x^{-1} y$ and $h' := h^{-1}$) is equivalent to $(z_i h') \cdot t H_j = (h' z_i) \cdot t H_j$ (for all $t \in G$ and all $h' \in H_i$). In other words, the action of $z_i$  on $G/H_j$ must commute with the action of $H_i$. The operation is thus well-defined if and only if the latter holds for every $j \in I$, which is exactly our first condition. It is then easy to check the first two quandle axioms, and the third is equivalent to $z_i H_i = H_i$, which holds if and only if $z_i \in H_i$.
\end{proof}

\begin{rmq}
These conditions are obviously satisfied when the data comes from a quandle by the construction of Proposition~\ref{prop_qd_from_gp}.
\end{rmq}

\begin{rmq}\label{z_i_in_H_i}
The hypothesis $z_i \in H_i$ is only used in the verification of the third quandle axiom; without this assumption, one does only get a \emph{rack}.
\end{rmq}

\begin{rmq}
The first condition is obviously satisfied if $z_i$ centralizes $H_i$. In fact, it is \emph{equivalent} to $z_i$ centralizing $H_i$ when the action of $G$ on $\bigsqcup G/H_j$ is faithful.
\end{rmq}

\begin{rmq}
The first condition can also be formulated  as follows: $[H_i, z_i] \subseteq \bigcap g H_j g^{-1}$ (where the intersection is over $j \in I$ and $g \in G$). Indeed, $K := \bigcap g H_j g^{-1}$ is the kernel of the action of $G$ on $\bigsqcup G/H_j$, so $h z_i h^{-1} z_i^{-1} \in K$ if and only if $h z_i \equiv z_i h\ [K]$, which is equivalent to the action of $z_i$ on $\bigsqcup G/H_j$ commuting with the action of $h$.
\end{rmq}

\subsection{Quandles, crossed sets and peripheral systems}\label{par_crossed_sets}

The constructions from \cref{qdl_from_grp} can be enhanced to a functorial correspondence between quandles and a category of groups endowed with some structure, that we call \emph{peripheral systems} (a terminology we borrow from topology). This correspondence is best described with the intermediate of \emph{crossed sets}.

\subsubsection*{Quandles and crossed sets}

Let us recall the classical definition of a crossed $G$-set from \cite[\S 4.2]{Freyd-Yetter}, which is a weak version of the very classical notion of a crossed $G$-module in homotopy theory:
\begin{defi}
Let $G$ be a group. A \emph{crossed $G$-set} $(X, \partial)$ is a set $X$ endowed with an action of $G$ and with a map $\partial : X \rightarrow G$ which is $G$-equivariant with respect to the conjugation action of $G$ on itself:
\[\forall g \in G,\ \forall x \in X,\ \partial(g \cdot x) = g \partial (x) g^{-1}.\]
A \emph{crossed set} $(G,X, \partial)$ is the data of a group $G$ and a crossed $G$-set $(X, \partial)$.
\end{defi}

\begin{defi}\label{def_idempotent_XSet}
A crossed set $(G,X, \partial)$ is called \emph{idempotent} if it satisfies the additional axiom:
\[\forall x \in X,\ \partial(x) \cdot x = x.\]
\end{defi}

\begin{conv}\label{conv_idempotent_XSet}
The only crossed sets relevant to quandle theory are idempotent ones, so every crossed set is assumed idempotent in the sequel.
\end{conv}

Let $X$ be a crossed $G$-set (assumed idempotent). The following formula is easily seen to define a quandle law on $\qd$ on $X$:
\[\forall x,y \in X,\ x \qd y = \partial (x) \cdot y.\]

This construction is functorial with respect to the obvious morphisms:
\begin{defi}
Let $(G,X, \partial_X)$ and $(H,Y, \partial_Y)$ be crossed sets. A \emph{morphism} between them is a couple $(\varphi, f)$ where $\varphi: G \rightarrow H$ is a group morphism and $f : X \rightarrow Y$ is a $\varphi$-equivariant map (for all $g \in G$ and $x \in X$, $f(g \cdot x) = \varphi(g) \cdot f(x)$), such that $\varphi \circ \partial_X = \partial_Y \circ f$. Crossed sets and their morphisms form a category $X\mathcal Set$.
\end{defi}

The above construction $(G,X, \partial_X) \mapsto (X, \qd)$ is easily seen to define a functor $\mathcal U$ from $X\mathcal Set$ to $\mathcal Qdl$. It has a left adjoint $\mathcal F$, defined by:
\[\mathcal F (Q) := (G(Q), Q, \eta),\]
where $G(Q)$ acts on $Q$ via the adjoint action. The fact that this defines a left adjoint to~$\mathcal U$ is easily checked, using Lemmas~\ref{equivariance_of_eta} and~\ref{equivariance}, together with the universal property of the enveloping group. By construction, we have:
\begin{fact}\label{UF=id}
$\mathcal U \circ \mathcal F = id_{\mathcal Qdl}$.
\end{fact}

\subsubsection*{Split quandles and split crossed sets}

The constructions from~\ref{qdl_from_grp} require the choice of representatives of orbits of quandles. This can also be seen as a choice of section of the canonical projection $Q \twoheadrightarrow \pi_0(Q) = Q/G(Q)$ associated to a quandle $Q$. 
\begin{defi}\label{def_splitting}
A \emph{split quandle} $(Q, \sigma)$ (resp.~a \emph{split crossed set} $(G,X,\partial,\sigma)$) is a quandle $Q$ (resp.~a crossed set $(G,X,\partial)$) endowed with a splitting $\sigma : \pi_0(Q) \rightarrow Q$ (resp.~$\sigma : X/G \rightarrow X$) of the canonical projection $Q \twoheadrightarrow \pi_0(Q)$ (resp.~$X \twoheadrightarrow X/G$).
\end{defi}
In other words, a splitting of a quandle (resp.~of a crossed set) is a choice of representatives for its orbits. We want the choice of representatives of orbits to be a mere tool in getting a statement on quandles up to isomorphisms, so we define:
\begin{defi}\label{def_cat_split}
The \emph{category $\mathcal Qdl_s$ of split quandles} (resp.~\emph{$X\mathcal Set_s$ of split crossed sets}) is the category whose objects are split quandles (resp.~split crossed sets), and whose morphisms are morphisms of quandles (resp.~of crossed sets), with no requirement with respect to splittings.
\end{defi}
In particular, two split quandles $(Q, \sigma)$ and $(R, \tau)$ are isomorphic if and only if the quandles $Q$ and $R$ are.

\begin{rmq}\label{omega_is_an_eq}
This definition makes the obvious forgetful functor $\omega : \mathcal Qdl_s \rightarrow \mathcal Qdl$ (resp.~$\omega : X\mathcal Set_s \rightarrow X\mathcal Set$) fully faithful. It is also essentially surjective, so it is an equivalence of categories. However, in order to keep our constructions explicit, we will refrain from choosing a quasi-inverse. The alert reader will notice that such a choice (a choice of representatives of orbits for every quandle) would be comparable to choosing a fixed basis for every vector space. 
\end{rmq}

If we want to adapt the above construction $\mathcal U$ to this setting, we soon run into a problem: orbits of the quandle $\mathcal U(G,X, \partial)$ are not $G$-orbits, so a choice of splitting of $(G,X, \partial)$ will not give a choice of splitting for $\mathcal U(G,X, \partial)$.

\begin{ex}
Let $G$ be any group acting non-trivially on a set $X$. Let $\partial(x) = 1$ for every $x \in X$. Then $(G,X, \partial)$ is a crossed set, and $\mathcal U(G,X, \partial)$ is $X$ seen as a trivial quandle. The orbits of this quandle are points, whereas the $G$-orbits can have any shape.
\end{ex}

This problem could certainly be solved by choosing explicit quasi-inverses for the forgetful functors $\omega$ (see~Remark~\ref{omega_is_an_eq}), but a more explicit solution is given by restricting to a subcategory of well-behaved crossed sets:
\begin{defi}
Let $X\mathcal Set^*$ (resp.~$X\mathcal Set^*_s$) be the full subcategory of $X\mathcal Set$ (resp.~$X\mathcal Set_s$) consisting of crossed sets $(G,X,\partial)$ (resp.~split crossed sets $(G,X,\partial, \sigma)$) such that $\partial(X)$ generates $G$. 
\end{defi}

The inclusion $\iota : X\mathcal Set^* \hookrightarrow X\mathcal Set$ has a right adjoint $\Phi$ defined by replacing $G$ with the subgroup generated by $\partial(X)$:
\[\Phi(G,X,\partial) = (\langle \partial(X) \rangle, X, \partial).\] 
The same holds for split crossed sets. In other words, these subcategories are \emph{coreflexive}.

The functor $\mathcal F : \mathcal Qdl \rightarrow X\mathcal Set$ takes values in $X\mathcal Set^*$. Indeed, the enveloping group $G(Q)$ of a quandle $Q$ is generated by $\eta(Q)$. As a consequence, the adjunction $\mathcal F \dashv \mathcal U$ between $\mathcal Qdl$ and $X\mathcal Set$ restricts to an adjunction between $\mathcal Qdl$ and $X\mathcal Set^*$, still denoted by $\mathcal F \dashv \mathcal U$ (see the diagram~\ref{diag_adj} below). With these notations, we have $\iota \circ \mathcal F = \mathcal F$, and we also notice that $\mathcal U \circ \Phi = \mathcal U$ (replacing $G$ with $\langle \partial(X) \rangle$ does not change the underlying quandle of a crossed set $(G, X, \partial)$). By restricting $\mathcal U$ to $X\mathcal Set^*$, we make the aforementioned problem go away:
\begin{lem}
The adjunction $\mathcal F \dashv \mathcal U$ between $\mathcal Qdl$ and $X\mathcal Set^*$ lifts to an adjunction between $\mathcal Qdl_s$ and $X\mathcal Set^*_s$, still denoted by $\mathcal F \dashv \mathcal U$, by posing:
\[\begin{cases}
\mathcal U(G,X,\partial, \sigma) := (\mathcal U(G,X,\partial), \sigma) \\
\mathcal F(Q, \sigma) := (\mathcal F(Q), \sigma)
\end{cases}\]
\end{lem}

\begin{proof}
It is clear that a splitting of a quandle $Q$ is also a splitting of the crossed set $\mathcal F(Q)$, since both are choices of representatives for the same action of $G(Q)$. We need to check that a splitting $\sigma$ of a crossed set $(G,X,\partial)$ satisfying $G = \langle \partial(X) \rangle$ is a splitting of the underlying quandle $\mathcal U(G,X,\partial) = (X, \qd)$. Notice first that $G$ acts by automorphisms of $(X, \qd)$. Its image in $\Aut(X, \qd)$ is generated by the $\partial(x)\cdot (-) = x \qd (-)$, so it is exactly $\Inn(X, \qd)$. This implies that $G$-orbits of $X$ are exactly the orbits under the action of inner automorphisms, which are exactly the orbits of the quandle $(X, \qd)$, whence the result.
\end{proof}

We can summarize the above constructions in the following commutative diagram:
\begin{equation}\label{diag_adj}
\begin{tikzcd}
\mathcal Qdl_s \ar[r, "\omega"] \ar[d, shift left, "\mathcal F"]
&\mathcal Qdl \ar[d, shift left, "\mathcal F"] \ar[r, shift left, "\mathcal F"]
&X\mathcal Set \ar[dl, shift left, "\Phi"] \ar[l, shift left, "\mathcal U"] \\[1.5em]
X\mathcal Set^*_s \ar[r, "\omega"] \ar[u, shift left, "\mathcal U"]
&X\mathcal Set^*. \ar[u, shift left, "\mathcal U"] \ar[ur, hook, shift left, "\iota"] &
\end{tikzcd}
\end{equation}

\subsubsection*{Split crossed sets and peripheral systems}

Lemma~\ref{lem_qd_from_gp} prompts us to define:
\begin{defi}\label{def_Ob(Pph)}
A \emph{peripheral system} $\left(G, (H_i, z_i)_{i \in I}\right)$ is the data of a group $G$, a set $I$ and, for each $i \in I$, a subgroup $H_i$ of $G$ and an element $z_i$ in the center of~$H_i$.
\end{defi}

\begin{rmq}
The conditions of Lemma~\ref{lem_qd_from_gp} are slightly weaker than those required here. Namely, asking that $h z_i gH_j = z_i h gH_j$ (for $i,j \in I$, $h \in H_i$ and $g \in G$), is weaker than requiring that $h z_i = z_i h$ (for $i \in I$ and $h \in H_i$). The construction $\mathcal Q$ would still be well-defined and could still be made functorial on such a larger category of peripheral systems. However, our goal is to construct an equivalence between $\mathcal Qdl$ and a subcategory of peripheral systems, whose elements satisfy this requirement. Moreover, the additional requirement also allows us to use the correspondence with crossed sets described below. That is why we make it part of the definition.
\end{rmq}

From a peripheral system $(G, (H_i, z_i)_{i \in I})$, we can get a crossed set, by letting $G$ act canonically on $X := \bigsqcup G/H_i$, and defining $\partial: X \rightarrow G$ by $\partial(gH_i) := gz_ig^{-1}$, for $i \in I$ and $g \in G$. The latter is well defined because we assume that $H_i$ centralizes $z_i$ (so $gz_ig^{-1}$ depends only on the class $gH_i$). This defines a crossed set:
\[\forall g_1, g_2 \in G,\ \forall i \in I,\ 
\partial(g_1g_2H_i) = g_1g_2z_i(g_1g_2)^{-1} = g_1(g_2z_ig_2^{-1})g_1^{-1} = g_1\partial(g_2H_i)g_1^{-1}.\]
Moreover, it is idempotent because we ask that $z_i \in H_i$, hence:
\[\forall g \in G,\ \forall i \in I,\ \partial(gH_i)gH_i = gz_iH_i = gH_i.\]
Notice that this crossed set is endowed with a canonical splitting, in the sense of Definition~\ref{def_splitting}, that is, with a canonical choice of representatives for its orbits. Indeed, each orbit $G/H_i$ has the canonical representative $H_i$ (the class of $1$). This construction then defines a map $\mathcal X$ from peripheral systems to split crossed sets:
\[\mathcal X\left(G, (H_i, z_i)_{i \in I}\right) := 
\left(G,\ \bigsqcup G/H_i,\ \partial: gH_i \mapsto gz_ig^{-1},\ \sigma: i \mapsto H_i\right)\]

Conversely, given a split crossed set $(G,X,\partial,\sigma)$, we can define a peripheral system:
\[\mathcal S(G,X,\partial,\sigma) := 
\left(G, \bigr(\Stab(\sigma(t)), \partial(\sigma(t))\bigl)_{t \in X/G}\right).\]
The idempotence of $X$ says that for every $x \in X$, $\partial(x) \in \Stab(x)$, and if $g \in \Stab(x)$, then $g\partial(x)g^{-1} = \partial(g \cdot x) = \partial(x)$, so $\partial(x)$ is central in $\Stab(x)$, and this is indeed a peripheral system. 

Notice that we clearly have $\mathcal S \circ \mathcal X = id$. The other way around, $\mathcal X \circ \mathcal S$ sends a split crossed set $(G,X,\partial,\sigma)$ to the following split crossed set (where we write $S_t$ for $\Stab(\sigma(t))$ for short):
\[\left(G,\ X' = \bigsqcup\limits_{t \in X/G} G/S_t,\ \partial' : gS_t \mapsto g\partial(\sigma(t))g^{-1} = \partial(g \cdot \sigma(t)),\ \sigma': t \mapsto S_t\right).\]
There is a canonical $G$-equivariant bijection $\varphi$ between $X'$ and $X$ given by $gS_t \mapsto g\sigma(t)$. Since $\partial \circ \varphi = \partial'$ and $\varphi \circ \sigma' = \sigma$, this will be an isomorphism in any reasonable category of split crossed sets, so $\mathcal S$ and $\mathcal X$ can become mutually inverse equivalences of categories if we make the right choices. In particular, using Definition~\ref{def_cat_split} of a category of split crossed sets, a careful investigation then reveals what morphisms of peripheral systems should be in order for $\mathcal X$ and $\mathcal S$ to become functors:
\begin{defi}\label{def_Mor(Pph)}
A \emph{morphism} $(\varphi, u, [g]) : \left(G, (H_i, z_i)_{i \in I}\right) \rightarrow \left(G', (H_j', z_j')_{j \in J}\right)$ of peripheral systems is the data of a group morphism $\varphi : G \rightarrow G'$, a map $u: I \rightarrow J$, and the class $[g]$ of a map $g: I \rightarrow G'$ denoted by $i \mapsto g_i$, where $g_i$ is defined only up to right multiplication by $H_{u(i)}'$. This data has to satisfy, for all $i \in I$:
\[\begin{cases}
g_i^{-1}\varphi(H_i)g_i \subseteq H_{u(i)}', \\ 
g_i^{-1}\varphi(z_i)g_i = z_{u(i)}'.
\end{cases}\] 
\end{defi}

\begin{conv}
We will often write only $g$ for $[g]$, remembering that what is meant by $g_i$ is really $g_iH_{u(i)}'$.
\end{conv}

Composition of morphisms is given by the expected formula:
\[(\varphi', u', g') \circ (\varphi, u, g) = (\varphi' \circ \varphi, u' \circ u, i \mapsto \varphi'(g_i) g'_{u(i)}) .\]
Notice that $\varphi'(g_i) g'_{u(i)}$ is well-defined modulo right multiplication by $H''_{u'(u(i))}$, as it should be. Indeed, if we replace $g_i$ with $g_ih$ for some $h \in H_{u(i)}'$, then $\varphi'(g_i) g'_{u(i)}$ is replaced with $\varphi'(g_i) \varphi'(h) g'_{u(i)}$, and the first condition in Definition~\ref{def_Mor(Pph)} ensures that $\varphi'(h) g'_{u(i)}$ belongs to $g'_{u(i)}H''_{u'(u(i))}$.

\begin{defi}\label{def_Pph}
All the above defines a category $\mathcal Pph$ of peripheral systems.
\end{defi}

The constructions $\mathcal X$ and $\mathcal S$ do become functors between these categories, with the obvious action on morphisms. Precisely, if $ (\varphi, f) : (G, X, \partial, \sigma) \rightarrow (G', X', \partial', \sigma')$ is a morphism of (split) crossed sets, then $\mathcal S(\varphi, f) = (\varphi, \bar f, g)$, where $\bar f : X/G \rightarrow X'/G'$ is induced by $f$, and for $t \in X/G$, the element $g(t)$ is defined (up to right multiplication by $\Stab(\sigma(\bar f(t)))$ on the right) by $f(\sigma(t)) = g(t)\sigma(\bar f(t))$. Conversely, if $(\varphi, u, g)$ is a morphism of peripheral systems, we pose $\mathcal X(\varphi, u, g) := (\varphi,\ f: gH_i \mapsto \varphi(g)g_iH'_{u(i)})$. It is then a matter of easy verification to prove:
\begin{prop}\label{XS=SX=id}
The functors $\mathcal X$ and $\mathcal S$ are mutually inverse equivalences of categories between $X\mathcal Set_s$ and $\mathcal Pph$.
\end{prop}

This equivalence induces an equivalence between the full subcategory $X\mathcal Set_s^*$ of $X\mathcal Set_s$ and a full subcategory $\mathcal Pph^*$ of $\mathcal Pph$. A straightforward investigation shows that $\mathcal Pph^* = \mathcal S(X\mathcal Set_s^*)$ consists of peripheral systems $(G, (H_i,z_i)_{i \in I})$ such that $G$ is generated by the $g z_i g^{-1}$ (for $i \in I$ and $g \in G$). In other words, $(G, (H_i,z_i)_{i \in I}) \in \mathcal Pph^*$ if and only if $G$ is normally generated by the $z_i$ (for $i \in I$).

\begin{rmq}\label{def_equiv}
An isomorphism in the category of peripheral systems is a morphism $(\varphi, u, g)$ where $\varphi$ is a group isomorphism, $u$ is bijective and the inclusions from Definition~\ref{def_Mor(Pph)} are equalities. In other words, up to re-indexation (through $u$), the isomorphism $\varphi$ sends $H_i$ and $z_i$ to conjugates of $H_i'$ and $z_i'$ by a common element $g_i$ (for each $i \in I$). This recovers the usual notion of an \emph{equivalence} of peripheral systems, dating back to \cite[Th.~2]{Milnor-Link_Groups} (with $z_i = \alpha_i$  and $H_i = \langle \alpha_i, \beta_i \rangle$). In particular, the data of a peripheral system $P$ up to equivalence is the same as the data of the crossed set $\mathcal X(P)$ up to isomorphism.
\end{rmq}

\subsubsection*{Quandles and peripheral systems}

We now come to the functorial version of Joyce's correspondence between quandles and peripheral systems. Let us begin with the following diagram:
\begin{equation}
\begin{tikzcd}
\mathcal Pph \supset \mathcal Pph^* \ar[r, shift right, swap, "\mathcal X"]
&X\mathcal Set_s^* \ar[l, shift right, swap, "\mathcal S"] \ar[r, shift right, swap, "\mathcal U"]
&\mathcal Qdl_s \ar[r, "\omega"] \ar[l, shift right, swap, "\mathcal F"]
&\mathcal Qdl
\end{tikzcd}
\end{equation}
We define $\mathcal Q := \mathcal U \circ \mathcal X$ and $\mathcal P := \mathcal S \circ \mathcal F$. One can check that $\mathcal Q(G, (H_i, z_i)_{i \in I})$ is indeed the quandle defined in Lemma~\ref{lem_qd_from_gp}, endowed with the representatives $H_i$ of its orbits $G/H_i$ (which are indeed its orbits if $(G, (H_i, z_i)_{i \in I}) \in \mathcal Pph^*$). The other way around, one can check that if $(Q, \sigma)$ is a split quandle, then:
\[\mathcal P(Q, \sigma) = \left(G(Q), (\Stab(\sigma(t)), \eta(\sigma(t)))_{t \in \pi_0(Q)} \right).\]

\begin{rmq}
By composition of the adjunctions $\mathcal F \dashv \mathcal U$ and $\mathcal X \dashv \mathcal S$ (mutually inverse equivalences of categories being adjoint to each other on both sides), we see that $\mathcal P$ is left adjoint to $\mathcal Q$.
\end{rmq}

We recover Joyce's result in this framework:
\begin{prop}\label{QP=1}
There is a canonical isomorphism $\mathcal Q \circ \mathcal P \cong id_{\mathcal Qdl_s}$.
\end{prop}

\begin{proof}
From Prop~\ref{XS=SX=id} and Fact~\ref{UF=id}, we get $\mathcal Q \mathcal P = \mathcal U (\mathcal X \mathcal S) \mathcal F \cong \mathcal U \mathcal F = id$.
\end{proof}

\begin{rmq}
If one takes a closer look at the isomorphism featuring in Proposition~\ref{QP=1}, one can see that it is in fact the one from Example~\ref{Q_from_G(Q)}, for $q_i = \sigma(i)$ (indexed by $I = \pi_0(Q)$). Thus, Proposition~\ref{QP=1} is indeed a categorical enhancement of the isomorphism of Example~\ref{Q_from_G(Q)}.
\end{rmq}

Recall (Remark~\ref{def_equiv}) that isomorphisms in $\mathcal Pph$ recover the classical notion of \emph{equivalence} of peripheral systems. Thus Proposition~\ref{QP=1} really is a categorical enhancement of the classical correspondence~\cite{Joyce}:

\begin{cor}\label{Qdl_vs_Pph}
The data of a quandle $Q$ up to isomorphism is the same as the data of the peripheral system $\mathcal P(Q)$ up to equivalence.
\end{cor}

\begin{rmq}
This Corollary can be obtained directly with a more elementary point of view (which was the one in~\cite{Joyce}). Namely, once we have the isomorphism $Q \cong \mathcal Q(\mathcal P(Q))$ from Example~\ref{Q_from_G(Q)}, it is easy to see that isomorphisms of quandles must correspond to equivalences of peripheral systems. In fact, our introduction of morphisms of peripheral systems is just a generalisation of this reasoning.
\end{rmq}

\section{Enveloping groups}\label{sec_G(Q)}

We now use the construction from~\cref{sec_construction} to study enveloping groups of quandles: in~\cref{par_env_group}, we get a simple presentation of the enveloping group of a quandle generated by representatives of its orbits, and in \cref{par_injectivity}, we get a characterisation of injective quandles in terms of subgroups of their enveloping groups.

\subsection{A presentation of enveloping groups}\label{par_env_group}

Let $Q$ be a quandle. We describe a presentation of $G(Q)$ under the following:
\begin{hypo}
Suppose that we can choose a generating family $(q_s)_{s \in S}$ such that $s \mapsto q_s$ induces a bijection $S \cong \pi_0(Q)$. In other words, assume that we have a generating set with exactly one element in each orbit.
\end{hypo} 

\begin{rmq}
Clearly, a generating set must have at least one element in each orbit. But several elements per orbit may be required. For instance, the only connected quandles satisfying our assumption are $\varnothing$ and $\{*\}$. This shows that it is quite strong a requirement. However, it is satisfied by nilpotent quandles, by Corollary \ref{generators_of_nilp_qdles}. Moreover, it also holds for some non-nilpotent quandles, such as free quandles.
\end{rmq}

Under this assumption, the map $x_s \mapsto q_s \qd(-)$ extends to a group morphism $F[S] \rightarrow \Inn(Q)$, which is an action of the free group $F[S]$ on $Q$ satisfying the hypotheses of Proposition \ref{prop_qd_from_gp} (see Remark~\ref{hyp_G_contains_Inn(Q)}). As a consequence, we get:
\[Q \cong \mathcal Q(F[S], (H_s, x_s)),\]
where $H_s$ is the stabilizer of $q_s$ with respect to the above action of $F[S]$ on $Q$.

\medskip

It turns out that we can get a presentation of the enveloping group of $Q$ from this description of $Q$:
\begin{theo}\label{presentation_of_G(Q)}
Let $Q$ be a quandle having a generating family $(q_s)_{s \in S}$ with exactly one element in each orbit. Then:
\[G(Q) \cong \Bigl\langle (x_s)_{s \in S} \Bigm| [H_s,x_s] = 1 \Bigr\rangle,\]
where $H_s$ is the stabilizer of $q_s$ with respect to the action of the free group $F[S]$ on $Q$ making $x_s$ act via $q_s \qd (-)$.
\end{theo}

We are going to show that the group defined by the presentation of the statement has the universal property of $G(Q)$, using the following characterisation of morphisms from $Q$:
\begin{lem}\label{morphisms_from_Q}
Let $T$ be any quandle and let $(t_s)_{s \in S}$ be a family of elements of $T$. Consider the action of the free group $F[S]$ on $T$ making $x_s$ act via $t_s \qd(-)$. Then, with the above notation, the map $q_s \mapsto t_s$ extends (uniquely) to a quandle morphism $f:Q \rightarrow T$ if and only if for all $s \in S$, every element of $H_s$ acts trivially on $t_s$.
\end{lem}

\begin{proof}
Suppose first that such an extension exists. Since the $q_s$ generate $Q$, it is unique. In fact, we must have $f(w \cdot q_s) = w \cdot t_s$ for all $w \in F[S]$ and all $s \in S$. Indeed, if $w = x_{s_1}^{n_1} \cdots x_{s_l}^{n_l}$, then the quandle morphism $f$ must send $w \cdot q_s = q_{s_1} \qdn{n_1} \cdots q_{s_l} \qdn{n_l} q_s$ to $t_{s_1} \qdn{n_1} \cdots t_{s_l} \qdn{n_l} t_s = w \cdot t_s$. Thus, if $w \cdot q_s = q_s$ (that is, if $w \in H_s$), we must have $w \cdot t_s = f(w \cdot q_s) = f(q_s) = t_s$, whence the necessity of the condition.

Suppose now that for all $s \in S$, $H_s$ acts trivially on $t_s$, and let us show that $w \cdot q_s \mapsto w \cdot t_s$ defines a morphism from $f:Q \rightarrow T$. First, we need to check that $f$ is well-defined. Here we use our hypothesis on $Q$: every element of $Q$ can be written as $w \cdot q_s$ for a \emph{unique} $s \in S$. Thus, $w \cdot q_s = w' \cdot q_{s'}$ implies $s = s'$, and $q_s = w^{-1}w' \cdot q_s$, which means that $w^{-1}w' \in H_s$, hence $t_s = w^{-1}w' \cdot t_s$, that is $w \cdot t_s = w' \cdot t_s$. Finally, we need to check that $f$ is a morphism, which is easy. Indeed, for all $v,w \in F[S]$, and all $r,s \in S$, we have:
\begin{align*}
f((w \cdot q_s) \qd (v \cdot q_r)) 
&= f(w \cdot (q_s \qd ( w^{-1}v \cdot q_r))) = f(w x_s w^{-1} v \cdot q_r) \\
&= w x_s w^{-1} v \cdot t_r = w \cdot (t_s \qd (w^{-1}v \cdot t_r)) \\
&= (w \cdot t_s) \qd (v \cdot t_r) = f(w \cdot q_s) \qd f(v \cdot q_r),
\end{align*}
where we have used the definition of $f$, the fact that $w$ acts as a quandle morphism on both $Q$ and $T$, and the fact that $x_s$ acts via $q_s \qd(-)$ on $Q$ and via $t_s \qd(-)$ on $T$.
\end{proof}

\begin{proof}[Proof of Theorem~\ref{presentation_of_G(Q)}]
Let $G_0$ be the group defined by the presentation of the theorem. Remark that if $w \in H_s$ and $s \in S$, then the relation $[w,x_s] = 1$ can be written as $w x_s w^{-1} = x_s$. Thus the family $(x_s)_{s \in S}$ in the conjugation quandle of $G_0$ satisfies the hypothesis of Lemma \ref{morphisms_from_Q} (the corresponding action of $F[S]$ on $G_0$ is induced by the conjugation action of $F[S]$ on itself). As a consequence, $q_s \mapsto x_s$ defines a quandle morphism from $Q$ to $G_0$.

Now, let us show that this morphism is the universal one from $Q$ to a group. Let $G$ be a group and $f: Q \rightarrow G$ be a quandle morphism. Consider the group morphism $\tilde f : F[S] \rightarrow G$ sending $x_s$ to $t_s := f(q_s)$. If $w \in H_s$, then $\tilde f(w x_s w^{-1}) = \tilde f(w) t_s \tilde f(w)^{-1}$ equals $w \cdot t_s$. Indeed, if $c$ denotes the morphism $g \mapsto g(-)g^{-1}$, from $G$ to $\Aut(G)$, then $c \circ \tilde f$ is the unique morphism from $F[S]$ to $\Aut(G)$ sending each $s$ to conjugation by $t_s$. By definition, this is the action of $F[S]$ on $G$ associated to the $t_s$, which then makes $w$ act by $c \circ \tilde f(w)$. Using Lemma~\ref{morphisms_from_Q}, we see that $w \cdot t_s = t_s = f(x_s)$. Thus $\tilde f$ vanishes on the relations defining $G_0$, which means that it factors to a morphism $\overline f : G_0 \rightarrow G$. Since $\overline f(x_s) = t_s = f(x_s)$, the morphism $\overline f$ is the required (unique) extension of $f$.
\end{proof}

\begin{rmq}
In fact, most of the above proof consists in showing that $Q$ has the right presentation, which can then be turned automatically into a presentation of $G(Q)$. Namely, we have shown that $Q$ has the presentation 
\[ Q = \langle S\ |\ w \cdot x_s = x_s \text{ for } s \in S \text{ and } w \in H_s \rangle.\] Notice that this makes sense as a quandle presentation, by considering the free quandle $Q[S]$ inside the free group $F[S]$, where the element $w \cdot x_s = w x_s w^{-1}$ is an element of $Q[S]$. With this point of view, Lemma~\ref{morphisms_from_Q} says exactly that $Q$ has the universal property of this presentation. 
\end{rmq}

\subsection{An injectivity criterion}\label{par_injectivity}

Let $Q$ be a quandle, and let us consider the canonical morphism $\eta: Q \rightarrow G(Q)$ to its enveloping group. If $q \in Q$, we can consider its stabilizer $\Stab(q) \leq G(Q)$. A closely related subgroup is the centralizer $C(\eta(q)) \leq G(Q)$ of $\eta(q)$. In fact, since $\eta$ is $G(Q)$-equivariant (with respect to the conjugation action on the target -- see Lemma~\ref{equivariance_of_eta}), if $g \cdot q = q$, then $g \eta(q) g^{-1} = \eta(g \cdot q) = \eta(q)$. Thus $\Stab(q) \subseteq C(\eta(q))$. However, the converse inclusion is not true in general: from $\eta(g \cdot q) = \eta(q)$, we cannot deduce that $g \cdot q = q$, except when $\eta^{-1}(\eta(q)) = \{q\}$. In particular, when $Q$ is injective (that is, when $\eta$ is injective), then for all $q \in Q$, we have $\Stab(q) = C(\eta(q))$. We show that the converse holds:

\begin{prop}\label{injectivity_criterion}
Let $Q$ be a quandle, and let us choose one element $q_i$ in each of its orbits (where $i \in I \cong \pi_0(Q)$). Then $Q$ is injective if and only if for all $i \in I$, the inclusion $\Stab(q_i) \subseteq C(\eta(q_i))$ is an equality.
\end{prop}

\begin{proof}
We apply Proposition \ref{prop_qd_from_gp} to the quandle $Q$ and to its image $\eta(Q)$, with respect to the action of the group $G(Q)$ on both of them. By definition of the $q_i$, there is one in each $G(Q)$-orbit, and $G(Q)$ surjects onto $\Inn(Q)$ by construction; precisely, the element $\eta(q_i)$ of $G(Q)$ acts via the inner automorphism $q_i \qd(-)$. Proposition \ref{prop_qd_from_gp} then gives (as already stated in Example~\ref{Q_from_G(Q)}): 
\[Q \cong \mathcal Q\left(G(Q), (\Stab(q_i), \eta(q_i))_{i \in I} \right).\] 
There is also one $\eta(q_i)$ (for a unique $i \in I$) in each $G(Q)$-orbit of $\eta(Q)$. Indeed, there is at least one, since $\eta(Q)$ is a $G(Q)$-equivariant quotient of $Q$, and if we have $g \eta(q_i) g^{-1} = \eta(q_j)$ for some $g \in G(Q)$ and some $i,j \in I$, we can project this relation in $G(Q)^{ab} \cong \Z [\pi_0(Q)]$, and we get $\overline q_i = \overline q_j$ in there, hence $i = j$. Moreover, the inner automorphism of $\eta(Q)$ associated to $\eta(q)$ (for any $q \in Q$) is precisely given by the action of $\eta(q) \in G(Q)$ by conjugation. Proposition \ref{prop_qd_from_gp} then gives:
\[\eta(Q) \cong \mathcal Q\left(G(Q), (C(\eta(q_i)), \eta(q_i))_{i \in I} \right).\]
Now, the morphism $\eta: Q \twoheadrightarrow \eta(Q)$ identifies as the morphism 
\[\mathcal Q\left(G(Q), (\Stab(q_i), \eta(q_i))_{i \in I} \right) \longrightarrow \mathcal Q\left(G(Q), (C(\eta(q_i)), \eta(q_i))_{i \in I} \right)\] 
induced by $id_{G(Q)}$. This is the disjoint union of the canonical maps $G(Q)/\Stab(q_i) \twoheadrightarrow G(Q)/C(\eta(q_i))$ induced by the inclusions $\Stab(q_i) \subseteq C(\eta(q_i))$, and the former are bijections if and only if the latter are equalities.
\end{proof}

\begin{rmq}
To show that $Q$ is injective, one needs to check the condition $\Stab(q) = C(\eta(q))$ for at least one $q$ in each orbit (and then it holds for all $q \in Q$). In particular, it is enough to check it for $q$ in a given generating subset $S$ of $Q$.
\end{rmq}


\section{\texorpdfstring{$2$}{2}-nilpotent quandles}\label{section_2-nilp}

In this section, we apply the results of the previous section to classify $2$-nilpotent quandles, and to compute their enveloping groups. $2$-nilpotent quandles are the ones called \emph{abelian} in \cite{Lebed-Mortier}, and we recover their classification results with a slightly different point of view, allowing us to extend it to all finitely generated $2$-nilpotent quandles, not only finite ones. In particular, we recover the calculations of \cite[Prop.~2.38]{Eisermann}.

\subsection{Classification}

By definition, a quandle $Q$ is $2$-nilpotent exactly when its group of inner automorphisms is abelian. That $\Inn(Q)$ is abelian means exactly that for all $x, y \in Q$, the automorphisms $x \qd (-)$ and $y \qd (-)$ commute, that is:
\[\forall x,y,z \in Q,\ \ x \qd (y \qd z) = y \qd (x \qd z).\]

Because of Lemma \ref{Orbits_under_Z(I)} (or directly from the above, putting $t = y \qd z$ to get $x \qd t = (y \qd x) \qd t)$), this implies that the behaviour of an element depends only on its orbit under $\Inn(Q)$ (or $G(Q)$), that is, on its class in $\pi_0(Q)$. In other words, the action of $Q$ on itself factors through $\pi_0(Q)$. This means that the action of $G(Q)$ on $Q$ factors through the enveloping group of the trivial quandle $\pi_0(Q)$, which is $\Z[\pi_0(Q)]$. 

\begin{rmq}
Each orbit of a $2$-nilpotent quandle $Q$ must be a trivial subquandle:
\[\forall x \in Q,\ \forall y \in G(Q) \cdot x,\ \ y \qd x = x \qd x = x.\]
\end{rmq}

In order to understand $2$-nilpotent quandles, we can apply the constructions of \cref{qdl_from_grp}, with acting group $G = \Z[\pi_0(Q)]$. In fact, we can start with any free abelian group $G = \Z^{(I)}$. Since this $G$ is abelian, the only relevant condition from Lemma~\ref{lem_qd_from_gp} is the second one. As an application of Proposition~\ref{prop_qd_from_gp} and Lemma~\ref{lem_qd_from_gp}, we get the following:
\begin{theo}\label{classification_2-nilp}
Let us denote by $(e_i)_{i \in I}$ the canonical basis of the free abelian group $\Z^{(I)}$.  For each $i \in I$, let $H_i$  be a subgroup of $\Z^{(I)}$ containing $e_i$.
Let us define:
\[\mathcal Q( (H_i)_{i \in I} ) := \bigsqcup\limits_i \Z^{(I)}/H_i.\]  
We get a $2$-nilpotent quandle structure on this set via:
\[\forall i,j,\ \forall x \in \Z^{(I)}/H_i,\ \forall y \in \Z^{(I)}/H_j,\ \ x \qd y := e_i  + y \in  \Z^{(I)}/H_j.\]
Every $2$-nilpotent quandle can be obtained in this way. Moreover, there is an isomorphism $\mathcal Q( (H_i)_{i \in I} ) \cong \mathcal Q( (K_j)_{j \in J} )$  if and only if there is a bijection $\sigma : I \cong J$ such that the corresponding isomorphism $\Z^{(I)} \cong \Z^{(J)}$ sends $H_i$ to $K_{\sigma(i)}$.
\end{theo}

\begin{proof}
The only part of the statement which still needs a proof is the last one. It is clear that such a $\sigma$ induces an isomorphism between $Q_1 = \mathcal Q( (H_i)_{i \in I} )$ and $Q_2 = \mathcal Q( (K_j)_{j \in J} )$. Conversely, remark that we can recover $I$ from $\mathcal Q( (H_i)_{i \in I} )$ as its set of orbits, and we can recover $H_i$ as the stabilizer of any element of the $i$-th orbit. Thus, an isomorphism $f: Q_1 \cong Q_2$ induces a bijection $\sigma := \pi_0(f)$ between the sets of orbits, and the induced isomorphism $\Z[\pi_0(f)]$ (which is induced by $f_*: G(Q_1) \rightarrow G(Q_2)$) must send the stabilizer of an element $q \in Q_1$ to the stabilizer of $f(q)$, because of Lemma~\ref{equivariance}.
\end{proof}

\begin{ex}\label{ex_G(subqdl)_not_ab_bis}
The quandle constructed in Example~\ref{ex_G(subqdl)_not_ab} is $2$-nilpotent, with three orbits. It corresponds to $H_1 = \langle e_1, n e_2, e_2 - e_3 \rangle$, $H_2 = \langle n e_1, e_2, e_1 - e_3 \rangle$ and $H_3 = \Z^3$.
\end{ex}

\subsection{Enveloping groups}

We can apply Theorem~\ref{presentation_of_G(Q)} to compute the enveloping group of a $2$-nilpotent quandle~$Q$. Indeed, since $Q$ is nilpotent, it satisfies the required hypothesis (see Corollary~\ref{generators_of_nilp_qdles}): if we choose one element $q_i$ in each orbit, these generate $Q$. We know that $G(Q)$ is $2$-nilpotent (Corollary~\ref{Q_nilp_ssi_G(Q)_nilp}), generated by the images of the $q_i$ so it is a quotient of the free $2$-nilpotent group $F_2[I]$ on $I$. Precisely, it is the quotient of $F_2[I]$ by the image of the relators in $F[I]$ from Theorem~\ref{presentation_of_G(Q)}.

The group $F_2[I]$ can be seen canonically as a central extension:
\[\Lambda^2(\Z^{(I)}) \hookrightarrow F_2[I] \twoheadrightarrow \Z^{(I)}.\]
The quotient $\Z^{(I)}$ is the abelianisation of $F_2[I]$, which is free abelian on the classes $e_i$ of the generators $x_i$, and the kernel $\Lambda^2(\Z^{(I)})$ identifies with $\Gamma_2(F_2[I])$ via $[x_i,x_j] \mapsto e_i \wedge e_j$.

\begin{theo}\label{presentation_of_G(2-nilp)}
Let $Q$ be a $2$-nilpotent quandle, and let $(q_i)_{i \in I}$ be a family of representatives of the orbits of $Q$ (where $I = \pi_0(Q)$). Then $G(Q)$ is the quotient of $F_2[I]$ by the following relations belonging to $\Gamma_2(F_2[I]) \cong \Lambda^2(\Z^{(I)})$:
\[e_i \wedge H_i = 0,\]
where $H_i$ is the stabilizer of $q_i$ with respect to the action of $\Z^{(I)}$ on $Q$ making $e_i$ act via $q_i \qd (-)$.
\end{theo}

\begin{proof}
This is a direct application of Theorem~\ref{presentation_of_G(Q)}, with the simplification discussed above. Namely, the $H_i$ of our current statement are the images of the $H_i$ from Theorem~\ref{presentation_of_G(Q)} by the projection $F[I] \twoheadrightarrow \Z^{(I)}$, and the $e_i \wedge H_i$ are the images of the relators from Theorem~\ref{presentation_of_G(Q)} in the free $2$-nilpotent group, modulo the identification of $\Gamma_2(F_2[I])$ with $\Lambda^2(\Z^{(I)})$.
\end{proof}

\begin{rmq}\label{rq_3-nilp_with_G(Q)_2-nilp}
The theorem also applies, with the same proof, to $3$-nilpotent quandles whose enveloping group happens to be $2$-nilpotent.
\end{rmq}

\begin{rmq}\label{finite_presentation}
If $I$ is finite (that is, if $Q$ is finitely generated), then we get a finite presentation of $G(Q)$ as a $2$-nilpotent group by replacing each $H_i$ (which is free abelian of finite rank) by a basis. We will do so below in our classification of finitely generated $2$-nilpotent quandles.
\end{rmq}

\begin{rmq}
If $G$ is a $2$-nilpotent group, a necessary condition for it to be the enveloping group of a nilpotent quandle is for its abelianisation to be free. However, this condition is not sufficient. For instance, if $|I| = 4$, the element $e_1 \wedge e_2 - e_3 \wedge e_4$ is easily seen not to be of the form $u \wedge v$, so the subgroup of $\Lambda^2(\Z^4)$ that it generates is not of the form $\sum \epsilon_i \wedge H_i$ for any choice of basis $(\epsilon_i)$ of $\Z^4$. As a consequence, $G := F_2[4]/([x_1, x_2] = [x_3, x_4])$ is not the enveloping group of a $2$-nilpotent quandle. In fact, it is not the enveloping group of any quandle. Indeed, if it were of the form $G(Q)$, $Q$ would need to be nilpotent of class at most $3$ (Corollary~\ref{Q_nilp_ssi_G(Q)_nilp}), and using Remark~\ref{rq_3-nilp_with_G(Q)_2-nilp}, $G$ would have to be of the form described in the theorem, and we just remarked that it is not.
\end{rmq}



\subsection{Applications}\label{Applications_classification}

We now apply our constructions to classify $2$-nilpotent quandles, and to compute their enveloping groups. We notably recover the classification results of \cite[\S 2-5]{Lebed-Mortier} and we extend them to quandles that are not finite but only finitely generated.

\subsubsection*{2-nilpotent quandles with two orbits}\label{par_2-orbits_2-nilp}

Let us apply Theorem~\ref{classification_2-nilp} with $|I| = 2$. The subgroups $H_1$ and $H_2$ of $\Z^2$ must be of the form $H_1 = \Z e_1 + m \Z e_2$ and $H_2 = \Z e_2 + n \Z e_1$, with $m,n \geq 0$. Hence, our two orbits are $\Z^2/H_1 \cong \Z/m$ and $\Z^2/H_2 \cong \Z/n$. Moreover, the quandle law that we get on $(\Z/m) \sqcup (\Z/n)$ is given by:
\[\forall x,x' \in \Z/m,\ \forall y, y' \in \Z/n,\ \ \ \ 
\begin{cases}
x \qd y = y+1; \\
y \qd x = x+1; 
\end{cases}\ 
\begin{cases}
x \qd x' = x'; \\
y \qd y' = y'.
\end{cases}
\]
This is exactly the quandle $Q_{m,n}$ defined in~\cite[Ex.~1.6]{Eisermann}. This should not come as a surprise: $2$-nilpotent quandles are coverings of trivial ones (Th.~\ref{nilpotency_via_coverings}), and the ones with two orbits are coverings of the trivial quandle with two elements. Remark that among them, the free $2$-nilpotent quandle on two generators, which is the one with the least number of relations, is the universal covering $Q_{0,0}$ of the trivial quandle with two elements.

We can use Theorem~\ref{presentation_of_G(2-nilp)} to recover \cite[Prop.~2.38]{Eisermann}. Namely, the free $2$-nilpotent group on $2$ generators $F_2[2]$ is the Heisenberg group:
\[\mathcal H = \left\langle\ x_1, x_2, z \ \middle| \ [x_1, x_2] = z,\ [z, x_1] = [z, x_2] = 1\ \right\rangle,\] which is a central extension of $\Z^2$ (generated by $e_1, e_2$) by $\Z = \Lambda^2(\Z^2)$ (generated by $e_1 \wedge e_2$). The group $G(Q_{m,n})$ is the quotient of $F_2[2]$ by the relators $e_1 \wedge (m e_2)$ and $(n e_1) \wedge e_2$, so that for $d := \gcd(m,n)$:
\[G(Q_{m,n}) = \mathcal H / z^d,\]
which is a central extension of $\Z^2$ by $\Z/d$. In particular, we recover the fact that the enveloping group of the free $2$-nilpotent quandle $Q_{0,0}$ is the free $2$-nilpotent group $\mathcal H$. We also remark that if $m$ and $n$ are coprime, then the enveloping group is $\Z^2$, which is abelian, even if $Q_{m,n}$ is not trivial. In fact, $G(Q_{m,n})$ is abelian if and only if $m$ and $n$ are coprime. Finally, we remark that the subquandle of $\mathcal H / z^d$  generated by $x_1$ and $x_2$ is isomorphic to $Q_{d,d}$, so the $Q_{d,d}$ are the only injective ones in this family, and if $d = \gcd(m,n)$, the image of $Q_{m,n}$ inside $G(Q_{m,n}) = \mathcal H / z^d$ is the quotient $Q_{d,d}$ of $Q_{m,n}$.

\subsubsection*{Finitely generated 2-nilpotent quandles}

Finite $2$-nilpotent quandles were classified in \cite[\S 2-5]{Lebed-Mortier} (where $2$-nilpotent quandles are called \emph{abelian}). Thus, we should recover their results as a particular case of ours. It is indeed the case, and we now explain the precise link between our construction and theirs. In particular, we remark that their construction applies with very little change to the classification of \emph{finitely generated} (but not necessarily finite) $2$-nilpotent quandles.

\medskip
 
Finitely generated $2$-nilpotent quandles are the $2$-nilpotent quandles with a finite number of orbits, say $n$ (see Corollary~\ref{generators_of_nilp_qdles}). Theorem~\ref{classification_2-nilp} (with $|I| = n$) classifies them: they are constructed from $n$-tuples $(H_1, ..., H_n)$ of subgroups of $\Z^n$, with each $H_i$ containing the $i$-th element $e_i$ of the canonical basis of $\Z^n$. If our quandle $Q$ is fixed, the $H_i$ are defined up to a permutation of the basis, which corresponds to a permutation of the orbits of $Q$. As in \cite[\S 2-5]{Lebed-Mortier}, we can fix a total ordering of the orbits of $Q$, so that the $H_i$ associated to $Q$ are uniquely defined: the $i$-th orbit of $Q$ is $\Z^n/H_i$, which is denoted by $G(M^{(i)})$ in \cite{Lebed-Mortier}.

\smallskip

If we want to do computations, we need a precise way of describing the $H_i$. This is the role played by \emph{parameters} from~\cite{Lebed-Mortier}. Before introducing them, let us remark that $H_i$ must contain $e_i$, so we only need to describe its image $\overline H_i$ in $\Z^n/e_i \cong \Z^{n-1}$. Now, we can filter $\Z^n/e_i$ by 
\[V_0 := \{0\} \subset V_1 := \langle e_{i+1} \rangle \subset V_2 := \langle e_{i+1}, e_{i+2} \rangle \subset V_3 := \langle e_{i+1}, e_{i+2}, e_{i+3} \rangle \subset \cdots\]
where the indices are to be understood modulo $n$. Then we can describe $H_i$ by parameters $m_k$ (for $1 \leq k  \leq n-1$) and $m_{k,l}$ (for $1 \leq l < k$) as follows:
\[\begin{cases}
\overline H_i \cap V_1 = \langle m_1 e_{i+1} \rangle; \\
\overline H_i \cap V_2 = \langle m_1 e_{i+1}, m_2 e_{i+2} + m_{2,1} e_{i+1} \rangle; \\
\overline H_i \cap V_3 = \langle m_1 e_{i+1}, m_2 e_{i+2} + m_{2,1} e_{i+1}, m_3 e_{i+3} + m_{3,2} e_{i+2} + m_{3,1} e_{i+1}  \rangle; \\\text{etc.}
\end{cases}\]
At the $k$-th step, we consider the image of $\overline H_i \cap V_k$ in $V_k/V_{k-1} \cong \Z \overline e_k$. It is generated by $m_k \overline e_k$, which we lift to an element of $\Z^{n-1}$. We then add this lift to the list of generators of $\overline H_i \cap V_{k-1}$ to get a list of generators of $\overline H_i \cap V_k$. If we take the convention that $0$ is always lifted to $0$, the list of generators is in fact a basis, with possibly zeros added. This holds at each step, and in particular for $\overline H_i \cap V_{n-1} = \overline H_i$. Remark that $m_k \neq 0$ for all $k$ if and only if $H_i$ has rank $n$, which is equivalent to the orbit $\Z^n/H_i$ being finite. This explains why the $m_k$ from \cite{Lebed-Mortier} were never trivial: they are concerned with finite quandles there.

In order for the above parameters to be uniquely determined by our data, we need to impose some conditions. Namely, we take $m_k \geq 0$, and we require that $0 \leq m_{k,l} < m_l$ if $m_l \neq 0$, and $m_{k,l} = 0$ if $m_k = 0$. This determines the choice of lift at each stage of the process; the fact that we can take a unique lift satisfying these conditions comes from the fact that we can add multiples of the previous generators to a given lift to perform Euclid's algorithm on each coordinate. The last condition did not appear in \cite{Lebed-Mortier}, again because the $m_k$ are not zero if $Q$ is finite (and not merely finitely generated).

The above correspondence explains why our $\Z^n/H_i$ are the $G(M^{(i)})$ from \cite{Lebed-Mortier}, their $x_j \in G(M^{(i)})$, also denoted by $x_j^{(i)}$, corresponding to the class of $e_{i + j}$ in $\Z^n/H_i$. We also see that the quandle law from Theorem~\ref{classification_2-nilp} coincides with theirs. Namely, if $a$ is in the $i$-th orbit and $b$ in the $j$-th one, $b \qd a$ (denoted by $a \rqd b$ in 
\cite{Lebed-Mortier}, where a convention opposite to ours is used for the quandle law) is $a + e_j$, and $e_j \in \Z^n/H_i$ is~$x_{j-i}^{(i)}$.

\medskip

Now that we have explained why our classification results are the same as theirs for finite quandles, and how it needs to be adapted in order to include all the finitely generated $2$-nilpotent quandles, let us turn to the computation of enveloping groups (called \emph{structure groups} in \cite{Lebed-Mortier}). Namely, \cite[Th.~3.2]{Lebed-Mortier} can be seen as another formulation of our Theorem~\ref{presentation_of_G(2-nilp)}. Precisely, it follows from Theorem~\ref{presentation_of_G(2-nilp)} that $G(Q)$ is a central extension of its abelianisation $\Z^n = \Z \pi_0(Q)$ by $\Lambda^2(\Z^n)$ quotiented by the $e_i \wedge H_i$. The kernel of this extension, which is also the commutator group of $G(Q)$, is thus generated by the elements $g_{ij} := e_i \wedge e_j$. The element $g_{ij}$ can be obtained by taking the commutator of any element of an image of the $i$-th orbit with an element in the image of the $j$-th orbit; the fact that this does not depend on the choices made is a particular instance of the theory of Lie rings associated to lower central series of groups. Relations defining the quotient are the ones defining $\Lambda^2(\Z^n)$ (namely, the $g_{ij}$ commute with one another, $g_{ii}$ is trivial and $g_{ji}$ is inverse to $g_{ij}$), together with a generating family for $\sum e_i \wedge H_i$. Such a generating family can easily be expressed in terms of bases of the $H_i$, whence in terms of the above parameters. This recovers the presentation of the \emph{parameter group} from~\cite{Lebed-Mortier}; their \emph{parameter matrix} is the matrix of coefficients of such a generating family with respect to the basis $(e_i \wedge e_j)_{i < j}$ of $\Lambda^2(\Z^n)$. Moreover, we also recover \cite[Th.~4.2]{Lebed-Mortier}: $G(Q)$ is abelian (isomorphic to $\Z^n$) if and only if the $e_i \wedge H_i$ generate $\Lambda^2(\Z^n)$, which holds if and only if the maximal minors of the above parameter matrix are globally coprime.




\section{Reduced quandles}\label{sec_reduced}

\emph{Reduced} quandles (called \emph{quasi-trivial} in \cite{Inoue}) were introduced in a topological context (to get invariants of links up to link-homotopy), and topology is our main motivation for studying them. In fact, our main contribution to the theory of reduced quandles (Cor.~\ref{RQ_as_inv}) is a refined version of the main result of \cite{Hughes}. Our result will follow from an investigation that is, however, algebraic in nature, leading up to Theorem~\ref{RQ_from_G(RQ)}. We introduce the relevant topological context only after that, in~\cref{par_reduced_periph}.

We first define reduced quandles (Definition~\ref{def_reduced}), and we construct the universal reduced quotient $R(Q)$ of any quandle $Q$ in \cref{par_reduced_quotient}. Applied to the free quandle, this construction gives the free reduced quandle, which we study next (\cref{par_reduced_free}). In particular, we show that it is injective (Prop.~\ref{RQ_injects_into_RF}), and that it is nilpotent when finitely generated, from which we deduce that finitely generated reduced quandles are nilpotent (Prop.~\ref{reduced_are_nilp}). Then we come to the technical heart of our study of reduced quandles, which is the detailed study of the universal reduced quotient done in~\cref{par_R(Q)_in_detail}. We are then able to understand precisely the relations between the peripheral system associated to $R(Q)$ and the one associated to $Q$ (Theorem~\ref{RQ_from_G(RQ)}). Finally, in~\cref{par_reduced_periph}, we apply this to link theory, to get the announced refinement of the result of \cite{Hughes} (see Cor.~\ref{RQ_as_inv}).

\subsection{Definition}

The following definition first appeared as \cite[Def.~4.1]{Hughes}: 

\begin{defi}\label{def_reduced}
A quandle is called \emph{reduced} if its orbits are trivial subquandles. In other words, $Q$ is reduced if and only if:
$\forall x \in Q,\ \forall \varphi \in \Inn(Q),\ \varphi(x) \qd x = x.$ Equivalently, $Q$ is reduced if and only if:
$\forall x \in Q,\ \forall g \in G(Q),\ (g \cdot x) \qd x = x.$
\end{defi}

The relevance of reduced quandles to our context comes from the fact that finitely generated reduced quandles are nilpotent (see Proposition~\ref{reduced_are_nilp} below).

\begin{rmq}
The equivalence between the last two characterisations comes from the definition of $\Inn(Q)$ as the image of $G(Q)$ in $\Aut(Q)$ by the canonical action. 
\end{rmq}

\subsection{The universal reduced quotient and its enveloping group}\label{par_reduced_quotient}

The following fact is easy to check:
\begin{fait}
The full subcategory $R\mathcal Qdl \subset \mathcal Qdl$ of reduced quandles is reflexive. Precisely, the universal reduced quotient of a quandle $Q$ is its quotient $R(Q)$ by the relations:
\[\forall x \in Q,\ \forall \varphi \in \Inn(Q),\ \varphi(x) \qd x = x.\]
\end{fait}

\begin{lem}\label{pi0(RQ)}
The projection $Q \twoheadrightarrow R(Q)$ induces an isomorphism $\pi_0(Q) \cong \pi_0(R(Q))$.
\end{lem}

\begin{proof}
The trivialisation $\pi_0(Q)$ is the quotient of $Q$ by all the relations $y \qd x = x$, which contain the relations $\varphi(x) \qd x = x$ defining $R(Q)$, whence the conclusion. Alternatively, one can argue that trivial quandles are reduced, so both $\pi_0$ and $\pi_0 \circ R$ are left adjoint to the inclusion of trivial quandles into quandles (which is the composition of the inclusion of trivial ones into reduced ones, followed by the inclusion of reduced quandles into quandles), hence they are canonically isomorphic.
\end{proof}

Let us now introduce a group-theoretic analogue of the construction $R(-)$.

\begin{defi}\label{def_R(G,X)}
Let $G$ be a group and $X$ be a subset of $G$. We define the group $R(G, X)$ as the quotient of $G$ by (the normal closure of) the elements $[x, gxg^{-1}]$, for $x \in X$ and $g \in G$.
\end{defi}

\begin{rmq}\label{R(G,X_up_to_conj)}
The group $R(G, X)$ depends only on the set of conjugacy classes of elements of $X$, since the relations defining it are exactly the ones forcing two elements of each of these classes to commute.
\end{rmq}

This is exactly the construction needed to describe the enveloping group of the universal reduced quotient of a quandle $Q$ from the enveloping group of $Q$:
\begin{prop}\label{R(GQ,Q)}
For any quandle $Q$, we have $G(R(Q)) \cong R(G(Q), \eta(Q))$.
\end{prop}

\begin{proof}
We can compare the universal properties. Namely, a group morphism from $R(G(Q), \eta(Q))$ to a group $G$ is the same thing as a morphism $f$ from $G(Q)$ to $G$ such that $f(g \eta(q)g^{-1})$ commutes with $f(\eta(q))$ for all $q \in Q$ and $g \in G(Q)$. Since $g \eta(q)g^{-1} = \eta(g \cdot q)$ (Lemma~\ref{equivariance_of_eta}), we can rephrase the condition on $f$ as follows: $f$ must send $(g \eta(q)g^{-1}) \eta(q) (g \eta(q)g^{-1})^{-1} = \eta(\eta(g \cdot q) \cdot q) = \eta((g \cdot q) \qd q)$ and $\eta(q)$ to the same element. But the action of $G(Q)$ on $Q$ factors through its image $\Inn(Q)$ in $\Aut(Q)$, so  (using the universal property of $G(Q)$), this is the same thing as a quandle morphism from $Q$ to $G$ identifying $\varphi(q) \qd q$ and $q$ for all $\varphi \in \Inn(Q)$ and all $q \in Q$. By definition of $R(-)$, this is the same thing as a quandle morphism from $R(Q)$ to $G$ or, equivalently, as a group morphism from $G(R(Q))$ to $G$. Thus $\Hom(R(G(Q), \eta(Q)), -) = \Hom(G(R(Q)), -)$, whence our claim.
\end{proof}

\subsection{Reduced free quandles}\label{par_reduced_free}

Recall (from~\cite{Milnor-Link_Groups}) that the \emph{reduced free group} $RF[X]$ on a set $X$ is the quotient of the free group on $X$ by the relations saying that each generator $x \in X$ commutes with all it conjugates. In the language of Definition~\ref{def_R(G,X)}, it is $R(F[X],X)$. In a similar fashion, the \emph{reduced free quandle} $RQ[X]$ on $X$ is the quotient of the free quandle on $X$ by the relations saying that the orbit of each generator $x \in X$ acts trivially on $x$: $RQ[X]$ is precisely $R(Q[X])$, the universal reduced quotient of the free quandle on $X$. From the latter point of view, we see that the \emph{reduced free quandle} is indeed the free object on $X$ in the category of reduced quandles. 

\begin{rmq}
The \emph{reduced free group} was introduced by Milnor to get invariants of links up to link-homotopy: it is what he calls the  link group of the trivial link in \cite{Milnor-Link_Groups}. However, ``being reduced" is not a property of a group, but only of a group endowed with a system of (conjugacy classes of) generators, making the group-theoretic language somewhat awkward when dealing with such objects. On the contrary, we can speak of reduced quandles.
\end{rmq}

One question arises naturally: is $RF[X]$ the enveloping group of $RQ[X]$? One can check that it is indeed so, since $G(RQ[X])$ and $RF[X]$ satisfy the same universal property. Precisely, both of them represent the functor from groups to sets:
\[G \longmapsto \{f : X \rightarrow G \ |\ \forall x \in X,\   \forall w \in \langle f(X) \rangle,\ [{}^w\!f(x),f(x)] = 1\}.\]

\begin{rmq}
This can also be deduced from Proposition~\ref{R(GQ,Q)}. Indeed, by definition, $RF[X]$ is $R(F[X],X)$, which is also (because of Remark~\ref{R(G,X_up_to_conj)}) $R(F[X], \eta(Q[X]))$, and Proposition~\ref{R(GQ,Q)} says that this is $G(R(Q[X]))$, that is, $G(RQ[X])$. 
\end{rmq}

We will now show that $RQ[X]$ injects into its enveloping group $RF[X]$. Thus, in this regard, the free reduced quandle behaves in the same way as the free quandle (as recalled at the beginning of \cref{par_free_nilp}) and the free nilpotent quandle (see Prop.~\ref{free_nilp_qdl}).

\begin{prop}\label{RQ_injects_into_RF}
The canonical map from $RQ[X]$ to its enveloping group $RF[X]$ is injective. As a consequence, $RQ[X]$ identifies with the union of the conjugacy classes of the basis $X$ inside $RF[X]$.
\end{prop}

When $X$ has $n$ elements, the group $RF[X]$ is $n$-nilpotent \cite[Lem.~1.3]{Habegger}. Thus, we can apply Lemma \ref{Q_nilp_ssi_G(Q)_nilp} to get:

\begin{cor}\label{RQn_is_n-nilp}
The reduced free quandle on $n$ elements is $n$-nilpotent.
\end{cor}

\begin{proof}[Proof of Prop.~\ref{RQ_injects_into_RF}]
We use the injectivity criterion (Prop.~\ref{injectivity_criterion}). Since $X$ generates $RQ[X]$, it is enough to show that $C(\eta(x)) \subseteq \Stab(x)$ for $x \in X$ (see the remark following Prop.~\ref{injectivity_criterion}). Moreover, the centralizer of $x$ (that is, of $\eta(x)$) in $RF[X]  = G(RQ[X])$ is exactly its normal closure $N_x$ \cite[Lemma 1.15]{Darne3}.

Now, let $g \in RF[X]$. By definition of reduced quandles, we have $(g\cdot x) \qd x = x$. Equivalently (see Lemma~\ref{equivariance_of_eta}),  $\eta(g\cdot x) \cdot x = x$. In other words, $\eta(g\cdot x) \in \Stab(x)$. Since $\eta(g\cdot x) = g\eta(x)g^{-1}$, this means that $\Stab(x)$ contains the conjugates of $\eta(x)$, that is, it contains $N_x = C(\eta(x)) \subset RF[X]$, whence our conclusion.
\end{proof}

\begin{rmq}
An alternative strategy of proof consists in adapting the proof of Proposition \ref{free_nilp_qdl} to the present context. But either way, the computation of centralisers of elements of $X$ in $RF[X]$ would have been central in the argumentation. In fact, in the opposite direction, one can interpret the application of \cite[Lem.~6.3]{Darne1} in the proof of Proposition \ref{free_nilp_qdl} as the computation of centralizers of elements of $X$ in $F[X]/\Gamma_{c+1}$, and rewrite this proof as another application of the injectivity criterion (Prop.~\ref{injectivity_criterion}).
\end{rmq}

We can now prove, as announced:

\begin{prop}\label{reduced_are_nilp}
Finitely generated reduced quandles are nilpotent. Precisely, if $Q$ is a finitely generated reduced quandle, it is nilpotent of class at most $|\pi_0 Q|$.
\end{prop}

\begin{proof}
Let $S \subseteq Q$ be a finite set of generators. The surjection $Q[S] \twoheadrightarrow Q$ from the free quandle on $S$ induces a surjection $RQ[S] \twoheadrightarrow Q$. Corollary~\ref{RQn_is_n-nilp} says that the reduced free quandle $RQ[S]$ is $|S|$-nilpotent. Thus, by Lemma~\ref{quotients_and_subquandles_of_nilp}, $Q$ is nilpotent of class at most $|S|$. Then Lemma \ref{generators_of_nilp_qdles} allows us to choose as a generating set $S$ in the reasoning above any set of representatives of the orbits of $Q$, whence the result.
\end{proof}

\begin{rmq}\label{RQ[X]_res_nilp}
Proposition~\ref{RQ_injects_into_RF} implies that reduced free quandles are residually nilpotent. Indeed, it allows us to apply Proposition~\ref{res_nilp_for_inj} to them, since reduced free groups are residually nilpotent~\cite[Prop.~1.3]{Darne3}. We can also deduce this fact by adapting the proof for groups, using Corollary~\ref{RQn_is_n-nilp} rather than Proposition~\ref{RQ_injects_into_RF}. Remark that if $X$ is infinite, $RQ[X]$ cannot be nilpotent, since it surjects onto  $RQ[Y]$ for every finite subset $Y$ of $X$, and $RQ[Y]$ is $|Y|$-nilpotent.
\end{rmq}

\subsection{A further study of the universal reduced quotient}\label{par_R(Q)_in_detail}

We now come to the technical heart of our study of reduced quandles. Our main goal (achieved at the end of~\cref{par_RQ_from_G(RQ)}) is Theorem~\ref{RQ_from_G(RQ)}, which describes the universal reduced quotient $R(Q)$ of a quandle $Q$ in terms of group-theoretic data, using the construction $\mathcal Q(-)$ from~\cref{qdl_from_grp} (precisely, from Lemma~\ref{lem_qd_from_gp}). From this we will deduce the announced topological result (Corollary~\ref{RQ_as_inv}).

In order to prove Theorem~\ref{RQ_from_G(RQ)}, we first need to understand precisely how the quotient $Q \twoheadrightarrow R(Q)$ works. Namely, we need to understand the congruence generated by the relations $\varphi(x) \qd x \sim x$ for $x \in Q$ and $\varphi \in \Inn(Q)$, or, equivalently, $(g \cdot x) \qd x \sim x$ for $x \in Q$ and $g \in G(Q)$. Let us first point out that the canonical projection $Q \twoheadrightarrow R(Q)$ can be factored as $Q \twoheadrightarrow Q/K \twoheadrightarrow R(Q)$, where $K$ is the kernel of $G(Q) \twoheadrightarrow G(R(Q))$ (see Corollary~\ref{quotients_by_ker_G(p*)}).

\begin{rmq}
This intermediate quotient $Q/K$ is the one called ``operator-reduced" in \cite{Hughes}.
\end{rmq}

Then, notice that since $(g \cdot x) \qd x = (g \eta(x) g^{-1}) \cdot x$ (see Lemma~\ref{equivariance_of_eta}), the class of $x$ must contain $N_x \cdot x$, where $N_x$ is the subgroup of $G(Q)$ normally generated by $\eta(x)$. Remark that the image of $N_x \cdot x$ in $Q/K$ is also $N_x \cdot x$, if we still denote by $x$ the class of $x$ modulo $K$ and by $N_x$ the subgroup of $G(Q/K)$ normally generated by $\eta(x)$.

\begin{lem}\label{congruence_Q/K}
In $Q/K$, the $N_x \cdot x$ are the classes of a congruence relation.
\end{lem}

The proof uses the following 
\begin{cor}[to Proposition~\ref{R(GQ,Q)}]\label{normal_gen_of_K}
The kernel $K$ of $G(Q) \twoheadrightarrow G(R(Q))$ is normally generated by the $[\eta(x), N_x]$, for $x \in Q$.
\end{cor}

\begin{proof}
Proposition~\ref{R(GQ,Q)} says exactly that $K$ is normally generated (in $G(Q)$) by the $[y, gyg^{-1}]$, for $y \in \eta(Q)$ and $g \in G(Q)$. These are clearly contained in the $[\eta(x), N_x]$, for $x \in Q$. Moreover, $N_x$ is generated by the $g\eta(x)g^{-1}$, so $K$ contains the $[\eta(x), N_x]$, because of the usual commutator identities $[a,bc] = [a,b]b[a,c]b^{-1}$ and $[a,b^{-1}] = b^{-1}[a,b]^{-1} b$.
\end{proof}

\begin{proof}[Proof of Lemma~\ref{congruence_Q/K}]
Let $x$ and $y$ be elements of $Q/K$ such that $(N_x \cdot x) \cap (N_y \cdot y) \neq \varnothing$. This means that $g \cdot x = h \cdot y$ for some $g \in N_x$ and $h \in N_y$. Then Lemma~\ref{equivariance_of_eta} implies: 
\[g \eta(x) g^{-1} = \eta(g \cdot x) = \eta(h \cdot y) = h \eta(y) h^{-1}.\] 
In particular, $\eta(x)$ and $\eta(y)$ generate the same normal subgroup $N_x = N_y$ of $G(Q/K)$. This subgroup clearly contains $h^{-1}g$, hence:
\[N_x \cdot x = N_x \cdot h^{-1}g \cdot x = N_y \cdot y.\] 
This proves that the $N_x \cdot x$ are a partition of $Q/K$. In order to show that the corresponding equivalence relation is a congruence relation, we need to show that for $x, y \in Q/K$, for $g \in N_x$ and $h \in N_y$, we have $(g \cdot x) \qdn{\pm 1} (h \cdot y) \sim x \qdn{\pm 1} y$. This uses the fact that the action of $K$ is trivial. Namely, using Corollary~\ref{normal_gen_of_K}, we have $[\eta(x)^{-1}, g] \in [\eta(x)^{-1}, N_x] \subseteq K$, so $(\eta(x)^{-1} g \eta(x) g^{-1}) \cdot (h \cdot y) = h \cdot y$, whence (still using Lemma~\ref{equivariance_of_eta}):  
\[(g \cdot x) \qd (h \cdot y) = (g \eta(x) g^{-1}) \cdot (h \cdot y) = \eta(x) \cdot (h \cdot y) = x \qd (h \cdot y).\] 
Finally, $x \qd (h \cdot y) = \eta(x) h \eta(x)^{-1} \cdot (\eta(x) \cdot y) \in N_{x \qd y } \cdot  (x \qd y)$, since $\eta(x) N_y \eta(x)^{-1} = N_{x \qd y}$ (both are normally generated by $\eta(x \qd y) = \eta(x) \eta(y) \eta(x)^{-1}$). This shows that $(g \cdot x) \qd (h \cdot y) \sim x \qd y$; the case of $\qdn{-1}$ follows from the same argument where $\qd$ is replaced by $\qdn{-1}$ and $\eta(x)$ by $\eta(x)^{-1}$.
\end{proof}

\begin{rmq}
Recall that compatibility of our relation with $\qd$ does not imply that it is compatible with $\qdn{-1}$, so that both of them needed to be established (see Remark~\ref{congruence_verif}).
\end{rmq}

\begin{rmq}
The first part of the proof works in general to show that the $N_x \cdot x$ are the classes of an equivalence relation, on any quotient $R$ of $Q$ (if we denote by $x$ the image of $x$ in $R$ and by $N_x$ the subgroup of $G(R)$ normally generated by $\eta(x)$). But this equivalence relation is not a congruence in general. In fact, the calculations in the second part of the proof show that it is a congruence relation if and only if $K$ (which is normally generated by the $[\eta(x), N_x]$ for $x \in Q$) acts trivially on $R$.
\end{rmq}

By construction, the congruence relation of Lemma~\ref{congruence_Q/K} (or its preimage by $Q \twoheadrightarrow Q/K$) is the one we are interested in, namely, the one generated by the relations $\varphi(x) \qd x \sim x$ for $x \in Q$ and $\varphi \in \Inn(Q)$. Thus, $Q/K\twoheadrightarrow R(Q)$ is well-understood:
\begin{cor}\label{RQ_from_Q/K}
The fibers of the canonical projection $Q/K\twoheadrightarrow R(Q)$ are the $N_x \cdot x$.
\end{cor}

\subsection{The universal reduced quotient from its enveloping group}\label{par_RQ_from_G(RQ)}

We now want to understand the peripheral system $\mathcal P(R(Q))$ (defined in Example~\ref{Q_from_G(Q)}) associated with the universal reduced quotient $R(Q)$ of a quandle $Q$. This will allow us to apply the constructions from~\cref{sec_construction} to reconstruct $R(Q)$ from group-theoretic data. We are almost ready to do so, but first, we need to apply the results from~\cref{par_R(Q)_in_detail} to understand stabilizers of the action of $G(R(Q))$ on $R(Q)$: 
\begin{prop}\label{Reduced_stabilizers}
Let $Q$ be a quandle, and $q \in Q$. The stabilizer of $\bar q \in R(Q)$ under the action of $G(R(Q))$ is the product $PN$ of the image $P$ of the stabilizer of $q$ in $G(Q)$ with the normal closure $N$ of $\eta(\bar q)$ in $G(R(Q))$. 
\end{prop}

\begin{proof}
First notice that $P$ is clearly contained in $\Stab(\bar q)$, and that the latter must also contain $N$ (by the same argument as in the proof of Prop.~\ref{RQ_injects_into_RF}). Since $N \triangleleft G(R(Q))$, the product $NP$ is the subgroup of $G(R(Q))$ generated by $N$ and $P$. 

Now, each element of $G(R(Q))$ is the class $\bar g$ of some $g \in G(Q)$, modulo the kernel $K$ from~\cref{par_R(Q)_in_detail}. Suppose that $\bar g$ fixes the element $\bar q \in R(Q)$, which is the class of $q \in Q$. Corollary~\ref{RQ_from_Q/K} says that the preimage of $\bar q$ by the projection $Q/K\twoheadrightarrow R(Q)$ is $N_q \cdot q$, so we have $g \cdot q \equiv n \cdot q \ [K]$ for some $n \in N_q$. In other words,  $g \cdot q = kn \cdot q$ for some $n \in N_q$ and some $k \in K$. But this equation means that $(kn)^{-1}g$ lies in the stabilizer $\widetilde P$ of $q$ in $G(Q)$. Finally, $g \in N_q K \widetilde P$, whence $\bar g \in NP$.
\end{proof}

\begin{theo}\label{RQ_from_G(RQ)}
Let $Q$ be a quandle, and let $(q_i)_{i \in I}$ be a family of representatives of its orbits. Let $P_i$ denote the stabilizer of $q_i$ in $G(Q)$, so that:
\[\mathcal P\left(Q, (q_i)_{i \in I}\right) = \left(G(Q), (P_i,\eta(q_i))_{i \in I} \right).\]
Let us denote by $\overline{P_i}$ (resp.~$\overline{q_i}$) the image of $P_i$ (resp.~$q_i$) in $G(R(Q))$ and let $N_i$ be the normal closure of $\eta(\overline{q_i})$ in $G(R(Q)) = R(G(Q), \eta(Q))$. Then:
\[\mathcal P\left(R(Q), (\overline{q_i})_{i \in I}\right) = \left(R(G(Q), \eta(Q)), \left(\overline{P_i}N_i, \eta(\overline{q_i})\right)_{i \in I}\right).\]
In particular, we have a canonical isomorphism:
\[R(Q) \cong \mathcal Q \bigl(R(G(Q), \eta(Q)), (\overline{P_i} N_i, \eta(\overline{q_i}))_{i \in I} \bigr).\]
\end{theo}

\begin{proof}
We apply the construction from Example~\ref{Q_from_G(Q)}. Notice first that Lemma~\ref{pi0(RQ)} ensures that the $\overline{q_i}$ are representatives of orbits of $R(Q)$. Then apply Proposition~\ref{R(GQ,Q)} for the description of the enveloping group, and Proposition~\ref{Reduced_stabilizers} for the description of stabilizers.
\end{proof}

\subsection{Reduced peripheral systems of links}\label{par_reduced_periph}

The fundamental quandle is equivalent, as an invariant of links, to (a weak version of) the peripheral system. This was the main motivation for introducing the fundamental quandle of links: it encodes their peripheral systems in a nice way, and the latter is a complete invariant of links up to taking reversed mirror images of non-split components (which can easily be strengthened to a complete invariant of links -- see Remark~\ref{rk_strong_periph}). 

We now investigate the link-homotopic version of this story. Precisely, we show in Corollary~\ref{RQ_as_inv} that the reduced fundamental quandle is equivalent, in a quite straightforward way, to the (weak) reduced peripheral system. This refines the main result of \cite{Hughes}, which stated that one can recover part of the reduced peripheral system (namely, the reduced fundamental group with the conjugacy classes of meridians) from the reduced fundamental quandle. Our work also relates to the recent \cite{Audoux-Meilhan}, where it is shown that the  (strong) reduced peripheral system (which is obtained from the weak version by adding a little bit of information) is a complete invariant of (welded) links up to link-homotopy \cite{Audoux-Meilhan}.

\subsubsection*{The fundamental quandle}

For the convenience of the reader, let us first recall how the fundamental quandle encodes the (weak) peripheral system of a link. 

\begin{defi}
A \emph{peripheral system of a link} $L$ is $(G_L, (P_i, \mu_i)_{i \in I})$, where $G_L$ is the link group (that is, the fundamental group of its complement), the connected components of $L$ are indexed by $I$, $\mu_i$ is a chosen meridian around the $i$-th component, and $P_i$ is the peripheral subgroup containing this meridian, which is generated by $\mu_i$, together with a longitude $\lambda_i$ along the $i$-th component. 
\end{defi}

Notice that each $P_i$ is a homomorphic image of the fundamental group of the torus, so it is abelian. Thus, $(G_L, (P_i, \mu_i)_{i \in I})$ is a peripheral system in our sense (Definition~\ref{def_Ob(Pph)}). In fact, it is exactly the peripheral system $\mathcal P\left(Q_L, (\mu_i)_{i \in I}\right)$ associated to the fundamental quandle $Q_L$ through the construction of Example~\ref{Q_from_G(Q)}. This follows from the classical theory of the fundamental quandle, as laid out for instance in \cite{Joyce} in the case of knots, with a straightforward adaptation to links. More precisely:
\begin{itemize}[topsep = 3pt, itemsep = 0pt]
    \item $G_L = G(Q_L)$ follows for instance from the fact that both sides are described by the Wirtinger presentation associated to a diagram of $L$ \cite[\S 15]{Joyce}.
    \item Meridians $\mu_i$ around components $L_i$ of $L$ are representatives of orbits of $Q_L$. This can be seen by considering the Wirtinger presentation.
    \item The stabilizer of $\mu_i$ in the action of $G_L$ on $Q_L$ is the peripheral subgroup $P_i$, which is the image of the fundamental group of a torus around $L_i$ in $G_L$. This follows from \cite[Th.~16.1]{Joyce}, whose adaptation to links is straightforward.
\end{itemize}

Then we can apply the machinery from \cref{sec_construction}. Namely, Example~\ref{Q_from_G(Q)} gives:
\[Q_L = \mathcal Q \left(G_L, (P_i,\mu_i)_{i \in I} \right).\]
And its categorical enhancement (Proposition~\ref{QP=1}) implies, through its Corollary~\ref{Qdl_vs_Pph}, the classical:
\begin{cor}\label{funda_Qdl_vs_Pph}
The data of the fundamental quandle $Q_L$ up to isomorphism is the same as the data of the peripheral system $\left(G_L, (P_i,\mu_i)_{i \in I}\right)$ up to equivalence.
\end{cor}

\begin{rmq}\label{rk_strong_periph}
It is a well-known consequence of a celebrated theorem of Waldhausen on Haken $3$-manifolds \cite{Waldhausen} that the peripheral system that we have just defined is a complete invariant of links up to taking reversed mirror images of non-split components. 
One can get a complete invariant of links by adding some information to it. In the world of peripheral systems, this is done by choosing, for each component, a \emph{preferred longitude} $\lambda_i \in P_i$, among the two longitudes (inverse to each other) whose linking number with the $i$-th component is $0$. The stronger invariant obtained by adding this information is the one usually called \emph{the} peripheral system in the literature (although we focus on the weaker version here). 
In the world of quandles, a similar strengthened invariant can be obtained by adding to the fundamental quandle a certain choice of basis for its second cohomology group \cite{Eisermann'}. 
\end{rmq}


\subsubsection*{The reduced fundamental quandle}

Let us now explain how this story adapts to links up to link-homotopy. This corresponds to replacing quandles with reduced ones:

\begin{defi}
The \emph{reduced fundamental quandle} of a link $L$ is obtained from the fundamental quandle $Q_L$ by taking its universal reduced quotient $R(Q_L)$.
\end{defi}

The main goal of the current section is to prove that the data of $R(Q_L)$ is equivalent to the data of the (weak) reduced peripheral system of $L$ up to equivalence. We first state our definition of the reduced peripheral system of a link (which is slightly weaker than the one of \cite{Audoux-Meilhan} -- see Remark~\ref{rk_strong_reduced_periph}).

\begin{defi}
A \emph{reduced peripheral system} $(RG_L, (P_i N_i, \mu_i)_{i \in I})$ for a link $L$ is obtained from a peripheral system $(G_L, (P_i, \mu_i)_{i \in I})$ by taking $RG_L = R(G_L, \{\mu_i\}_i)$ and, for each $i$, the images of the $\mu_i$ in this quotient (still denoted by $\mu_i$) and the products $P_i N_i$ of the image of $P_i$ (still denoted by $P_i$) with the normal closure $N_i$ of $\mu_i$ in $RG_L$
\end{defi}

\begin{rmq}
$RG_L$ is what Milnor called the link group of $L$ in~\cite{Milnor-Link_Groups}. Arguably, given the current state of terminology, it should rather be called the \emph{reduced link group} of $L$, in order to avoid all ambiguity with $G_L$ itself.
\end{rmq} 

A straightforward combination of Theorem~\ref{RQ_from_G(RQ)} with the description of the peripheral system $\mathcal P\left(Q_L, (\mu_i)_{i \in I}\right)$ from the previous paragraph (using Remark~\ref{R(G,X_up_to_conj)}, to see that $RG_L = R(G_L, \eta(Q_L))$) shows that the reduced peripheral system is exactly the peripheral system associated to the reduced fundamental quandle $R(Q_L)$ through the construction of Example~\ref{Q_from_G(Q)}. Namely:
\[\mathcal P\left(R(Q_L), (\mu_i)_{i \in I}\right) = \left(RG_L, (P_i N_i, \mu_i)_{i \in I}\right).\]

Then we can apply the machinery from \cref{sec_construction} to the reduced fundamental quandle $R(Q_L)$. Precisely, Example~\ref{Q_from_G(Q)} gives:
\[R(Q_L) = \mathcal Q\left(RG_L, (P_i N_i, \mu_i)_{i \in I}\right).\]
Furthermore, by specializing Corollary~\ref{Qdl_vs_Pph} to this context, we obtain the main result of the present section, which is a reduced version of the classical equivalence between the fundamental quandle and the peripheral system of a link (recalled above as Corollary~\ref{funda_Qdl_vs_Pph}):
\begin{cor}\label{RQ_as_inv}
As an invariant of links, the reduced fundamental quandle (up to isomorphism) is equivalent to the (weak) reduced peripheral system (up to equivalence).
\end{cor}


\begin{rmq}
This was partially obtained as \cite[Th.~4.5]{Hughes}, who showed that part of the reduced peripheral system could be obtained from the reduced fundamental quandle.
\end{rmq}

\begin{rmq}\label{rk_strong_reduced_periph}
Like for links up to isotopy (Remark~\ref{rk_strong_periph}), our version of the reduced peripheral system is a weak one, compared to the (strong) one shown to be a complete invariant of (welded) links up to welded link-homotopy in \cite{Audoux-Meilhan}. Again, this strengthening corresponds to choosing \emph{preferred longitudes}. So one way to turn the result of \cite{Audoux-Meilhan} into a result about quandles would be to add some information to our reduced quandles (which could probably take the form of classes in its cohomology) in order to get an invariant equivalent to theirs. Alternatively, one could look for a result about the strength of the reduced fundamental quandle as an invariant of welded links. In fact, we conjecture that the reduced fundamental quandle is a complete invariant of welded links up to welded link-homotopy and taking reversed mirror images of non-split components (in the $4$-dimensional sphere).
\end{rmq}

\section{Nilpotent quandles and invariants of welded links}\label{sec_braids}

Here we give yet another characterisation of nilpotent quandles, this time in terms of braid colourings. Although the main result of the section (Theorem~\ref{nilp_via_action_of_wP}) can be thought of as an algebraic result about nilpotency of a quandle $Q$ in terms of the canonical action of $\Aut(Q_n)$ on $Q^n$, we really have invariants of welded links in mind. For this reason, we take some time to explain how one can interpret this result in terms of colouring invariants, and how it fits in the landscape of invariants of knots and links.

\subsection{Automorphisms of free quandles}\label{Aut(Qn)}

As recalled at the beginning of \cref{par_free_nilp}, the free quandle $Q_n$ on $n$ generators $x_1, ..., x_n$ identifies with the union of the conjugacy classes of the $x_i$ in the free group $F_n$, which is its enveloping group. The functor $G(-)$ induces a map $G_* : \Aut(Q_n) \rightarrow \Aut(G(Q_n)) = \Aut(F_n)$, extending any automorphism of $Q_n$ to an automorphism of $F_n$. Since the $x_i$ generate $F_n$, an automorphism of $F_n$ is determined by the image of the $x_i$, hence this map must be injective. Moreover, an element in the image of $G_*$ must send each generator $x_i$ into $Q_n$, so it belongs to the subgroup $\Aut_C(F_n)$ of basis-conjugating automorphisms of $F_n$, sending each $x_i$ to a conjugate of some $x_j$. The inclusion $\ima(G_*) \subseteq \Aut_C(F_n)$  is in fact an equality: any basis-conjugating automorphism of $F_n$ must preserve $Q_n \subset F_n$, so it must belong to the image of $G_*$. Thus, we have recovered the classical:
\[\Aut(Q_n) \cong \Aut_C(F_n).\]

Such an identification has the following consequence: for any quandle $Q$, the group $\Aut_C(F_n)$ acts (from the right) on the set $\Hom(Q_n,Q) \cong Q^n$. This is the action via \emph{Nielsen transformation} \cite[Chap.~3]{MKS}, which is similar to the action of $\Aut(F_n)$ on the set $\Hom(F_n, G) \cong G^n$ for any group $G$. The latter has the following explicit description: if $\varphi$ is an automorphism of $F_n$ sending each $x_i$ to a word $w_i(x_1, ..., x_n)$, then:
\[\varphi \cdot (g_1, ..., g_n) = \left(w_i(g_1, ..., g_n) \right)_i.\]
It is easy to see that the same description applies to the action of $\Aut_C(F_n)$ on $Q^n$: if $\beta \in \Aut_C(F_n)$ sends each $x_i$ to ${}^{w_i}\! x_{\sigma(i)}$ ($\sigma \in \Sigma_n$ and $w_i = w_i(x_1,... x_n) \in F_n$), then:  
\[\beta  \cdot (q_1, ..., q_n) = \left(w_i(q_1, ..., q_n) \cdot q_{\sigma(i)} \right)_i,\]
where $w_i(q_1, ..., q_n) \in G(Q)$ is obtained from $w_i$ by evaluating each $x_j$ at $\eta(q_j) \in G(Q)$.

\medskip

In fact, this action may already be known to the reader, with a more geometric description: it corresponds to colouring (welded) braids by elements of our quandle. Indeed, recall that $\Aut_C(F_n)$ identifies with the group $wB_n$ of loop braids (also called welded braids) (\cite{FRR_Braids}; see also \cite{Damiani}), whence:
\[\Aut(Q_n) \cong wB_n.\]
Now, if we colour the top of the strands of $\beta \in wB_n$ by the colours $q_1, ..., q_n$, there is a unique way to colour the rest of the braid by following the rules of Figure~\ref{fig_colourings}, and $\beta  \cdot (q_1, ..., q_n)$ is exactly the list of colours that we get at the bottom of our braid.

\begin{figure}
\centering     
\subfigure{\label{fig:a}\includegraphics[width=6em]{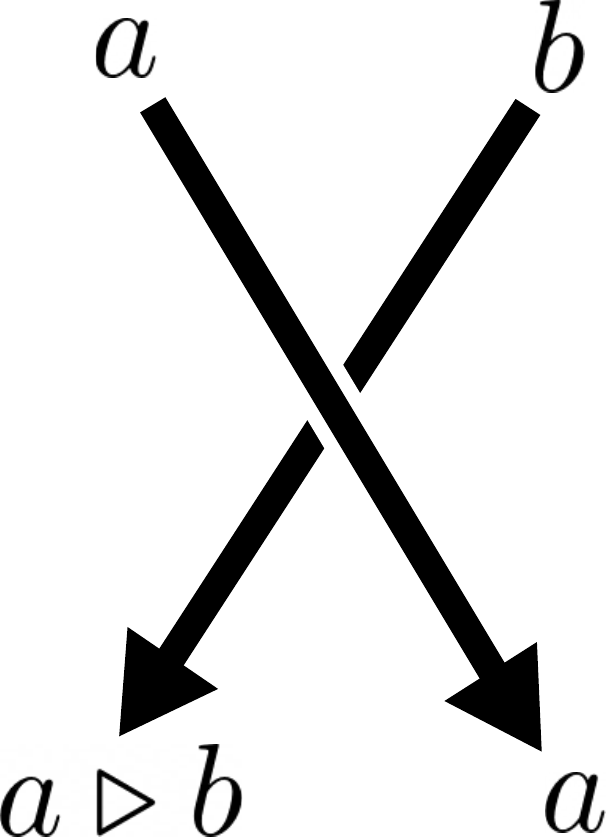}}
\hspace{3.5em}
\subfigure{\label{fig:b}\includegraphics[width=6.4em]{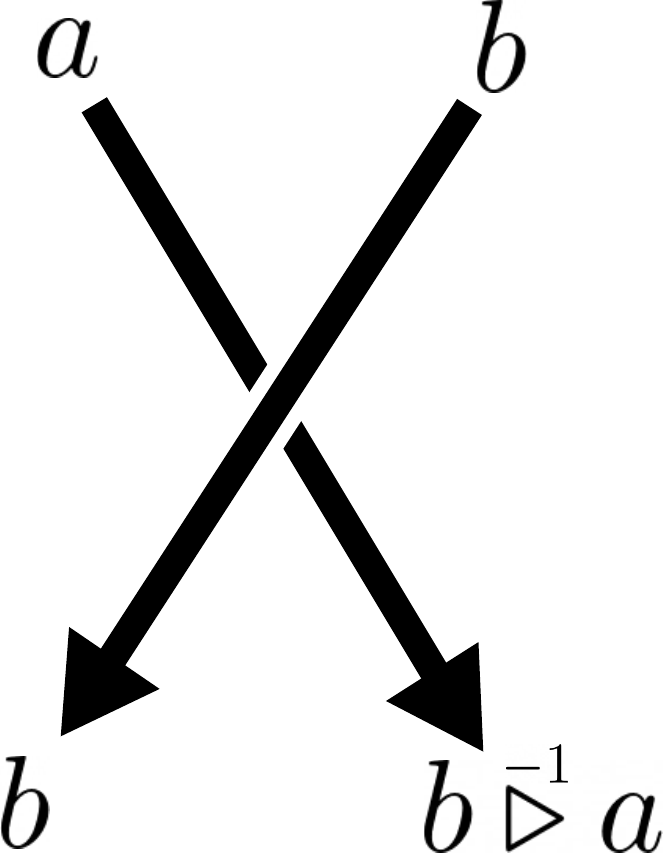}}
\hspace{3em}
\subfigure{\label{fig:c}\includegraphics[width=5.4em]{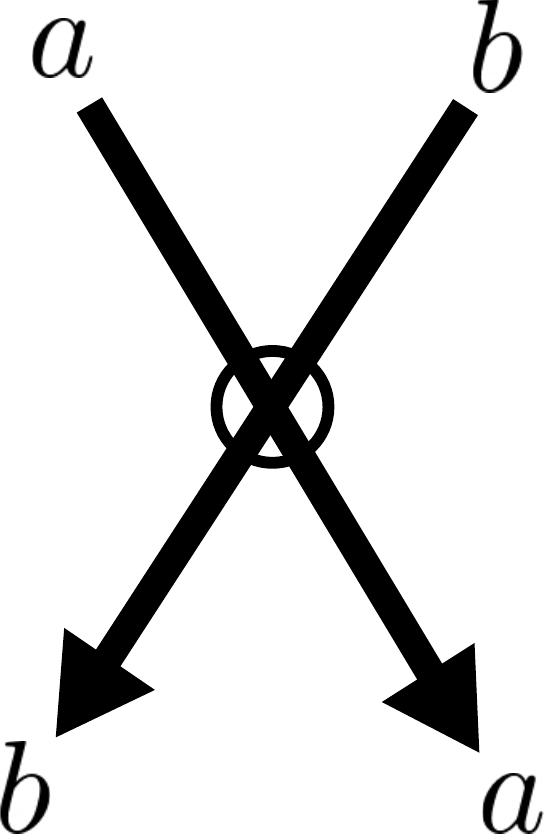}}
\caption{Colouring of a (welded) braid by a quandle}
\label{fig_colourings}
\end{figure}

\subsection{Nilpotency via Nielsen transformations}\label{par_Nielsen}

We have seen (Prop.~\ref{free_nilp_qdl}) that the free nilpotent quandle $Q_{n,c} = Q_n/\Gamma_c$  of class $c$ on $n$ generators embeds into the free nilpotent group $F_{n,c} = F_n/\Gamma_{c+1}$ and that, as a consequence, it identifies with the union of the conjugacy classes of the $x_i$ in $F_{n,c}$. Since $F_{n,c} = G(Q_{n,c})$ (see the proof of Prop.~\ref{free_nilp_qdl}), we can use the same reasoning as in \cref{Aut(Qn)} to describe its group of automorphisms as the group $\Aut_C(F_{n,c})$ of basis-conjugating automorphisms of $F_{n,c}$. 

\begin{rmq}
The kernel of the canonical morphism $\varphi$ from $\Aut_C(F_n) = \Aut(Q_n)$ to $\Aut_C(F_{n,c}) = \Aut(Q_{n,c})$ is, by definition, the $c$-th Andreadakis kernel $\mathcal A_c^C$, which consists of (basis-conjugating) automorphisms of $F_n$ acting trivially on $F_{n,c}$. In terms of quandle theory, $\mathcal A_c^C$ is the group of automorphisms of $Q_n$ acting trivially on $Q_{n,c} = Q_n/\Gamma_c$. The image of $\varphi$ (isomorphic to $\Aut(Q_n)/\mathcal A_c^C$) is a \emph{proper} subgroup of $\Aut(Q_{n,c})$ if $n \geq 2$ and $c \geq 3$, called the subgroup of \emph{tame} automorphisms of $Q_{n,c}$. A proof of the latter statement, relying on the Morita-Satoh trace for automorphisms of the free group, is given in Appendix~\ref{sec_non-tame}.
\end{rmq}

We state the next theorem in terms of (welded) braids, which is the author's preferred point of view; however, this point of view is not really used in the proof, so the reader not familiar with them may prefer to see ``$wB_n$" as an abbreviation for ``$\Aut(Q_n)$" (or ``$\Aut_C(F_n)$"). Then the subgroup $wP_n$ of \emph{pure} welded braids corresponds to the group of \emph{pure} basis-conjugating automorphisms of $F_n$, which is the subgroup $\mathcal A_1^C$ of $\Aut_C(F_n)$ consisting of automorphisms sending each $x_i$ to a conjugate of $x_i$ (not of some $x_j$). In quandle-theoretic terms, $\mathcal A_1^C = wP_n$ is the group of automorphisms of $Q_n$ acting trivially on $\pi_0(Q_n) = Q_{n,1}$, \emph{i.e.}~preserving orbits.

\begin{theo}\label{nilp_via_action_of_wP}
Let $c \geq 1$ be an integer, and $Q$ be a quandle. The following are equivalent:
\begin{itemize}[itemsep=-0.3em, topsep = 4pt]
   \item $Q$ is $c$-nilpotent,
   \item The canonical action of $\Gamma_{c}(wP_n)$ (resp.~of $\mathcal A_c^C$)  on $Q^n$ is trivial for all $n$,
   \item The canonical action of $\Gamma_{c}(wP_n)$ (resp.~of $\mathcal A_c^C$) on $Q^n$ is trivial for some $n > c$.
\end{itemize}
\end{theo}

\begin{rmq}
The classical inclusion $\Gamma_{c}(wP_n) \subseteq \mathcal A_c^C$ had been conjectured for some time to be an equality, but this was disproved very recently~\cite{DES}. The main difficulty in tackling this problem can in fact be observed in the present context: most of the time both $\Gamma_{c}(wP_n)$ and $\mathcal A_c^C$ tend to behave in the same way.
\end{rmq}

\begin{proof}[Proof of Theorem~\ref{nilp_via_action_of_wP}]
If $Q$ is $c$-nilpotent, then $Q^n = \Hom(Q_n, Q) = \Hom(Q_n/\Gamma_c, Q)$. Thus the action of $wB_n$ on $Q^n$ factorizes through $wB_n/\mathcal A_c^C$, so the action of $\mathcal A_c^C$ on $Q^n$ is trivial. Moreover, $\Gamma_{c}(wP_n) \subseteq \mathcal A_c^C$, hence $\Gamma_{c}(wP_n)$ acts trivially on~$Q^n$ whenever $\mathcal A_c^C$ does.

Suppose now that the action of $\Gamma_{c}(wP_n)$ on $Q^n$ is trivial for some fixed $n > c$. For $v \in F_c = \langle x_1, ..., x_c \rangle$, let us denote by $\varphi_v$ the automorphism of $F_n$ fixing all $x_i$ save $x_{c+1}$, which it sends to $x_{c+1}^v$ (it is an automorphism because $x_{c+1}$ does not appear in $v$, so $x_{c+1} \mapsto {}^v\! x_{c+1}$ defines its inverse). Notice that $v \mapsto \varphi_v$ is a group morphism from $F_c$ to $wP_n$.
Let $w = [x_1,[x_2, ..., [x_{c-1}, x_c]...]] \in F_n$, and consider the automorphism $\gamma := \varphi_w^{-1}$. We have $\gamma \in \Gamma_c(wP_n)$, since $\varphi_w = [\varphi_{x_1},[\varphi_{x_2}, ..., [\varphi_{x_{c-1}},\varphi_{x_c}]...]]$. The corresponding Nielsen transformation sends $(q_1, ..., q_n) \in Q^n$ to $(q_1, ..., q_c, w(q_1, ..., q_c) \cdot q_{c+1}, q_{c+2}, ..., q_n)$, where  $w(q_1, ..., q_c)$ is computed in $G(Q)$ (precisely, it is $w(\eta(q_1), ..., \eta(q_c))$). Thus, $\gamma$ acts trivially on $Q^n$ if and only if $w(q_1, ..., q_c) \cdot q_{c+1} = q_{c+1}$ for all $q_1, ..., q_{c+1} \in Q$. Which means exactly that $w(q_1, ..., q_c)$ acts trivially on $Q$, for all $q_1, ..., q_c \in Q$. Since these elements $w(q_1, ..., q_c)$ generate $\Gamma_c(G(Q))$ (Proposition~\ref{generating_identities}), this means that $\Gamma_c(G(Q))$ acts trivially on $Q$, hence $Q$ is $c$-nilpotent (Lemma~\ref{c-nilp_in_terms_of_G}).
\end{proof}

In the context of invariants of (welded) links, this theorem implies that the number of colorings of a link by a finite nilpotent quandle $Q$ is invariant by cutting open a link and inserting an element of $\Gamma_{c}(wP_n)$. This kind of condition has been studied for instance in \cite{Stanford1, Stanford2}, where it is shown that invariants of \emph{knots} that satisfy this invariance by insertion of elements of $\Gamma_{c}(P_n)$ contain exactly as much information as Vassiliev invariants of degree at most $n$. However, the number of colorings by a finite quandle is a Vassiliev invariant only if it is constant (see \cite{Eisermann''}, whose statement for groups can easily be adapted to quandles). Which is coherent with the following:
\begin{fact}\label{colourings_of_knots}
The only colorings of a (welded) \emph{knot} by a nilpotent quandle are trivial ones.
\end{fact}

\begin{proof}
Colourings of a knot $K$ by a quandle $Q$ are morphisms from $Q_K$ to $Q$, and $Q_K$ is connected if $K$ is a knot, so Corollary~\ref{Hom(connected, nilp)} applies.
\end{proof}

This says that, strictly speaking, nilpotent quandles are not relevant to \emph{knot} theory: they are very relevant, however, to \emph{link} theory. In particular, the number of colorings by $Q$ is in general not constant on (welded) links with a fixed number $n \geq 2$ of components, so (again by \cite{Eisermann''}) it is not a Vassiliev invariant of \emph{links}, even if it is invariant by inserting elements of $\Gamma_{c}(wP_n)$. The study of these invariants and, more generally, of cocycle invariants associated to nilpotent quandles, will be the subject of further research.

\appendix

\section{Appendix : Nilpotent racks} \label{sec_racks}

We now review how our notion of nilpotency can be adapted to racks. Recall that \emph{racks} are defined by the first two axioms of Definition~\ref{def_qdl}. Compared to quandles, the third axiom is no longer required. Since this axiom is the one ensuring that colourings are compatible with the first Reidemeister move, racks can no longer be used to colour knots. They are, however, very suited to colouring \emph{braids}, for which the first Reidemeister move is irrelevant. In particular, they play a prominent role in the theory of set-theoretic solutions of the Yang Baxter equation, where a rack (called the \emph{derived solution}) is associated to any non-degenerate solution (see \cite[Th.~2.3]{Soloviev} or \cite[Def.~5.8]{Lebed-Vendramin2}). 

The inclusion $\iota$ of quandles into racks is that of a subvariety (in the sense of universal algebra), so it has a left adjoint $Q(-)$ (in fact, $Q(R)$ is the quotient of $R$ by the relations $x \qd x = x$). By composing adjunctions, we see that $\iota \circ c$ (where $c$ is the functor from Prop.~\ref{adj_c_G}) has a left adjoint $G \circ Q$, and we usually denote by $G(R)$ the group $G(Q(R))$, which is the usual enveloping group (or \emph{structure group}) of a rack. Remark that the natural map from $R$ to $G(R)$ factors through $Q(R)$, so it cannot be injective unless $R$ is actually a quandle. As a consequence, injectivity issues do not make sense for general racks, and we can disregard them in this appendix.

Looking for a notion of $c$-nilpotency having nice properties and giving back the previous one for racks that are in fact quandles, we find two different (but closely related) notions that seem worth investigating. The nicest one comes from a direct generalisation of Definition~\ref{def_nilp_qdl}:
\begin{defi}\label{def_nilp_rack}
A rack $R$ is called \emph{nilpotent of class at most $c$}, or \emph{$c$-nilpotent}, if its group of inner automorphisms is $(c-1)$-nilpotent.
\end{defi}

With this definition, most of our results obtained for quandles can be generalised easily to racks, with little to no change. Precisely, here is a list of the changes required:

\begin{itemize}[topsep=2pt, itemsep= -2pt]
    \item The characterisation in terms of coverings (Th.~\ref{nilpotency_via_coverings}) needs to be modified slightly: a rack is $c$-nilpotent if and only if it is a $(c-1)$-fold covering of a trivial rack. This is slightly stronger than being a $c$-fold covering of the one-point rack -- see the discussion about weak nilpotency (see Definition~\ref{weak_nilp}) below.
    \item The free $c$-nilpotent rack is easier to understand than the free $c$-nilpotent quandle. Recall that the free rack on a set $X$ is $R[X] = F[X] \times X$, with the law $(v,x) \qd (w,y) = (vxv^{-1}w, y)$. Its enveloping group is $F[X]$, acting by left multiplication on the first factor, so the free $c$-nilpotent rack is the quotient by the action of $\Gamma_c(F[X])$, which is $(F[X]/\Gamma_c) \times X$.
    \item In Corollary~\ref{nilp_connected_are_*}, the conclusion should be replaced by: our rack is generated by a single element. Which means that it is isomorphic to $\Z/n$, endowed with the law $i \qd j = j+1$, for some $n \geq 0$.
    \item Remark~\ref{z_i_in_H_i} is the key to adapting results of sections~\ref{sec_construction},~\ref{sec_G(Q)} and~\ref{section_2-nilp}: the condition $z_i \in H_i$ (which becomes $e_i \in H_i$ later on) needs to be dropped. Accordingly, in the definition of peripheral systems (Definition~\ref{def_Ob(Pph)}), the element $z_i$ should only be assumed to belong to the centralizer of $H_i$ (not its center). Moreover, the idempotency condition for crossed sets (Definition~\ref{def_idempotent_XSet} and Convention~\ref{conv_idempotent_XSet}) needs to be dropped. Apart from this, the results of these three sections can be adapted to racks in a straightforward way. 
    \item It does not mean anything for a rack to be reduced (resp.~injective) unless it is a quandle, so results about reduced (resp.~injective) quandles do not have any meaningful generalisation to racks.
    \item The results of \cref{sec_braids} hold without much change, the action of the welded braid group $wB_n$ on $R^n$ (for a rack~$R$) being depicted by the same pictures, corresponding to the same formulas. In fact, this corresponds to an identification of $wB_n = \Aut(Q_n)$ with a subgroup of $\Aut(R_n) \cong \Z^n \rtimes \Aut(Q_n)$ (see~\cite[\S~4.2]{FRR_Braids} for the latter decomposition).
\end{itemize}

The first of these changes points to another notion of nilpotency for racks: 
\begin{defi}\label{weak_nilp}
We call a rack \emph{weakly $c$-nilpotent} if it is a $c$-fold covering of the one-point rack. 
\end{defi}
Since any trivial rack is a covering of the one-point rack, a $c$-nilpotent rack is weakly $c$-nilpotent. Since the one-point rack is a trivial rack, a weakly $c$-nilpotent rack is $(c+1)$-nilpotent. 

\begin{ex}
The free rack $R_1$ on one generator (which is $\Z$ endowed with the law $i \qd j = j+1$) is a covering of the one-point rack, but it is not a trivial rack: $\Inn(R_1) \cong \Z$. Thus, it is weakly $1$-nilpotent, but only $2$-nilpotent, not $1$-nilpotent. 
\end{ex}
Weak nilpotency (defined as weak $c$-nilpotency for some $c$) is the same as nilpotency, even if the nilpotency class may differ.  Thus, all the results assuming only nilpotency that work for nilpotent racks do hold for weakly nilpotent ones: they are generated by representatives of orbits, they are Hopfian, etc. However, the results depending on the nilpotency class cannot be adapted directly. For instance, we lose the description of the universal $c$-nilpotent quotient as the quotient by a group action. Nevertheless, we are going to show that such a universal quotient does exist, by showing that the subcategory of weakly $c$-nilpotent racks is in fact a subvariety of the variety of racks, described by equations very similar to the ones from Corollary~\ref{nilp_is_reductive}.

Before showing this, let us examine the reason why our two notions agree for quandles. A rack is a covering of the one-point rack if and only if all its elements behave in the same way. For quandles, they then have to behave trivially, since $x \qd y = y \qd y = y$, so being a covering  of the one-point quandle is the same as being trivial. But for racks, these two notions may differ: the equations $x \qd y = y$ and $x \qd y = x' \qd y$ define the same subvariety for quandles, but not for racks. The following proposition generalises this by exhibiting equations defining weak $c$-nilpotency which are equivalent to the ones from Corollary~\ref{nilp_is_reductive} for quandles (take $x_1'= x_2$), but not for racks.

\begin{prop}\label{weak_nilp_is_reductive}
A rack $R$ is \emph{weakly} $c$-nilpotent if and only if it satisfies the following identities:
\[\forall x_1, ..., x_{c+1}, x_1' \in R,\ \left( \cdots ((x_1 \qd x_2) \qd x_3) \qd \cdots \right) \qd x_{c+1} = \left( \cdots ((x_1' \qd x_2) \qd x_3) \qd \cdots \right) \qd x_{c+1}.\]
\end{prop}

For the proof, we need to introduce the rack $\inn(R)$, which is the image of $R$ in $\Aut(R)$ by the rack morphism $p_0: x \mapsto x \qd (-)$. The latter (corestricted to $\inn(R)$) is a covering, and a surjective rack morphism $p: R \rightarrow S$ is a covering if and only if there is a (necessarily unique and surjective) morphism $f : S \twoheadrightarrow \inn(R)$ such that $f \circ p = p_0$.

\begin{proof}[Proof of Proposition~\ref{weak_nilp_is_reductive}]
For $c = 1$, this is clear. Suppose that we already know the result for $c-1$. Remark that this implies that a quotient of a weakly $(c-1)$-nilpotent rack is again weakly $(c-1)$-nilpotent.  For short, let us write $w_c(x_1, ..., x_{c+1})$ for $\left( \cdots ((x_1 \qd x_2) \qd x_3) \qd \cdots \right) \qd x_{c+1}$. That $w_c(x_1, ..., x_{c+1})$ does not depend on $x_1$ means exactly that $w_{c-1}(x_1, ..., x_c) \qd (-)$ does not. But the latter is $p_0(w_{c-1}(x_1, ..., x_c)) = w_{c-1}(p_0(x_1), ..., p_0(x_c))$ which, by the induction hypothesis, does not depend on $x_1$ if and only if $\inn(R)$ is weakly $(c-1)$-nilpotent. So we need to show that $R$ is weakly $c$-nilpotent if and only if $\inn(R)$ is weakly $(c-1)$-nilpotent. One implication follows clearly from the fact that $R$ is a covering of $\inn(R)$. Conversely, if $R$ is weakly $c$-nilpotent, then there is a covering $p: R \rightarrow S$ with $S$ weakly $(c-1)$-nilpotent, whence also its quotient $\inn(R)$.
\end{proof}

\begin{rmq}
The same proof allows one to get a direct link between the characterisation of $c$-nilpotency in terms of identities (Cor.~\ref{nilp_is_reductive}) and the one in terms of coverings (Th.~\ref{nilpotency_via_coverings}). 
\end{rmq}

We can deduce from this that every rack $R$ has a \emph{universal weakly $c$-nilpotent quotient} $N_c(R)$, which is the quotient of $R$ by (the congruence generated by) the relations from Proposition~\ref{weak_nilp_is_reductive}:
\begin{cor}
The inclusion of the (full) subcategory of weakly $c$-nilpotent racks into racks has a left adjoint $N_c$. 
\end{cor}

\begin{rmq}
 Since weak $c$-nilpotency is weaker than $c$-nilpotency, the quotient $N_c(R)$ of $R$ is bigger than $R/\Gamma_c$. However, since the two sets of relations are equivalent for quandles, we can see that $Q(N_c(R)) = Q(R/\Gamma_c)$, which is also $Q(R)/\Gamma_c$ (in other words, $Q \circ N_c$, $Q \circ (-)/\Gamma_c$ and $Q(-)/\Gamma_c$ are all left adjoint to the same functor of inclusion of $c$-nilpotent quandles into racks, so they are equal). Moreover, this implies that $N_c(R)$ and $R/\Gamma_c$ have the same enveloping group $G(R)/\Gamma_{c+1}$. 
\end{rmq}

\section{Appendix : Non-tame automorphisms} \label{sec_non-tame}

Let $F_n$ denote the free group on $n$ generators and $F_{n,c} := F_n/\Gamma_{c+1}$ be its universal $c$-nilpotent quotient. It is known that the canonical morphism from $\Aut(F_n)$ to $\Aut(F_{n,c})$ is not surjective if $c \geq 3$ and $n \geq 2$ \cite{BGLM}. Automorphisms of $F_{n,c}$ not coming from an automorphism of $F_n$ are called \emph{non-tame}. The main obstruction to tameness is the Morita-Satoh trace \cite[Th.~3.2]{BGLM}. It is also the only one for $n \geq c+1$ \cite[\S 2.6]{Darne1}. 

We now investigate the similar situation for quandles, which corresponds to the restriction to basis-conjugating automorphisms. Precisely, we show that the canonical morphism from $\Aut(Q_n) \cong \Aut_C(F_n)$ ($\cong wB_n$) to $\Aut(Q_{n,c}) \cong \Aut_C(F_{n,c})$ is not surjective if $c \geq 3$ and $n \geq 2$, using the Morita-Satoh trace. Here, $\Aut_C$ denotes the set of automorphisms sending each element $x_i$ of a fixed basis to a conjugate of some $x_j$, and the isomorphisms are obtained using the reasoning from~\cref{Aut(Qn)} (which can be readily adapted to the nilpotent case, as remarked at the beginning of~\cref{par_Nielsen}). We have the following commutative diagram:
\[\begin{tikzcd}
\Aut(F_n) \ar[r,"\varphi"]  
&\Aut(F_{n,c}) \\
{\begin{array}{cc}
\Aut_C(F_n)  \\[2pt]
\cong \Aut(Q_n) 
\end{array}} 
\ar[r, shift left = 3,  "\varphi"] \ar[u, hook]
&{\begin{array}{cc}
\Aut_C(F_{n,c})  \\[2pt]
\cong \Aut(Q_{n,c}), 
\end{array}} 
\ar[u, hook] 
\end{tikzcd}\]
where $\varphi$ denotes both the canonical maps.

\paragraph*{Disclaimer:} This appendix is using techniques that are very different from the rest of the paper. It is also quite technical, using freely a lot of material coming from the theory of automorphisms of free groups. However, the results are quite striking, and we feel that even the non-specialist could benefit from reading it and getting a taste of the techniques involved.

\subsection{Andreadakis filtrations}

Let us recall \cite[\S 1.3]{Darne1} that for any group $G$, the automorphism group $\Aut(G)$ is filtered by the \emph{Andreadakis filtration} $\mathcal A_*(G)$, where $\mathcal A_j(G)$ is the subgroup of automorphisms acting trivially on $G/\Gamma_{j+1}$:
\[\Aut(G) = \mathcal A_0 \supseteq \mathcal A_1(G) \supseteq \mathcal A_2(G) \supseteq  \cdots.\]
Let us define 
\[\gr_j(\mathcal A_*(G)) := \mathcal A_j(G)/\mathcal A_{j+1}(G).\]
If $j \geq 1$, then this quotient is abelian, and the direct sum of these groups bears the structure of a graded Lie ring $\gr(\mathcal A_*(G))$, whose bracket is induced by commutators in $\Aut(G)$.

\medskip\noindent
Let us remark that $\ker(\varphi) = \mathcal A_c(F_n)$, so the previous square becomes a square of injections:
\[\begin{tikzcd}
\Aut(F_n)/\mathcal A_c(F_n) \ar[r, hook, "\varphi"]  
&\Aut(F_{n,c}) \\
\Aut_C(F_n)/\mathcal A_c^C(F_n)
\ar[r, hook, "\varphi"] \ar[u, hook]
&\Aut_C(F_{n,c}), 
\ar[u, hook] 
\end{tikzcd}\]
where $\mathcal A_*^C(F_n) := \mathcal A_*(F_n) \cap \Aut_C(F_n)$ is the filtration on $\Aut_C(F_n)$ induced by the Andreadakis filtrations of $\Aut(F_n)$. Notice that the filtrations considered on these groups are finite ones, since $\mathcal A_{c+1}(F_{n,c}) = \{id\}$, by construction. On the associated graded, one gets, for $j \leq c$:
\[\begin{tikzcd}
\gr_j(\mathcal A_*(F_n)) \ar[r, hook, "\varphi_j"]  
&\gr_j(\mathcal A_*(F_{n,c})) \\
\gr_j(\mathcal A_*^C(F_n))
\ar[r, hook,  "\varphi_j"] \ar[u, hook]
&\gr_j(\mathcal A_*^C(F_{n,c})).
\ar[u, hook] 
\end{tikzcd}\]
The vertical maps are injective, by definition of $\mathcal A_*^C(F_n)$ and $\mathcal A_*^C(F_{n,c})$. The injectivity of the horizontal maps comes from $\varphi^{-1}(\mathcal A_{j+1}(F_{n,c})) = \mathcal A_{j+1}(F_n)$. We are looking for obstructions to the surjectivity of the horizontal maps, from which we will deduce that the $\varphi$ are also not surjective.

\begin{itemize}[topsep=2pt, itemsep= -2pt]
\item For $j = 0$, the quotients in the square are the images of the previous automorphism groups in $\Aut(F_n^{ab}) = GL_n(\Z)$. The top groups are both identified to $GL_n(\Z)$, the bottom ones to the subgroup of permutation matrices, and both $\varphi_0$ are isomorphisms.
\item For $j \geq 1$, these squares are the homogeneous components of a commutative square of graded Lie rings, which we now investigate.
\end{itemize}

\subsection{Associated Lie rings.}

Recall \cite[\S 1.4]{Darne1} that for any group $G$, the action of $\Aut(G)$ on $G$ induces an action of $\gr(\mathcal A_*(G))$ on $\gr(G) := \gr(\Gamma_*(G))$ by homogeneous derivations. This action can be encoded by a morphism of graded Lie rings:
\[\tau : \gr(\mathcal A_*(G)) \hookrightarrow \Der_*(\gr(G)),\]
which is injective by definition of $\mathcal A_*(G)$. This morphism, called the \emph{Johnson morphism} associated to $G$ (by analogy with the usual Johnson morphism for mapping class groups), is given by the explicit formula:
\[\forall \sigma \in \mathcal A_j,\ \forall x \in \Gamma_i(G),\ \ \tau(\overline \sigma) : \overline x \longmapsto \overline{\sigma(x) x^{-1}}.\]

Let us recall that derivations of the free Lie ring decompose uniquely as sums of homogeneous derivations. This holds also for the free $c$-nilpotent Lie ring, and $\Der(\mathbb L_{n,c})$ is in fact the truncation $\Der_{< c}(\mathbb L_n)$ of $\Der(\mathbb L_n)$. The following result is due to Bartholdi \cite[Th.~5.1]{Bartholdi} (see also \cite[Cor.~2.46]{Darne1}):
\begin{prop}\label{Johnson_iso}
The Johnson morphism associated to $F_{n,c}$ is an isomorphism: 
\[\gr(\mathcal A_*(F_{n,c})) \cong \Der(\mathbb L_{n,c}) \cong \Der_{< c}(\mathbb L_n).\]
\end{prop}

Let $\Der_t(\mathbb L_n)$ be the Lie subalgebra of \emph{tangential derivations} of $\Der(\mathbb L_n)$, which are the ones sending each element $X_i$ of the fixed basis to $[X_i, W_i]$, for some $W_i \in \mathbb L_n$. It is easy to see that the Johnson morphism sends classes of basis-conjugating automorphisms to tangential derivations. Moreover, we can adapt the previous result to this context:
\begin{prop}\label{Johnson_iso_bc}
The previous isomorphism identifies the subalgebra $\gr(\mathcal A_*^C(F_{n,c}))$ of $\gr(\mathcal A_*(F_{n,c}))$ with the subalgebra $\Der_t(\mathbb L_{n,c}) \cong \Der_t(\mathbb L_n)_{<c}$ of $\Der_{< c}(\mathbb L_n)$.
\end{prop}

\begin{proof}
The Johnson morphism is injective by construction, and sends classes of basis-conjugating automorphisms to tangential derivations. We need to show that every homogeneous tangential derivation is the image of such an automorphism of $F_{n,c}$. Let $d$ be such a derivation, of degree $k < c$, sending $X_i$ to $[X_i, W_i]$, for some $W_i \in (\mathbb L_n)_k$. This $W_i$ is the class of some $w_i \in \Gamma_k(F_{n,c})$ in $(\mathbb L_n)_k \cong \Gamma_k(F_n)/\Gamma_{k+1}(F_n)$. Let us define an endomorphism $f$ of $F_{n,c}$ by $x_i \mapsto w_i x_i w_i^{-1}$. Since $f$ induces the identity on $F_{n,c}^{ab}$, it is an automorphism \cite[Lem.~2.42]{Darne1}. Moreover, it is clear that $f \in \mathcal A_k(F_{n,c})$, and that $\tau\left(\bar f\right) = d$.
\end{proof}

The \emph{Morita-Satoh} trace can be described using free differential calculus \cite{Darne1}, but we will use the following explicit description:  
\[\begin{tikzcd}
\tr : \Der_k(\mathbb L_n) \cong V^* \otimes (\mathbb L_n)_{k+1} \ar[r, hook, "\iota"]
&V^* \otimes V^{\otimes k+1} \ar[r, "\Phi"]
&V^{\otimes k} \ar[r, two heads]
&C_k(V),
\end{tikzcd}\]
where $V := \Z^n$, $V^*$ is its dual, $\iota$ is induced by the usual inclusion of the free Lie algebra into the tensor algebra, $C_k(V) = V^{\otimes k}/(\Z/k)$ is the cyclic power, and $\Phi$ is the contraction with respect to the first variable:
\[\Phi(\omega \otimes v_1 \cdots v_{k+1}) = \omega(v_1)v_2 \cdots v_{k+1}.\]

\begin{prop}\textup{\cite[Prop.~2.36]{Darne1}.}\label{vanishing_tr}
The trace vanishes on the image of the Johnson morphism $\tau : \gr(\mathcal A_*(F_n)) \hookrightarrow \Der(\mathbb L_n)$.
\end{prop}

Using Propositions~\ref{Johnson_iso} and~\ref{Johnson_iso_bc}, we can see that the commutative square of the previous section identifies with the left square in:
\[\begin{tikzcd}
\gr_{<c}(\mathcal A_*(F_n)) \ar[r, hook, "\tau"]  
&\Der(\mathbb L_{n,c}) \ar[r,"\Tr"]  
&C\\
\gr_{<c}(\mathcal A_*^C(F_n))
\ar[r, hook,  "\tau"] \ar[u, hook]
&\Der_t(\mathbb L_{n,c}).
\ar[u, hook] \ar[ur, swap,"\Tr"]  
\end{tikzcd}\]

The following easy lemma gives a lot of tangential derivations whose trace is not trivial (see also \cite[Prop.3.7]{Satoh}):
\begin{lem}\label{non-trivial_tr}
Let $n \geq 2$ and $l \geq 1$, and fix $i \in \{1, ..., n\}$. Let $d$ be the derivation sending all the $X_j$ to $0$, save $X_i$, which is sent to $[X_i, [X_{i_1}, [X_{i_2}, \cdots [X_{i_l}, X_i] \cdots ]]]$, where the $i_k \in \{1, ..., n\}$ are different from $i$. Then: 
\[\Tr(d) = X_{i_l} X_{i_{l-1}} \cdots X_{i_1} X_i \in C_{l+1}. \]
\end{lem}

\begin{prop}\label{existence_of_non-tame}
Let $c \geq 3$ and $n \geq 2$. The canonical morphism from $\Aut(Q_n) \cong \Aut_C(F_n)$ ($\cong wB_n$) to $\Aut(Q_{n,c}) \cong \Aut_C(F_{n,c})$ is not surjective.
\end{prop}

\begin{proof}
The following formulas define an endomorphism of $F_{n,c}$:
\[\begin{cases}
x_1 \mapsto [x_1, x_2]x_1[x_1,x_2]^{-1} \\
x_i \mapsto x_i\ \text{ for } i \geq 2. 
\end{cases}\]
This endomorphism $\sigma$ acts trivially on the abelianisation, so it is an automorphism \cite[Lem.~2.42]{Darne1}: $\sigma \in \Aut_C(F_{n,c})$. In fact, it acts trivially modulo $\Gamma_3$, so it is in $\mathcal A_2(F_{n,c})$. The associated (tangential) derivation $\tau(\overline \sigma)$ (which is well defined since $2 < c$, whence $\gr_2(\mathcal A_*(F_{n,c})) \cong \Der_2(\mathbb L_n)$) sends $X_1$ to $[[X_1,X_2],X_1]$ and $X_i$ to $0$ if $i \geq 2$. Its trace is thus:
\[\Tr(\tau(\overline \sigma)) = -[X_1,X_2] + X_2X_1 = X_2X_1 \in C_2.\]
It is not trivial, so, by Proposition~\ref{vanishing_tr}, $\tau(\overline \sigma)$ cannot be $\tau(\overline \varphi)$ for $\varphi \in \Aut(F_n)$, which implies that $\sigma$ cannot come from an automorphism of $F_n$.
\end{proof}

\begin{rmq}
Lemma~\ref{non-trivial_tr} can be used to be more precise: it allows one to construct many elements of $\mathcal A_l(F_{n,c})$ not in the image of $\Aut_C(F_n)$ for every $l \in \{2, ..., c-1\}$. Namely, apply the reasoning of the proof to the endomorphism $\sigma$ of $F_{n,c}$ fixing $x_2, ..., x_n$ and sending $x_1$ to $wx_1w^{-1}$, where $w = [x_{i_1}, [x_{i_2}, \cdots [x_{i_l}, x_1] \cdots ]]$, with the $i_k$ different from $1$.
\end{rmq}

\begin{rmq}
For $c=1$ and $c = 2$, there are no non-tame automorphisms of $Q_{n,c}$. For $c = 1$, this is easy to see: $Q_{n,1}$ is trivial with $n$ elements, so its automorphism group is just the symmetric group $\mathfrak S_n$, and every permutation lifts trivially to a basis-conjugating automorphism of the free group $F_n$ (the projection $\Aut_C(F_n) \rightarrow \Aut_C(\Z^n) = \mathfrak S_n$ is in fact split). For $c = 2$, we need to use the tools introduced above. Namely, $\Aut_C(F_{n,2})$ is an extension of $\mathfrak S_n$ by $\mathcal A_1^C(F_{n,2})$. Since $\mathcal A_2^C(F_{n,2}) = 1$, Proposition~\ref{Johnson_iso_bc} says that $\mathcal A_1^C(F_{n,2}) \cong \Der_t(\mathbb L_{n,2}) \subset \Hom(\Z^n, \Lambda_2(\Z^n))$, via the Johnson morphism $\tau_1$ (where $(\mathbb L_n)_2$ is identified with $\Lambda_2(\Z^n))$ via $[u,v] \mapsto u \wedge v$). Then $\varphi$ induces a morphism of extensions:
\[\begin{tikzcd}
P\Sigma_n \ar[r, hook] \ar[d, "\tau_1"]
&\Aut_C(F_n) \ar[r, two heads] \ar[d, "\varphi"]
&\mathfrak S_n \ar[d, "="] \\
\Der_t(\mathbb L_{n,2}) \ar[r, hook]
&\Aut_C(F_{n,2}) \ar[r, two heads]
&\mathfrak S_n,
\end{tikzcd}\]
where the kernel $P\Sigma_n$ (the McCool group) is the group of automorphisms of the form $x_i \mapsto w_i x_i w_i^{-1}$. A basis of $\Der_t(\mathbb L_{n,2})$ is given by the $d_{ij} : e_k \mapsto \delta_{ik} \cdot e_i \wedge e_j$ (for $i \neq j$). Since $e_i \wedge e_j$ is the class of $[x_i, x_j]$, we have $d_{ij} = \tau(K_{ij})$, where $K_{ij}$ fixes the $x_k$ if $k \neq i$ and $K_{ij}(x_i) = x_j x_i x_j^{-1}$. Thus $\tau_1$ is surjective, and so is $\varphi$.
\end{rmq}

\bibliographystyle{alpha}
\bibliography{Ref_nilp_qdles}

\end{document}